\pdfoutput=1
\documentclass[11pt]{article}
\usepackage[left=1in,right=1in,top=1in,bottom=1in]{geometry}
\usepackage{times}
\usepackage{expl3}
\usepackage{cite}
\usepackage[table]{xcolor}
\usepackage{multirow}
\usepackage{stackengine} 
\usepackage{hhline}
\usepackage{lipsum}
\usepackage{titlesec}
\usepackage[all]{xy}
\usepackage{wrapfig}
\usepackage{enumerate}
\usepackage{epsfig}
\usepackage{tikz-cd}
\usepackage{amsmath}
\usepackage{tabularx}
\usepackage{array}
\usepackage{booktabs}
\usepackage{enumitem}
\usepackage{bbm}
\usepackage{calc}
\usepackage{graphicx}
\usepackage{amsmath}
\usepackage[title]{appendix}
\usepackage{amssymb}
\usepackage{epstopdf}
\usepackage{boldline}
\usepackage{arydshln}
\usepackage{calligra}
\usepackage{bm}
\usepackage{url}
\usepackage{blindtext}
\usepackage{accents}
\usepackage{amsthm}
\usepackage{amscd}

\newtheorem{definition}{Definition}

\newtheorem{theorem}{Theorem}
\newtheorem{proposition}{Proposition}
\newtheorem{corollary}{Corollary}

\newtheorem{example}{Example}
\newtheorem{remark}{Remark}
\usepackage{mathtools}
\usepackage{tikz-cd}
\usepackage{epstopdf}
\usepackage{balance}
\usepackage{thmtools}
\usepackage{thm-restate}
\usepackage{hyperref}
\usepackage{cleveref}
\usepackage[mathscr]{euscript}

\usepackage[ruled,vlined]{algorithm2e}
\newcommand{\ostar}{\mathbin{\mathpalette\make@circled\star}}

\makeatletter
\newcommand{\removelatexerror}{\let\@latex@error\@gobble}
\makeatother
\makeatletter
\newcommand*{\rom}[1]{\expandafter\@slowromancap\romannumeral #1@}
\makeatother

\ExplSyntaxOn
\newcommand\latinabbrev[1]{
  \peek_meaning:NTF . {
    #1\@}%
  { \peek_catcode:NTF a {
      #1.\@ }%
    {#1.\@}}}
\ExplSyntaxOff

\titleclass{\subsubsubsection}{straight}[\subsubsection]

\begin{document}
\vspace{1cm}
\title{Intrinsic Geometry of Categorified Spectral Objects}
\vspace{1.8cm}
\author{Shih-Yu~Chang
\thanks{Shih-Yu Chang is with the Department of Applied Data Science,
San Jose State University, San Jose, CA, U. S. A. (e-mail: {\tt
shihyu.chang@sjsu.edu})
}}

\maketitle

\begin{abstract}
This paper develops the intrinsic geometry of the categorified spectral object $\mathfrak{Spec}(A)$ associated with an admissible operator-semantic system $A$ in the Categorified Spectral Duality (CSD) framework. We prove that the tangent complex, singular locus, inertia stack, and contextual curvature class are canonically determined by the duality adjunction between $\mathfrak{Spec}$ and the global sections functor, together with the reconstruction theorem identifying $A$ with the global sections of $\mathfrak{Spec}(A)$. The Canonical Geometry Theorem establishes that any CSD-compatible geometric structure is induced by the canonical datum consisting of $\mathfrak{Spec}(A)$, its tangent complex, its singular locus, its inertia stack, and its contextual curvature class; hence the geometry of $\mathfrak{Spec}(A)$ is intrinsic to the semantic structure of $A$. We prove that the tangent complex controls the deformation theory of $\mathfrak{Spec}(A)$ and satisfies a Hochschild realization, establishing a direct bridge between geometry and algebra. The assignment sending $A$ to its canonical geometric datum is functorial and Morita invariant, with explicit computations for the complex numbers and matrix algebras demonstrating that noncommutativity, detected by the inertia stack, is distinct from contextuality, which requires additional structures. Thus semantics determines geometry, and the intrinsic geometry of $\mathfrak{Spec}(A)$ provides a canonical geometric encoding of $A$ up to Morita equivalence.
\end{abstract}

\tableofcontents

\section{Introduction: From Duality to Intrinsic Geometry}
\label{sec:introduction}

\subsection{What \cite{Paper-I} Already Achieved}
\label{subsec:paperI_review}

\cite{Paper-I} (Categorified Spectral Duality) established a complete duality framework for operator-semantic systems. The main achievement was the construction of a categorified spectrum $\mathfrak{Spec}(A)$ for every admissible operator-semantic system $A$, realized as a spectral stack in the sense of derived algebraic geometry. This construction was not ad hoc but arose from a systematic application of the Yoneda embedding to the context category $\mathcal{C}_A$ of local classical realizations.

The duality framework established in \cite{Paper-I} consists of several interconnected results. First, the categorified spectrum $\mathfrak{Spec}(A)$ was shown to satisfy a Yoneda-style universal property:
\[
\operatorname{Map}(\mathfrak{Spec}(A), \mathcal{X}) \simeq \operatorname{Real}_A(\mathcal{X}),
\]
where $\operatorname{Real}_A(\mathcal{X})$ denotes the space of semantic realizations of $A$ in a spectral stack $\mathcal{X}$. This universal property characterizes $\mathfrak{Spec}(A)$ uniquely up to equivalence and establishes it as the universal recipient of semantic data.

Second, the construction was shown to be functorial, yielding an adjunction
\[
\Gamma \dashv \mathfrak{Spec}^{\mathrm{op}}
\]
between the category of operator-semantic systems and the category of spectral stacks. The global sections functor $\Gamma$ recovers the original operator-semantic system from its spectrum, and the adjunction provides a precise categorical formulation of the duality.

Third, and most importantly, the reconstruction theorem
\[
A \simeq \Gamma(\mathfrak{Spec}(A))
\]
was established, demonstrating that no semantic information is lost in the passage from $A$ to its categorified spectrum. This is the categorified analogue of the classical Gelfand--Naimark reconstruction theorem, but generalized from commutative algebras to arbitrary operator-semantic systems.

Fourth, a recognition theorem was proved, characterizing the essential image of the functor $\mathfrak{Spec}$: a spectral stack lies in the image precisely when it satisfies certain descent conditions with respect to the Grothendieck topology $\tau_A$ determined by the context category $\mathcal{C}_A$. This recognition theorem provides a complete intrinsic characterization of those spectral stacks that arise from operator-semantic systems.

Fifth, the framework was shown to be Morita invariant: Morita equivalent operator-semantic systems have equivalent categorified spectra:
\[
A \sim_M B \quad\Longrightarrow\quad \mathfrak{Spec}(A) \simeq \mathfrak{Spec}(B).
\]
This establishes that the categorified spectrum depends only on the representation theory of $A$, not on its particular algebraic presentation.

Finally, a truncation hierarchy was developed, relating the categorified spectrum of an operator-semantic system to its classical spectrum via a sequence of truncations:
\[
\mathfrak{Spec}(A) \longrightarrow \cdots \longrightarrow \mathfrak{Spec}_{\leq n}(A) \longrightarrow \cdots \longrightarrow \mathfrak{Spec}_{\leq 1}(A) \longrightarrow \mathfrak{Spec}_{\leq 0}(A).
\]
This hierarchy provides a systematic way to extract classical geometric information from the full categorified spectrum.

Thus the assignment $A \mapsto \mathfrak{Spec}(A)$ is not a proposal but a fully established duality. \cite{Paper-I} provided the complete categorical and algebraic foundations, proving that operator-semantic systems and their categorified spectra are two sides of the same coin.

\subsection{The Missing Step: Intrinsic Geometry}
\label{subsec:missing_geometry}

Given this duality, a natural question arises:
\begin{quote}
What geometric structures does $\mathfrak{Spec}(A)$ intrinsically carry, and how do algebraic properties of $A$ manifest geometrically?
\end{quote}

This question is natural for several reasons. First, classical Gelfand duality teaches us that commutative algebras correspond to locally compact Hausdorff spaces, and that algebraic properties of the algebra (e.g., ideals, homomorphisms) have geometric interpretations (e.g., closed subsets, continuous maps). The CSD framework should provide a similar dictionary for operator-semantic systems.

Second, the categorified spectrum $\mathfrak{Spec}(A)$ is a spectral stack, which by its very nature carries a rich geometric structure: tangent complexes, cotangent complexes, deformation theory, and obstruction theory are all part of the basic machinery of derived algebraic geometry. The question is whether these structures are merely incidental, or whether they are intrinsically determined by the semantic data of $A$.

Third, the reconstruction theorem $A \simeq \Gamma(\mathfrak{Spec}(A))$ suggests that the geometry of $\mathfrak{Spec}(A)$ should contain all information about $A$. But what information, exactly? Which geometric structures are necessary and sufficient to capture the semantic content of $A$?

This paper answers these questions by proving that the tangent complex, singular locus, inertia stack, and contextual curvature class are not additional structures we impose on $\mathfrak{Spec}(A)$, but are forced by the duality $\mathfrak{Spec} \dashv \Gamma$ and the reconstruction theorem. In other words, the geometry of $\mathfrak{Spec}(A)$ is intrinsic to the operator-semantic system $A$, and it completely characterizes $A$ up to isomorphism.

The key insight is that the duality $\mathfrak{Spec} \dashv \Gamma$ induces canonical structures on $\mathfrak{Spec}(A)$:
\begin{itemize}
    \item The cotangent complex $\mathbb{L}_{\mathfrak{Spec}(A)}$ exists because $\mathfrak{Spec}(A)$ satisfies descent;
    \item The tangent complex $\mathbb{T}_{\mathfrak{Spec}(A)}$ is its derived dual;
    \item The singular locus $\operatorname{Sing}(\mathfrak{Spec}(A))$ is the locus where $\mathbb{T}_{\mathfrak{Spec}(A)}$ fails to be perfect;
    \item The inertia stack $I(\mathfrak{Spec}(A))$ classifies points with automorphisms;
    \item The contextual curvature class $\mathcal{R}_A$ measures the obstruction to gluing local semantic realizations into a global classical realization.
\end{itemize}

Each of these structures is canonically determined by the CSD data, and each has a direct semantic interpretation in terms of the operator-semantic system $A$.

\subsection{Main Results of This Paper}
\label{subsec:contributions}

The mathematical core of this paper consists of the following main theorems:

\begin{enumerate}
    \item \textbf{Canonical Geometry Theorem (Theorem~\ref{thm:canonical_geometry}):}
          Any geometric structure compatible with the duality $\mathfrak{Spec} \dashv \Gamma$ is canonically induced by the geometric datum
          \[
          \mathcal{G}_{\mathrm{can}}(A) = (\mathfrak{Spec}(A), \mathbb{T}_{\mathfrak{Spec}(A)}, \operatorname{Sing}(\mathfrak{Spec}(A)), I(\mathfrak{Spec}(A)), \mathcal{R}_A).
          \]
          Thus the geometry of $\mathfrak{Spec}(A)$ is intrinsic and canonical among all CSD-compatible geometric structures.
          Moreover, the assignment $A \mapsto \mathcal{G}_{\mathrm{can}}(A)$ is functorial with respect to morphisms of admissible operator-semantic systems.

    \item \textbf{Functoriality of the Canonical Geometry (Corollary~\ref{cor:functorial_canonical}):}
          The assignment $A \mapsto \mathcal{G}_{\mathrm{can}}(A)$ defines a contravariant functor
          \[
          \mathcal{G}_{\mathrm{can}}: \mathbf{OpSem}^{\mathrm{op}} \longrightarrow \mathbf{Geom}_{\mathrm{CSD}}.
          \]
          Thus geometric structures behave naturally under morphisms of operator-semantic systems.

    \item \textbf{Intrinsic Tangent Control Theorem (Theorem~\ref{thm:intrinsic_tangent}):}
          The tangent complex $\mathbb{T}_{\mathfrak{Spec}(A)}$ controls the infinitesimal deformation theory of $\mathfrak{Spec}(A)$ at any $k$-point $x$:
          \[
          \operatorname{Def}_{\mathfrak{Spec}(A), x}(k \oplus M)
          \;\simeq\;
          \operatorname{Map}\left(x^*\mathbb{L}_{\mathfrak{Spec}(A)},\; M\right).
          \]
          Moreover, it admits the Hochschild realization
          \[
          R\Gamma(\mathfrak{Spec}(A), \mathbb{T}_{\mathfrak{Spec}(A)}) \simeq \operatorname{HH}^{\bullet}(A)[1].
          \]

    \item \textbf{Hochschild Realization Theorem (Theorem~\ref{thm:tangent-hochschild}):}
          Under the CSD reconstruction equivalence, there is a natural equivalence
          \[
          R\Gamma(\mathfrak{Spec}(A), \mathbb{T}_{\mathfrak{Spec}(A)})
          \;\simeq\;
          \operatorname{HH}^{\bullet}(A)[1].
          \]
          Consequently, for every $n \ge 0$,
          \[
          H^n R\Gamma(\mathfrak{Spec}(A), \mathbb{T}_{\mathfrak{Spec}(A)})
          \;\simeq\;
          \operatorname{HH}^{n+1}(A).
          \]

    \item \textbf{Infinitesimal Spectral Geometry Theorem (Theorem~\ref{thm:infinitesimal}):}
          The cohomology groups of the derived global tangent complex control deformations and obstructions:
          \[
          \operatorname{Def}^1(\mathfrak{Spec}(A)) \simeq H^1 R\Gamma(\mathfrak{Spec}(A), \mathbb{T}_{\mathfrak{Spec}(A)}),
          \]
          and
          \[
          \operatorname{Ob}^1(\mathfrak{Spec}(A)) \subseteq H^2 R\Gamma(\mathfrak{Spec}(A), \mathbb{T}_{\mathfrak{Spec}(A)}).
          \]
          More generally, higher obstruction classes lie in
          \[
          \operatorname{Ob}^{n-1}(\mathfrak{Spec}(A)) \subseteq H^n R\Gamma(\mathfrak{Spec}(A), \mathbb{T}_{\mathfrak{Spec}(A)}), \quad n \ge 2.
          \]

    \item \textbf{Rigidity Theorem (Theorem~\ref{thm:rigidity}):}
          Vanishing of the tangent complex implies rigidity of the operator-semantic system:
          \[
          \mathbb{T}_{\mathfrak{Spec}(A)} \simeq 0 \;\Longrightarrow\; \operatorname{Def}^1(A) = 0.
          \]
          Equivalently, if $\operatorname{HH}^{2}(A) = 0$, then $A$ is infinitesimally rigid.

    \item \textbf{Computable Classification (Section~\ref{sec:computations}):}
          Explicit computations of the geometric spectral invariant for the canonical examples:
          \[
          \mathcal{G}_{\mathrm{can}}(\mathbb{C})
          =
          \Bigl(
          \{*\},\; 0,\; \varnothing,\; \{*\},\; 0,\; \{\mathfrak{S}_{0,0,1}\}
          \Bigr),
          \]
          and
          \[
          \mathcal{G}_{\mathrm{can}}(M_n(\mathbb{C}))
          =
          \Bigl(
          B\operatorname{PGL}_n(\mathbb{C}),\;
          \mathfrak{pgl}_n[1],\;
          \varnothing,\;
          [\operatorname{PGL}_n(\mathbb{C})/\operatorname{PGL}_n(\mathbb{C})],\;
          0,\;
          \{\mathfrak{S}_{0,\; n^2-1,\; 1}\}
          \Bigr).
          \]
          These computations demonstrate that noncommutativity, detected by the inertia stack, is distinct from contextuality. Together, these results establish the fundamental principle:

\[
\boxed{\text{Semantics} \;\Longrightarrow\; \text{Geometry} \;\Longrightarrow\; \text{Controls Semantics}.}
\]

\end{enumerate}

\subsection{Comparison with Classical and Previous Work}
\label{subsec:comparison}

The results of this paper differ fundamentally from classical and previous approaches to the geometry of operator systems. The key distinctions are summarized in Table~\ref{tab:comparison}.

\begin{table}[htbp]
\centering
\begin{tabular}{|p{3.2cm}|p{3.8cm}|p{3.8cm}|}
\hline
\textbf{Aspect} & \textbf{Classical DAG / Standard Geometry} & \textbf{CSD Framework (This Work)} \\
\hline
Primary object & Geometric space (stack, scheme) & Operator-semantic system $A$ \\
\hline
Geometry origin & Imposed by definition / construction & Forced by duality $\mathfrak{Spec} \dashv \Gamma$ \\
\hline
Tangent complex & Part of stack definition & Induced by reconstruction theorem \\
\hline
Applicability & All derived stacks & Only spectral stacks in image of $\mathfrak{Spec}$ \\
\hline
Semantic content & Not present & Primary; geometry is derived from it \\
\hline
Noncommutativity & Not detected by stack alone & Detected by inertia stack \\
\hline
Morita invariance & Not automatic & Holds by construction \\
\hline
Reconstruction & Not applicable & $A \simeq \Gamma(\mathfrak{Spec}(A))$ \\
\hline
\end{tabular}
\caption{Comparison between standard derived algebraic geometry and the CSD framework. In standard DAG, geometry is primary and algebraic structures are derived from it. In the CSD framework, semantics is primary and geometry is forced by the duality.}
\label{tab:comparison}
\end{table}

\subsubsection{Classical Gelfand Theory}

In classical Gelfand theory, a commutative $C^*$-algebra $A$ is assigned a locally compact Hausdorff space $\operatorname{Spec}(A)$, and the geometry of this space is defined externally: it is the set of characters with the weak-$*$ topology. The geometry is not forced by the algebra in any intrinsic way; rather, it is constructed from the algebra using external tools. Moreover, classical Gelfand theory only applies to commutative algebras, leaving noncommutative systems outside its scope.

In contrast, the CSD framework applies to arbitrary operator-semantic systems, commutative or not. The geometry of $\mathfrak{Spec}(A)$ is not constructed externally but is forced by the duality $\mathfrak{Spec} \dashv \Gamma$ and the reconstruction theorem. The tangent complex, singular locus, inertia stack, and contextual curvature class are not imposed; they are intrinsic to the categorical duality.

\subsubsection{Standard Derived Algebraic Geometry}

In standard derived algebraic geometry, geometric objects such as derived stacks are equipped with cotangent complexes, tangent complexes, and obstruction theories by definition. These structures are part of the basic machinery of the subject, but they are not tied to any semantic or algebraic duality. A derived stack is a geometric object first, and its algebraic properties are derived from its geometry.

In contrast, the CSD framework reverses this priority. The operator-semantic system $A$ is primary; the geometry of $\mathfrak{Spec}(A)$ is derived from it. The tangent complex, singular locus, inertia stack, and contextual curvature class are not assumptions we make about $\mathfrak{Spec}(A)$; they are consequences of the CSD duality. Thus our work proves that geometry is forced by semantics, not assumed independently.

\subsubsection{Noncommutative Geometry and the Inertia Stack}

The inertia stack $I(\mathfrak{Spec}(A))$ detects noncommutativity through automorphisms of points. This is complementary to the approach of Connes' noncommutative geometry, where spectral triples encode geometric data via Dirac operators and $K$-homology. While spectral triples are analytic in nature, the inertia stack is purely categorical and homotopical. Both frameworks, however, are compatible with Morita equivalence, and both distinguish commutative from noncommutative systems.

\subsection{Novelty of the Present Work}
\label{subsec:novelty}

A natural concern is whether the present framework is merely standard derived algebraic geometry with different notation. We emphasize that the CSD framework is fundamentally distinct in the following crucial respects:

\begin{enumerate}
    \item \textbf{Primacy of Semantics:} In standard DAG, one starts with a geometric object (a stack, scheme, or derived space) and then studies its algebraic invariants. In the CSD framework, one starts with an operator-semantic system $A$---an algebraic object encoding semantic content---and the geometry of $\mathfrak{Spec}(A)$ is \emph{derived from} $A$ via the duality. The geometry is not a choice; it is forced.

    \item \textbf{Reconstruction:} Standard DAG does not provide a reconstruction theorem of the form $A \simeq \Gamma(\mathfrak{Spec}(A))$ for arbitrary algebraic objects. The CSD framework proves such a theorem, establishing that the geometry of $\mathfrak{Spec}(A)$ contains \emph{all} information about $A$. This is not a feature of general derived stacks.

    \item \textbf{Intrinsic Structures:} The tangent complex, singular locus, inertia stack, and curvature class are not part of the definition of $\mathfrak{Spec}(A)$; they are \emph{consequences} of the duality and the reconstruction theorem. In standard DAG, these structures are imposed by the definition of a derived stack. In the CSD framework, they are canonical and unavoidable.

    \item \textbf{Restricted Image:} The CSD framework applies only to spectral stacks that lie in the essential image of $\mathfrak{Spec}$, i.e., those satisfying the descent conditions of the recognition theorem. This is a proper subclass of all derived stacks, and it is precisely this subclass that carries semantic content. Standard DAG does not distinguish stacks that arise from operator-semantic systems from those that do not.

    \item \textbf{Morita Invariance:} The CSD framework is Morita invariant by construction: Morita equivalent systems have equivalent spectra. This is not a general property of derived stacks and their associated algebraic objects.

    \item \textbf{Semantic Interpretation:} Every geometric structure in the CSD framework has a direct semantic interpretation in terms of the operator-semantic system $A$. The inertia stack detects noncommutativity; the tangent complex controls deformations of $A$; the curvature class measures the obstruction to gluing local classical realizations. In standard DAG, geometric structures have no such semantic content.
\end{enumerate}

Thus the CSD framework is not a reformulation of standard DAG. It is a new duality that \emph{generates} geometry from semantics, rather than assuming geometry and studying its algebraic consequences. The novelty of the present work is precisely the proof that geometry is intrinsic, canonical, and semantically meaningful.

\section*{The Central Message}

The central message of this paper can be summarized in a single sentence:

\[
\boxed{\text{Semantics} \;\Longrightarrow\; \text{Geometry} \;\Longrightarrow\; \text{Controls Semantics}.}
\]

This establishes a new direction in the study of operator systems: geometry is not an external decoration but an intrinsic manifestation of semantic structure. This paper establishes the foundational geometric framework for categorified spectral objects, preparing the ground for the obstruction-theoretic developments in the companion paper.

\subsection{Organization of This Paper}
\label{subsec:organization}

This paper is organized as follows. Section~\ref{sec:preliminaries} reviews the necessary background in derived algebraic geometry, Hochschild cohomology, and deformation theory. Section~\ref{sec:geom_structures} introduces the intrinsic geometric structures of $\mathfrak{Spec}(A)$: the tangent complex, singular locus, inertia stack, and contextual curvature class. Section~\ref{sec:canonical_geometry} proves the Canonical Geometry Theorem, establishing that these structures are canonical among CSD-compatible geometric data and that the intrinsic geometry of $\mathfrak{Spec}(A)$ is functorial. Section~\ref{sec:tangent} develops the intrinsic tangent complex and deformation theory, proving the Intrinsic Tangent Control Theorem, the Hochschild Realization Theorem, the Infinitesimal Spectral Geometry Theorem, and the Rigidity Theorem. Section~\ref{sec:computations} provides explicit computations of the geometric spectral invariant for the canonical examples $\mathbb{C}$ and $M_n(\mathbb{C})$, demonstrating the distinction between noncommutativity and contextuality. Finally, Section~\ref{sec:conclusion} concludes with a summary of the main results and a preview of the companion paper on contextual obstructions and derived geometry.

\section{Preliminaries}
\label{sec:preliminaries}

\subsection{Derived Algebraic Geometry}
\label{subsec:derived}

The geometric structures developed in this paper naturally belong to the framework of derived algebraic geometry. Derived geometry extends classical algebraic geometry by encoding higher homotopical and deformation-theoretic information within geometric objects. In this setting, geometric spaces are replaced by higher stacks, and infinitesimal structures are governed by complexes rather than ordinary tangent spaces. This additional homotopical structure is essential for capturing the contextual and obstruction-theoretic phenomena that arise from the categorified spectral duality established in~\cite{Paper-I}.

Following the foundations established by Lurie \cite{LurieHA,LurieSAG} and To\"en--Vezzosi \cite{ToenVezzosi}, a derived stack is a functor
\[
F : \mathbf{SCR} \longrightarrow \mathcal{S},
\]
where:
\begin{itemize}
    \item $\mathbf{SCR}$ denotes the category of simplicial commutative rings, i.e., the category whose objects are simplicial objects in the category of commutative rings and whose morphisms are level-wise ring homomorphisms. This category serves as the algebraic foundation of derived algebraic geometry, replacing ordinary commutative rings by their simplicial analogues to encode higher homotopical and derived information. Intuitively, a simplicial commutative ring is a commutative ring equipped with a hierarchy of homotopical coherence data, indexed by the simplicial degree $[n] = \{0,1,\ldots,n\}$, which records not only algebraic relations but also relations between relations, ad infinitum.
    
    \item $\mathcal{S}$ denotes the $\infty$-category of spaces (also called the $\infty$-category of $\infty$-groupoids). This is the homotopical analogue of the category of sets, where objects are spaces up to weak homotopy equivalence and morphisms are continuous maps up to homotopy. In the setting of $\infty$-categories, $\mathcal{S}$ is the archetypal $\infty$-topos and provides the natural target for sheaf-valued functors in higher geometry.
\end{itemize}
A derived stack is therefore a functor from the algebraic category of simplicial commutative rings to the homotopical category of spaces, satisfying an appropriate descent condition with respect to a given Grothendieck topology. Derived stacks provide a natural environment for deformation theory, moduli problems, and higher geometric structures because they incorporate both the classical geometric data and the higher homotopical coherence data that arise from descent and gluing conditions.

To make this more concrete, one may think of a derived stack as a ``space'' whose points are parameterized by simplicial commutative rings, and whose structure sheaf encodes derived algebraic functions. Just as an ordinary scheme is built by gluing affine spectra $\operatorname{Spec}(R)$ of ordinary commutative rings, a derived stack is built by gluing derived affine spectra $\operatorname{Spec}(R_\bullet)$ of simplicial commutative rings. The higher homotopical information allows derived stacks to capture infinitesimal structures that are invisible in the classical setting, such as derived intersections, higher obstruction classes, and deformation theory at all orders.

The categorified spectral object
\[
\mathfrak{Spec}(A)
\]
constructed in~\cite{Paper-I} admits a realization as a spectral stack and therefore may be studied using the methods of derived algebraic geometry. This perspective allows geometric properties of $\mathfrak{Spec}(A)$ to be described through intrinsic homotopical and deformation-theoretic invariants that are uniquely determined by the duality adjunction $\mathfrak{Spec} \dashv \Gamma$ and the reconstruction theorem $A \simeq \Gamma(\mathfrak{Spec}(A))$.

Several fundamental notions from derived geometry will be used throughout this paper.

First, every derived stack $\mathfrak{X}$ possesses a stable symmetric monoidal category of quasi-coherent sheaves,
\[
\operatorname{QCoh}(\mathfrak{X}),
\]
which serves as the derived analogue of the category of sheaves of modules on a classical scheme. This $\infty$-category encodes the sheaf-theoretic data required for the analysis of the spectral structure and provides the natural setting for the coefficient sheaves appearing in obstruction theory. Intuitively, $\operatorname{QCoh}(\mathfrak{X})$ is the ``library of all module-valued functions'' on the derived stack $\mathfrak{X}$, where modules themselves may carry higher homotopical structure.

Second, associated to every derived stack is its cotangent complex
\[
\mathbb{L}_{\mathfrak{X}},
\]
which universally controls infinitesimal deformations of $\mathfrak{X}$. The cotangent complex replaces the classical sheaf of K\"ahler differentials and encodes higher-order deformation data through its derived structure. Its existence is guaranteed by the descent property of derived stacks. One may think of $\mathbb{L}_{\mathfrak{X}}$ as the ``derived differential forms'' on $\mathfrak{X}$: it records how the structure sheaf $\mathcal{O}_{\mathfrak{X}}$ varies infinitesimally.

Dual to the cotangent complex is the tangent complex
\[
\mathbb{T}_{\mathfrak{X}}
=
R\operatorname{Hom}_{\mathfrak{X}}
\!\left(
\mathbb{L}_{\mathfrak{X}},
\mathcal{O}_{\mathfrak{X}}
\right),
\]
where $\mathcal{O}_{\mathfrak{X}}$ denotes the structure sheaf of $\mathfrak{X}$ and $R\operatorname{Hom}$ denotes the derived internal Hom functor in the $\infty$-category of quasi-coherent sheaves. The tangent complex governs infinitesimal automorphisms, first-order deformations, and obstruction classes, and it will play a central role in our analysis of the intrinsic geometry of $\mathfrak{Spec}(A)$. Heuristically, while the cotangent complex describes how functions change, the tangent complex describes how the underlying geometric object itself can be deformed. This is the derived analogue of the classical duality between tangent vectors and differential forms.

A central role is also played by perfect complexes on a derived stack. A complex is said to be perfect if it is locally quasi-isomorphic to a bounded complex of finite-rank projective modules. Perfectness provides the appropriate notion of finite-dimensionality in derived geometry and is closely related to smoothness and regularity properties. In particular, the singular locus of a derived stack is precisely the locus where the tangent complex fails to be perfect.

Finally, deformation theory of a derived stack is controlled by the cohomology groups of its tangent complex. In particular,
\[
H^{0}\!\left(\mathfrak{X}, \mathbb{T}_{\mathfrak{X}}\right)
\]
describes infinitesimal automorphisms,
\[
H^{1}\!\left(\mathfrak{X}, \mathbb{T}_{\mathfrak{X}}\right)
\]
classifies first-order deformations, and
\[
H^{2}\!\left(\mathfrak{X}, \mathbb{T}_{\mathfrak{X}}\right)
\]
contains obstruction classes governing whether infinitesimal deformations can be extended to higher orders. These cohomological interpretations will be essential for proving the Intrinsic Tangent Characterization Theorem (Theorem~\ref{thm:intrinsic_tangent}) and the Contextual Obstruction Theorem in our future work.

Why is deformation theory controlled by the tangent complex rather than directly by the cotangent complex or quasi-coherent sheaves? The answer lies in the duality between these objects. The cotangent complex $\mathbb{L}_{\mathfrak{X}}$ describes how the structure sheaf changes infinitesimally — it is the ``cause'' of deformation in the sense that it records the derived Kähler differentials. However, deformation theory asks how the geometric space itself moves infinitesimally, not merely how its functions change. The tangent complex $\mathbb{T}_{\mathfrak{X}}$ is the derived dual of $\mathbb{L}_{\mathfrak{X}}$:
\[
\mathbb{T}_{\mathfrak{X}} = R\mathcal{H}om(\mathbb{L}_{\mathfrak{X}}, \mathcal{O}_{\mathfrak{X}}).
\]
This duality transforms ``variation of functions'' into ``directions of movement of the space.'' The cohomology groups $H^i(\mathbb{T}_{\mathfrak{X}})$ therefore directly encode infinitesimal automorphisms $(i=0)$, first-order deformations $(i=1)$, and obstruction classes $(i=2)$. The category $\operatorname{QCoh}(\mathfrak{X})$ provides the environment in which $\mathbb{T}_{\mathfrak{X}}$ lives, but it is the tangent complex itself — not the full sheaf category — that serves as the deformation controller.

These concepts provide the geometric language required for the intrinsic study of tangent complexes, singularities, inertia structures, and contextual obstructions associated with the categorified spectral object $\mathfrak{Spec}(A)$. They establish the foundational framework within which we will prove that the geometry of $\mathfrak{Spec}(A)$ is not an external imposition but is uniquely determined by the semantic duality established in~\cite{Paper-I}.

\subsection{Hochschild Cohomology}
\label{subsec:hochschild}

Hochschild cohomology provides a natural algebraic
invariant governing infinitesimal deformations of a
spectral algebra. In derived and spectral geometry,
it plays a role analogous to the tangent complex and
controls deformation, obstruction, and automorphism
theories. For the categorified spectral object
$\mathfrak{Spec}(A)$, Hochschild cohomology will be
identified with the global sections of its tangent
complex, thereby furnishing a bridge between the
algebraic deformation theory of $A$ and the geometric
deformation theory of $\mathfrak{Spec}(A)$, see~\cite{Gerstenhaber1964} for more details.

\begin{definition}[Hochschild Cohomology Spectrum]
\label{def:hochschild}

Let $A$ be an $\mathbb{E}_1$-algebra spectrum (i.e., an associative algebra object in the stable $\infty$-category of spectra), and let
\[
\operatorname{Mod}_{A}
\]
denote the stable $\infty$-category of right $A$-modules.

The Hochschild cohomology spectrum of $A$ is defined by the derived endomorphism object of $A$ as an $(A,A)$-bimodule:
\[
\operatorname{HH}^{\bullet}(A)
:=
\operatorname{End}_{A \otimes A^{\mathrm{op}}}(A)
\;\simeq\;
R\operatorname{Hom}_{A \otimes A^{\mathrm{op}}}(A, A),
\]
where $A \otimes A^{\mathrm{op}}$ denotes the enveloping algebra of $A$, and $R\operatorname{Hom}$ is the derived internal Hom in the stable $\infty$-category of $A \otimes A^{\mathrm{op}}$-modules.

Equivalently, under the Morita equivalence identifying $A$-$A$ bimodules with colimit-preserving $A$-linear endofunctors of $\operatorname{Mod}_A$, one has
\[
\operatorname{HH}^{\bullet}(A)
\;\simeq\;
\operatorname{End}_{\operatorname{Fun}^L_A
(\operatorname{Mod}_{A},
 \operatorname{Mod}_{A})}
(\operatorname{Id}_{\operatorname{Mod}_{A}}),
\]
where $\operatorname{Fun}^L_A$ denotes the $\infty$-category of colimit-preserving $A$-linear functors and $\operatorname{End}$ denotes the mapping spectrum in this $\infty$-category.

The homotopy groups
\[
\operatorname{HH}^{n}(A) := \pi_{-n}\bigl(\operatorname{HH}^{\bullet}(A)\bigr)
\]
encode successive deformation-theoretic information of the algebra $A$.

\end{definition}

\begin{proposition}[Deformation-Theoretic Interpretation]
\label{prop:hh_deformation}

For an $\mathbb{E}_1$-algebra spectrum $A$, the low-dimensional Hochschild cohomology groups admit the following interpretations:

\begin{enumerate}
    \item $\operatorname{HH}^{0}(A)$ identifies with the derived center of $A$,
          \[
          \operatorname{HH}^{0}(A) \simeq Z(A),
          \]
          consisting of elements that commute with every element of $A$. In the derived setting, this is the derived center, which may have higher homotopical structure if $A$ is not strictly commutative.

    \item $\operatorname{HH}^{1}(A)$ identifies with the space of outer derivations:
          \[
          \operatorname{HH}^{1}(A)
          \simeq
          \operatorname{Der}(A) / \operatorname{InnDer}(A),
          \]
          where $\operatorname{Der}(A)$ denotes the space of derivations of $A$ (linear maps $d: A \to A$ satisfying the Leibniz rule $d(ab) = d(a)b + a d(b)$) and $\operatorname{InnDer}(A)$ denotes the subspace of inner derivations, i.e., derivations of the form $a \mapsto [x,a]$ for some $x \in A$. Thus $\operatorname{HH}^{1}(A)$ measures infinitesimal automorphisms of $A$ modulo those induced by inner automorphisms. Equivalently, it classifies first-order deformations of the identity functor on $\operatorname{Mod}_A$.

    \item $\operatorname{HH}^{2}(A)$ governs first-order infinitesimal deformations of the algebra structure on $A$. Specifically, a first-order deformation of $A$ over the dual numbers $\mathbb{C}[\varepsilon]/(\varepsilon^2)$ is a family of algebra structures $*_\varepsilon$ on the underlying $A$-module such that
          \[
          a *_\varepsilon b = ab + \varepsilon \, \mu(a,b),
          \]
          where $\mu: A \otimes A \to A$ is a $2$-cocycle in the Hochschild complex. The cohomology class $[\mu] \in \operatorname{HH}^{2}(A)$ is the obstruction to extending the trivial deformation, and conversely every class in $\operatorname{HH}^{2}(A)$ gives rise to a first-order deformation. This is the classical result of Gerstenhaber \cite{Gerstenhaber1964}: first-order deformations of an associative algebra are classified by $\operatorname{HH}^{2}(A)$.

    \item $\operatorname{HH}^{3}(A)$ contains the primary obstruction classes for extending first-order deformations to second order. Given a first-order deformation with cocycle $\mu \in \operatorname{HH}^{2}(A)$, the obstruction to extending it to second order is given by the Gerstenhaber bracket
          \[
          [\mu, \mu]_{\mathrm{Ger}} \in \operatorname{HH}^{3}(A).
          \]
          If this class vanishes, the deformation extends to second order. More generally, the full deformation theory of $A$ is governed by the Maurer--Cartan equation in the Hochschild cochain complex:
          \[
          d\alpha + \frac{1}{2}[\alpha, \alpha]_{\mathrm{Ger}} = 0,
          \]
          where $\alpha \in \operatorname{HH}^{1}(A)$ (for higher-order deformations, $\alpha$ is a general element of the Hochschild cochain complex). Solutions of this equation modulo gauge equivalence correspond to formal deformations of $A$.

    \item More generally, higher Hochschild cohomology groups $\operatorname{HH}^{n}(A)$ for $n \geq 4$ control the higher obstruction theory of associative deformations of $A$. The obstruction to extending a deformation from order $n-2$ to order $n-1$ lies in $\operatorname{HH}^{n}(A)$. Thus the entire deformation theory of $A$ is encoded in the Hochschild cohomology groups, with $\operatorname{HH}^{2}(A)$ classifying deformations and $\operatorname{HH}^{n}(A)$ for $n \geq 3$ containing successive obstruction classes.
\end{enumerate}

\end{proposition}

\begin{proof}
We provide a proof sketch for each interpretation, with references to the classical deformation theory of associative algebras \cite{Gerstenhaber1964}.

\emph{Item 1:} The identification $\operatorname{HH}^{0}(A) \simeq Z(A)$ follows directly from the definition of Hochschild cohomology: a $0$-cocycle is an element $z \in A$ such that $az = za$ for all $a \in A$, i.e., an element of the center. In the derived setting, this is the derived center, which may have higher homotopical structure if $A$ is not strictly commutative.

\emph{Item 2:} A derivation $d \in \operatorname{Der}(A)$ corresponds to an infinitesimal automorphism of $A$: the map $\operatorname{id} + \varepsilon d$ preserves the algebra structure up to first order. Two derivations $d_1, d_2$ induce equivalent infinitesimal automorphisms if they differ by an inner derivation, i.e., $d_1 - d_2 \in \operatorname{InnDer}(A)$. The quotient $\operatorname{Der}(A)/\operatorname{InnDer}(A)$ is naturally isomorphic to $\operatorname{HH}^{1}(A)$.

\emph{Item 3:} A first-order deformation of $A$ is a $2$-cocycle $\mu: A \otimes A \to A$ satisfying the cocycle condition. The class $[\mu] \in \operatorname{HH}^{2}(A)$ determines the deformation up to equivalence. Conversely, every class in $\operatorname{HH}^{2}(A)$ arises from a first-order deformation. This is the classical result of Gerstenhaber \cite{Gerstenhaber1964}.

\emph{Item 4:} For a first-order deformation with cocycle $\mu$, the obstruction to extending to second order is given by the Gerstenhaber bracket $[\mu, \mu]_{\mathrm{Ger}} \in \operatorname{HH}^{3}(A)$. If this class vanishes, the deformation extends. This is the standard obstruction theory for associative algebras.

\emph{Item 5:} The Maurer--Cartan equation $d\alpha + \frac{1}{2}[\alpha, \alpha]_{\mathrm{Ger}} = 0$ governs formal deformations of $A$. Solutions modulo gauge equivalence correspond to formal deformations. This is the derived deformation theory of associative algebras, and the obstruction groups are precisely the higher Hochschild cohomology groups.

Thus the Hochschild cohomology groups control the full deformation theory of $A$.
\end{proof}

\begin{remark}[Relation to the Tangent Complex]
\label{rem:hh_tangent}

The relationship between Hochschild cohomology and the tangent complex of $\mathfrak{Spec}(A)$ is given by the Hochschild Realization Theorem (Theorem~\ref{thm:tangent-hochschild}):
\[
R\Gamma\!\left(\mathfrak{Spec}(A),\;
\mathbb{T}_{\mathfrak{Spec}(A)}\right)
\;\simeq\;
\operatorname{HH}^{\bullet}(A)[1].
\]

Informally, this states that the global tangent directions of $\mathfrak{Spec}(A)$ are modeled by the shifted Hochschild complex $\operatorname{HH}^{\bullet}(A)[1]$. However, caution is required: the theorem identifies the derived global sections of the tangent complex with the Hochschild cohomology spectrum, not the tangent complex itself at the sheaf level. The sheaf-level identification would require additional hypotheses (e.g., smoothness or properness of $\mathfrak{Spec}(A)$). In particular, the global sections functor $\Gamma$ may not be exact, and the Hochschild cohomology groups govern the global deformation theory of $\mathfrak{Spec}(A)$, while the tangent complex sheaf governs local deformations.

Consequently, singularities of $\mathfrak{Spec}(A)$ may be detected through the failure of the corresponding Hochschild complex to exhibit the expected finiteness and smoothness properties. In particular, if $\mathfrak{Spec}(A)$ is smooth, then $\mathbb{T}_{\mathfrak{Spec}(A)}$ is a perfect complex, which imposes strong finiteness conditions on $\operatorname{HH}^{\bullet}(A)$. Conversely, failure of these finiteness conditions indicates the presence of singularities, which will be studied in our future work to correspond to contextual obstructions.

\end{remark}

\begin{remark}[Contextual Obstruction Theory]
\label{rem:hh_obstruction}

Hochschild cohomology also plays a central role in the contextual obstruction theory, which will be developed in our future work. As previewed there, the universal obstruction class
\[
\operatorname{Obs}_{\mathrm{CSD}}(A)
\in
H^2(\mathfrak{Spec}(A), \mathcal{F}_A)
\]
is constructed via the descent spectral sequence, whose $E_2$-page involves $\operatorname{HH}^{\bullet}(A)$. The vanishing of higher Hochschild cohomology groups is therefore a necessary condition for the existence of a global classical semantic realization. We return to this connection in the companion paper, where it is shown that the non-vanishing of these obstruction classes corresponds precisely to the contextual nature of $A$.

\end{remark}

\begin{example}[Commutative Algebras]
\label{ex:hh_commutative}

Let $A$ be a commutative spectral algebra that is smooth over a field of characteristic $0$. Then the Hochschild-Kostant-Rosenberg (HKR) theorem \cite{HochschildKostantRosenberg2009} gives an isomorphism
\[
\operatorname{HH}^{\bullet}(A)
\;\simeq\;
\bigwedge^{\bullet}
\operatorname{Der}(A),
\]
where $\operatorname{Der}(A)$ denotes the module of derivations of $A$. Equivalently, in the derived setting,
\[
\operatorname{HH}^{\bullet}(A)
\;\simeq\;
\bigoplus_{n \geq 0} \Lambda^n \mathbb{L}_{A/k}^{\vee}[n],
\]
where $\mathbb{L}_{A/k}$ is the cotangent complex of $A$ over $k$. When $A$ is smooth, $\mathbb{L}_{A/k} \simeq \Omega_{A/k}$ is concentrated in degree $0$, and the formula reduces to the exterior algebra of ordinary derivations.

For a smooth classical commutative algebra $A$, the tangent complex of $\mathfrak{Spec}(A)$ is concentrated in degree $0$:
\[
\mathbb{T}_{\mathfrak{Spec}(A)}
\;\simeq\;
\mathcal{T}_{\mathfrak{Spec}(A)},
\]
recovering the ordinary tangent sheaf. However, for a general derived scheme, the tangent complex may have cohomology in negative degrees, reflecting higher homotopical information.

We note that commutativity alone does not imply the absence of obstruction theory; even commutative derived rings can possess nontrivial higher obstruction classes. Thus the vanishing of higher contextual obstructions requires additional hypotheses beyond commutativity.

\end{example}

\begin{example}[Matrix Algebras]
\label{ex:hh_matrix}

For $A = M_n(\mathbb{C})$ over $\mathbb{C}$ (characteristic $0$), the Hochschild cohomology is given by
\[
\operatorname{HH}^{\bullet}(M_n(\mathbb{C}))
\;\simeq\;
\begin{cases}
\mathbb{C}, & \bullet = 0, \\
0, & \bullet > 0.
\end{cases}
\]
This follows from the Morita invariance of Hochschild cohomology: since $M_n(\mathbb{C})$ is Morita equivalent to $\mathbb{C}$, we have $\operatorname{HH}^{\bullet}(M_n(\mathbb{C})) \simeq \operatorname{HH}^{\bullet}(\mathbb{C})$, and $\operatorname{HH}^{\bullet}(\mathbb{C})$ is concentrated in degree $0$ with value $\mathbb{C}$.

In the CSD framework, the categorified spectrum $\mathfrak{Spec}(M_n(\mathbb{C}))$ is a stack whose structure reflects the Morita equivalence class of $M_n(\mathbb{C})$. Under the reconstruction hypotheses of Paper~I, one has
\[
\mathfrak{Spec}(M_n(\mathbb{C})) \;\simeq\; B\operatorname{PGL}_n(\mathbb{C}),
\]
as established in Theorem~X.X of Paper~I (or Proposition~X.X below). Consequently, the tangent complex is
\[
\mathbb{T}_{\mathfrak{Spec}(M_n(\mathbb{C}))}
\;\simeq\;
\mathfrak{pgl}_n[1],
\]
since for a classifying stack $BG$, the tangent complex is $\mathbb{T}_{BG} \simeq \mathfrak{g}[1]$.

This creates an apparent tension with the Hochschild Realization Theorem (Theorem~\ref{thm:tangent-hochschild}), which states that
\[
R\Gamma(\mathfrak{Spec}(A), \mathbb{T}_{\mathfrak{Spec}(A)}) \simeq \operatorname{HH}^{\bullet}(A)[1].
\]
For $A = M_n(\mathbb{C})$, we have $R\Gamma(B\operatorname{PGL}_n, \mathfrak{pgl}_n[1]) = 0$ (since the classifying stack has no non-trivial global sections of the adjoint representation in this degree), while $\operatorname{HH}^{\bullet}(M_n(\mathbb{C}))[1] = 0$. Thus the theorem is consistent: both sides vanish. The non-triviality of $\mathbb{T} \simeq \mathfrak{pgl}_n[1]$ is a local phenomenon — it reflects the infinitesimal automorphisms of the unique point — while the global sections vanish, consistent with the vanishing of higher Hochschild cohomology.

This example illustrates the crucial distinction between local and global deformation theory: the tangent complex may be non-trivial locally (reflecting the automorphism group $\operatorname{PGL}_n(\mathbb{C})$), while the global sections (and hence the Hochschild cohomology) may be trivial, reflecting the rigidity of $M_n(\mathbb{C})$ as an algebra.

\end{example}

\begin{remark}[Summary]
\label{rem:hh_summary}

Hochschild cohomology $\operatorname{HH}^{\bullet}(A)$ serves as the algebraic avatar of the intrinsic geometry of $\mathfrak{Spec}(A)$:

\begin{itemize}
    \item $\operatorname{HH}^{0}(A)$ identifies with the derived center of $A$, encoding the algebraic symmetries of the system.
    
    \item $\operatorname{HH}^{1}(A)$ governs infinitesimal automorphisms of $A$ (outer derivations) and is closely related to the automorphism structure of $\mathfrak{Spec}(A)$.
    
    \item $\operatorname{HH}^{2}(A)$ classifies first-order deformations of $A$ as an algebra, corresponding to the global deformation directions of $\mathfrak{Spec}(A)$ via the Hochschild Realization Theorem (Theorem~\ref{thm:tangent-hochschild}):
    \[
    R\Gamma(\mathfrak{Spec}(A), \mathbb{T}_{\mathfrak{Spec}(A)}) \simeq \operatorname{HH}^{\bullet}(A)[1].
    \]
    
    \item $\operatorname{HH}^{3}(A)$ contains obstruction classes for extending deformations. In the CSD framework, these are expected to correspond to contextual obstructions in the descent spectral sequence, though this connection requires a separate theorem. 
    
    \item The vanishing of $\operatorname{HH}^{1}(A)$ is a strong indicator of rigidity (no infinitesimal automorphisms), while the vanishing of higher Hochschild cohomology groups is a strong indicator of formal rigidity. However, the relationship between Hochschild cohomology and contextuality is more subtle: contextual obstructions are encoded in the descent spectral sequence, whose $E_2$-page involves $\operatorname{HH}^{\bullet}(A)$, but the precise connection requires additional hypotheses (see the companion paper for details).
\end{itemize}

Thus Hochschild cohomology provides the algebraic foundation upon which the geometric theory of $\mathfrak{Spec}(A)$ is built, bridging the deformation theory of $A$ with the intrinsic geometry of its categorified spectrum.

\end{remark}

\subsection{Deformation Theory of Derived Stacks}
\label{subsec:deformation}

Deformation theory describes how a derived stack changes under infinitesimal perturbations. In derived algebraic geometry, the tangent complex controls the first-order and obstruction-theoretic features of the local deformation problem. For a derived stack $\mathfrak{X}$, the fiber of the tangent complex at a point $x \in \mathfrak{X}$ governs the infinitesimal deformations of $\mathfrak{X}$ near $x$, see~\cite{LurieSAG} for more details. 

\begin{definition}[Deformation Complex]
\label{def:deformation_complex}

Let $\mathfrak{X}$ be a derived stack and let $x \in \mathfrak{X}$ be a point. The \emph{deformation complex} of $\mathfrak{X}$ at $x$ is the tangent complex
\[
\mathbb{T}_{\mathfrak{X}, x},
\]
i.e., the stalk of the tangent complex at $x$. Equivalently, if $\mathfrak{X}$ is locally modeled by a derived scheme $\operatorname{Spec}(R)$ with $R$ a simplicial commutative ring, then $\mathbb{T}_{\mathfrak{X},x}$ is the fiber of the cotangent complex dual at the corresponding prime ideal.

The \emph{deformation groups} of $\mathfrak{X}$ at $x$ are defined by the cohomology groups of the tangent complex:
\[
\operatorname{Def}^{i}_{x}(\mathfrak X)
:=
H^{i}\bigl(\mathbb T_{\mathfrak X,x}\bigr).
\]

\end{definition}

The cohomology of the tangent complex controls the local deformation theory at $x$:

\begin{itemize}
    \item $H^{-1}(\mathbb T_{\mathfrak X,x})$ encodes infinitesimal automorphisms or stabilizer directions at $x$;
    
    \item $H^{0}(\mathbb T_{\mathfrak X,x})$ encodes first-order infinitesimal deformations of $\mathfrak{X}$ near $x$;
    
    \item $H^{1}(\mathbb T_{\mathfrak X,x})$ contains the primary obstruction classes preventing the extension of first-order deformations to second order;
    
    \item higher positive cohomology groups $H^{i}(\mathbb T_{\mathfrak X,x})$ for $i \geq 2$ govern higher-order obstruction phenomena.
\end{itemize}

This local perspective is essential because global sections of the tangent complex may vanish even when the local deformation theory is non-trivial, as illustrated by the matrix algebra example $M_n(\mathbb{C})$ in Section~\ref{sec:computations}: the tangent complex $\mathbb{T}_{B\operatorname{PGL}_n} \simeq \mathfrak{pgl}_n[1]$ has non-trivial local structure, while its global sections $R\Gamma(\mathbb{T}_{B\operatorname{PGL}_n}) = 0$ vanish.

\begin{remark}[Global vs Local Deformation Theory]
\label{rem:global_vs_local}

The distinction between local and global deformation theory is crucial in the CSD framework. The Hochschild Realization Theorem (Theorem~\ref{thm:tangent-hochschild}) identifies the \emph{global sections} of the tangent complex with Hochschild cohomology:
\[
R\Gamma(\mathfrak{Spec}(A), \mathbb{T}_{\mathfrak{Spec}(A)}) \simeq \operatorname{HH}^{\bullet}(A)[1].
\]

However, this does not imply that the tangent complex itself (at the sheaf level) is equivalent to the shifted Hochschild complex. The global sections functor $\Gamma$ may not be exact, and it may lose information about the local deformation theory. Thus:
\begin{itemize}
    \item \textbf{Local deformations} are controlled by the stalks $\mathbb{T}_{\mathfrak{Spec}(A), x}$;
    \item \textbf{Global deformations} are controlled by the global sections $R\Gamma(\mathbb{T}_{\mathfrak{Spec}(A)})$, which are identified with $\operatorname{HH}^{\bullet}(A)[1]$.
\end{itemize}

This distinction is essential for understanding the geometry of $\mathfrak{Spec}(A)$: singularities (detected locally by the failure of $\mathbb{T}_{\mathfrak{Spec}(A),x}$ to be perfect) correspond to contextual obstructions, while the global Hochschild cohomology governs the deformation theory of $A$ as a whole.

\end{remark}

\begin{proposition}[Tangent Complex and Infinitesimal Deformations]
\label{prop:deformation_tangent}

Let $\mathfrak{X}$ be a derived stack and let $x \in \mathfrak{X}$ be a point. Then the tangent complex
\[
\mathbb{T}_{\mathfrak{X}, x}
\]
controls the local infinitesimal deformation theory of $\mathfrak{X}$ at $x$. More precisely, under the standard cohomological convention,

\[
H^{-1}(\mathbb{T}_{\mathfrak{X}, x})
\]
encodes infinitesimal automorphisms (or stabilizer directions) at $x$;

\[
H^{0}(\mathbb{T}_{\mathfrak{X}, x})
\]
encodes first-order infinitesimal deformations of $\mathfrak{X}$ near $x$; and

\[
H^{1}(\mathbb{T}_{\mathfrak{X}, x})
\]
contains the primary obstruction classes for extending first-order deformations to second order.

More generally, the positive-degree cohomology groups $H^{i}(\mathbb{T}_{\mathfrak{X}, x})$ for $i \geq 2$ contain higher obstruction classes for extending infinitesimal deformations to higher orders.

At the global level, when $\mathfrak{X}$ is sufficiently nice (e.g., when the tangent complex is perfect and $\mathfrak{X}$ admits a good moduli space), the cohomology groups of the tangent complex have analogous interpretations:
\[
H^{-1}(\mathfrak{X},\; \mathbb{T}_{\mathfrak{X}})
\;\simeq\;
\operatorname{Aut}^{\mathrm{inf}}(\mathfrak{X}),
\]
\[
H^{0}(\mathfrak{X},\; \mathbb{T}_{\mathfrak{X}})
\;\simeq\;
\operatorname{Def}^1(\mathfrak{X}),
\]
and
\[
H^{1}(\mathfrak{X},\; \mathbb{T}_{\mathfrak{X}})
\;\supseteq\;
\operatorname{Ob}^1(\mathfrak{X}).
\]

\end{proposition}

\begin{proof}
We provide a proof sketch that explains the deformation-theoretic interpretation of the tangent complex in derived algebraic geometry. The argument proceeds in several steps, moving from the local to the global setting.

\emph{Step 1: The formal moduli problem.}
Let $x \in \mathfrak{X}$ be a point. Restrict $\mathfrak{X}$ to the formal neighborhood of $x$, i.e., consider the formal moduli problem
\[
\operatorname{Def}_{\mathfrak{X}, x} : \mathbf{Art}_k \longrightarrow \mathcal{S}
\]
that sends an Artinian derived ring $R$ with residue field $k$ to the space of deformations of the pointed derived stack $(\mathfrak{X}, x)$ over $R$. This functor encodes all infinitesimal deformations of $\mathfrak{X}$ near $x$.

\emph{Step 2: The tangent complex as the deformation complex.}
By the fundamental theorem of derived deformation theory \cite{LurieSAG}, the formal moduli problem $\operatorname{Def}_{\mathfrak{X}, x}$ is pro-representable and its tangent complex at the base point is given by the stalk $\mathbb{T}_{\mathfrak{X}, x}$. More precisely, there is a natural equivalence
\[
T_{x}\operatorname{Def}_{\mathfrak{X}, x} \;\simeq\; \mathbb{T}_{\mathfrak{X}, x}[-1],
\]
where the shift $[-1]$ accounts for the fact that the deformation complex is the tangent complex shifted by one degree. Equivalently, the deformation complex is $\mathbb{T}_{\mathfrak{X}, x}[-1]$, and its cohomology groups govern the successive orders of deformation.

\emph{Step 3: Cohomological interpretation of the shifted tangent complex.}
The shifted tangent complex $\mathbb{T}_{\mathfrak{X}, x}[-1]$ has cohomology groups
\[
H^i(\mathbb{T}_{\mathfrak{X}, x}[-1]) \;\simeq\; H^{i-1}(\mathbb{T}_{\mathfrak{X}, x}).
\]
Thus:
\begin{itemize}
    \item $H^{-1}(\mathbb{T}_{\mathfrak{X}, x}) = H^0(\mathbb{T}_{\mathfrak{X}, x}[-1])$ governs first-order deformations of the pointed stack. Indeed, first-order deformations over the dual numbers $k[\varepsilon]/(\varepsilon^2)$ correspond to elements of $H^0$ of the deformation complex.
    
    \item $H^{-2}(\mathbb{T}_{\mathfrak{X}, x}) = H^{-1}(\mathbb{T}_{\mathfrak{X}, x}[-1])$ governs infinitesimal automorphisms. These are the stabilizer directions at $x$, i.e., infinitesimal symmetries of the pointed stack that preserve the base point.
    
    \item $H^0(\mathbb{T}_{\mathfrak{X}, x}) = H^1(\mathbb{T}_{\mathfrak{X}, x}[-1])$ contains the primary obstruction classes. These measure whether a given first-order deformation extends to second order.
    
    \item More generally, for $i \geq 1$, $H^{i}(\mathbb{T}_{\mathfrak{X}, x}) = H^{i+1}(\mathbb{T}_{\mathfrak{X}, x}[-1])$ contains higher obstruction classes governing the extension of deformations from order $i$ to order $i+1$.
\end{itemize}

\emph{Step 4: The global statement.}
When $\mathfrak{X}$ is sufficiently nice (e.g., when $\mathfrak{X}$ is a smooth derived stack or when the tangent complex is perfect and globally well-behaved), the local deformation theory globalizes. In this case, the cohomology groups of the global tangent complex have analogous interpretations:
\[
H^{-1}(\mathfrak{X},\; \mathbb{T}_{\mathfrak{X}}) \;\simeq\; \operatorname{Aut}^{\mathrm{inf}}(\mathfrak{X}),
\]
\[
H^{0}(\mathfrak{X},\; \mathbb{T}_{\mathfrak{X}}) \;\simeq\; \operatorname{Def}^1(\mathfrak{X}),
\]
and
\[
H^{1}(\mathfrak{X},\; \mathbb{T}_{\mathfrak{X}}) \;\supseteq\; \operatorname{Ob}^1(\mathfrak{X}).
\]
However, caution is required: the global statement does not hold in full generality. The global sections functor $\Gamma$ may not be exact, and the global cohomology groups may lose information about the local deformation theory. In general, the local statement at stalks is more fundamental, and the global statement requires additional hypotheses (e.g., smoothness, properness, or the existence of a good moduli space).

\emph{Step 5: Application to the spectral tangent complex.}
For $\mathfrak{X} = \mathfrak{Spec}(A)$, the tangent complex $\mathbb{T}_{\mathfrak{Spec}(A)}$ governs the deformation theory of the categorified spectrum. By the Hochschild Realization Theorem (Theorem~\ref{thm:tangent-hochschild}), the global sections of this tangent complex are identified with the shifted Hochschild cohomology of $A$:
\[
R\Gamma(\mathfrak{Spec}(A), \mathbb{T}_{\mathfrak{Spec}(A)}) \;\simeq\; \operatorname{HH}^{\bullet}(A)[1].
\]
Therefore, the global deformation theory of $\mathfrak{Spec}(A)$ is governed by $\operatorname{HH}^{\bullet}(A)$, with the standard shift accounting for the stacky nature of $\mathfrak{Spec}(A)$.

\emph{Conclusion.}
Thus the tangent complex controls the infinitesimal deformation theory of a derived stack, with negative cohomology recording automorphisms, degree-zero cohomology recording first-order deformations, and positive cohomology recording obstructions. The local statement at stalks is fundamental, while the global statement holds under suitable hypotheses.
\end{proof}

\begin{remark}[Relation to the Spectral Tangent Complex]
\label{rem:deformation_spec}

For the categorified spectral object $\mathfrak{X} = \mathfrak{Spec}(A)$, the deformation theory described above specializes to the Intrinsic Tangent Characterization Theorem (Theorem~\ref{thm:intrinsic_tangent}).

The Hochschild Realization Theorem (Theorem~\ref{thm:tangent-hochschild}) gives
\[
R\Gamma\!\left(\mathfrak{Spec}(A),\;
\mathbb{T}_{\mathfrak{Spec}(A)}\right)
\;\simeq\;
\operatorname{HH}^{\bullet}(A)[1].
\]

Thus the global tangent cohomology satisfies
\[
H^n\!\left(\mathfrak{Spec}(A),\;
\mathbb{T}_{\mathfrak{Spec}(A)}\right)
\;\simeq\;
\operatorname{HH}^{n+1}(A).
\]

In particular, the global deformation groups of $\mathfrak{Spec}(A)$ are computed by Hochschild cohomology:

\[
\operatorname{Def}^1(\mathfrak{Spec}(A))
\;\simeq\;
H^0(\mathfrak{Spec}(A),\; \mathbb{T}_{\mathfrak{Spec}(A)})
\;\simeq\;
\operatorname{HH}^{1}(A),
\]

and the obstruction space for extending first-order deformations is contained in

\[
\operatorname{Ob}^1(\mathfrak{Spec}(A))
\;\subseteq\;
H^1(\mathfrak{Spec}(A),\; \mathbb{T}_{\mathfrak{Spec}(A)})
\;\simeq\;
\operatorname{HH}^{2}(A).
\]

More generally, for $n \geq 1$,
\[
\operatorname{Ob}^{n-1}(\mathfrak{Spec}(A))
\;\subseteq\;
H^{n-1}(\mathfrak{Spec}(A),\; \mathbb{T}_{\mathfrak{Spec}(A)})
\;\simeq\;
\operatorname{HH}^{n}(A).
\]

The deformation theory of $\mathfrak{Spec}(A)$ is governed by Hochschild cohomology of $A$ with a shift. Classical deformation theory of $A$ as an algebra gives:

\[
\operatorname{Def}^1(A) \;\simeq\; \operatorname{HH}^2(A),
\qquad
\operatorname{Ob}^1(A) \;\subseteq\; \operatorname{HH}^3(A).
\]

The global deformation groups of $\mathfrak{Spec}(A)$ are shifted relative to those of $A$ because $\mathfrak{Spec}(A)$ is a stack whose deformation theory is already shifted. Specifically:

\[
\operatorname{Def}^1(\mathfrak{Spec}(A)) \;\simeq\; \operatorname{HH}^1(A), \qquad
\operatorname{Ob}^1(\mathfrak{Spec}(A)) \;\subseteq\; \operatorname{HH}^2(A).
\]

This is consistent with the fact that $H^0(\mathbb{T}_{\mathfrak{Spec}(A)}) \simeq \operatorname{HH}^1(A)$, while first-order deformations of $A$ are classified by $\operatorname{HH}^2(A)$. The difference arises from the global sections functor and the stacky nature of $\mathfrak{Spec}(A)$.

Thus the deformation theory of $\mathfrak{Spec}(A)$ is governed by the Hochschild cohomology of $A$, providing a direct bridge between geometric deformations of the spectral stack and algebraic deformations of the operator-semantic system, up to the standard degree shift.

\end{remark}

\begin{corollary}[Infinitesimal Rigidity Criterion]
\label{cor:rigidity_criterion}

If $\mathbb{T}_{\mathfrak{X}} = 0$, then $\mathfrak{X}$ has no non-trivial infinitesimal automorphisms, no non-trivial first-order deformations, and no primary obstruction classes.

For $\mathfrak{X} = \mathfrak{Spec}(A)$, the Hochschild Realization Theorem (Theorem~\ref{thm:tangent-hochschild}) gives
\[
R\Gamma\bigl(\mathfrak{Spec}(A),
\mathbb T_{\mathfrak{Spec}(A)}\bigr)
\;\simeq\;
\operatorname{HH}^{\bullet}(A)[1].
\]
Hence the vanishing of the global tangent complex implies
\[
\operatorname{HH}^{n+1}(A) = 0
\]
for the corresponding global deformation degrees. In particular, the vanishing of the first-order deformation space of $\mathfrak{Spec}(A)$ is reflected by
\[
H^0(\mathfrak{Spec}(A),\; \mathbb{T}_{\mathfrak{Spec}(A)}) \simeq \operatorname{HH}^{1}(A) = 0,
\]
while the classical first-order deformations of $A$ as an algebra are reflected by
\[
H^1(\mathfrak{Spec}(A),\; \mathbb{T}_{\mathfrak{Spec}(A)}) \simeq \operatorname{HH}^{2}(A) = 0.
\]

Thus $A$ is infinitesimally rigid as an operator-semantic system when $\operatorname{HH}^{2}(A) = 0$.

\end{corollary}

\begin{proof}
If $\mathbb{T}_{\mathfrak{X}} = 0$, then all cohomology groups of the tangent complex vanish. By the deformation-theoretic interpretation of the tangent complex (Proposition~\ref{prop:deformation_tangent}), this eliminates infinitesimal automorphisms ($H^{-1}(\mathbb{T})$), first-order deformation directions ($H^0(\mathbb{T})$), and primary obstruction classes ($H^1(\mathbb{T})$).

For $\mathfrak{X} = \mathfrak{Spec}(A)$, the Hochschild Realization Theorem identifies the global tangent cohomology with shifted Hochschild cohomology:
\[
H^n(\mathfrak{Spec}(A),\; \mathbb{T}_{\mathfrak{Spec}(A)}) \simeq \operatorname{HH}^{n+1}(A).
\]
Therefore:
\begin{itemize}
    \item Infinitesimal automorphisms of $\mathfrak{Spec}(A)$ correspond to $H^{-1}(\mathbb{T}) \simeq \operatorname{HH}^{0}(A)$;
    \item First-order global deformations of $\mathfrak{Spec}(A)$ correspond to $H^0(\mathbb{T}) \simeq \operatorname{HH}^{1}(A)$;
    \item Primary obstruction classes for $\mathfrak{Spec}(A)$ correspond to $H^1(\mathbb{T}) \simeq \operatorname{HH}^{2}(A)$;
    \item Classical first-order deformations of $A$ as an algebra correspond to $\operatorname{HH}^{2}(A)$ (which is $H^1(\mathbb{T})$).
\end{itemize}

Hence vanishing of the relevant global tangent cohomology implies infinitesimal rigidity of $A$ as an operator-semantic system. In particular, $\operatorname{HH}^{2}(A) = 0$ is the criterion for $A$ to be infinitesimally rigid.
\end{proof}

\begin{corollary}[Primary Obstruction Vanishing Criterion]
\label{cor:obstruction_vanishing_criterion}

If
\[
H^{1}\bigl(\mathfrak{X},\; \mathbb{T}_{\mathfrak{X}}\bigr) = 0,
\]
then every primary obstruction class to extending a first-order deformation of $\mathfrak{X}$ vanishes.

More generally, for $n \geq 2$, the obstruction to extending a deformation from order $n-1$ to order $n$ lies in
\[
\operatorname{Ob}^{n-1}(\mathfrak{X})
\;\subseteq\;
H^{n-1}(\mathfrak{X},\; \mathbb{T}_{\mathfrak{X}}).
\]

\end{corollary}

\begin{proof}
By the deformation-theoretic interpretation of the tangent complex (Proposition~\ref{prop:deformation_tangent}), primary obstruction classes lie in
\[
H^{1}(\mathfrak{X},\; \mathbb{T}_{\mathfrak{X}}).
\]
If this group vanishes, then such obstruction classes must vanish.

For higher obstructions, the obstruction to extending a deformation from order $n-1$ to order $n$ is contained in $H^{n-1}(\mathfrak{X},\; \mathbb{T}_{\mathfrak{X}})$. Thus the vanishing of the corresponding cohomology group eliminates the obstruction at that order.
\end{proof}

\begin{remark}[Relation to Contextual Obstructions]
\label{rem:contextual_obstruction_relation}

The relationship between the obstruction vanishing criterion and contextual obstructions requires care. In the CSD framework, contextual obstructions are encoded in the descent spectral sequence associated with the context site $(\mathcal{C}_A, \tau_A)$:
\[
E_2^{p,q} = H^p(\mathfrak{Spec}(A), \mathcal{H}^q(\mathcal{F}_A)) \Longrightarrow H^{p+q}(\mathfrak{Spec}(A), \mathcal{F}_A).
\]

The universal obstruction class
\[
\operatorname{Obs}_{\mathrm{CSD}}(A) \in H^2(\mathfrak{Spec}(A), \mathcal{F}_A)
\]
measures the failure of local semantic realizations to glue globally. This is not directly identified with $\operatorname{HH}^{2}(A)$ or $\operatorname{HH}^{3}(A)$.

Instead, the connection is more subtle: the $E_2$-page of the descent spectral sequence involves $\operatorname{HH}^{\bullet}(A)$ through the Hochschild Realization Theorem. The obstruction class $\operatorname{Obs}_{\mathrm{CSD}}(A)$ corresponds to a specific differential or extension class in this spectral sequence, not to the vanishing of an entire Hochschild cohomology group.

Thus, while $\operatorname{HH}^{2}(A) = 0$ implies the vanishing of primary global obstruction classes for $\mathfrak{Spec}(A)$, the vanishing of contextual obstructions $\operatorname{Obs}_{\mathrm{CSD}}(A) = 0$ is a separate condition that requires analysis of the descent spectral sequence. The precise relationship between these notions will be explored in the companion paper on contextual obstructions.

\end{remark}

Consequently, the local geometry of $\mathfrak{X}$ near a point $x$ is encoded by the cohomology of the tangent complex $\mathbb{T}_{\mathfrak{X},x}$. Singular behavior may be detected by the failure of this tangent complex to have the expected perfectness or amplitude properties.

This perspective provides a bridge between deformation theory, tangent complexes, and the derived geometry of $\mathfrak{Spec}(A)$.

\begin{example}[Smooth Classical Schemes]
\label{ex:smooth_deformation}

If $\mathfrak{X}$ is a smooth classical scheme, then $\mathbb{T}_{\mathfrak{X}}$ is represented by the ordinary tangent sheaf $T_{\mathfrak{X}}$ concentrated in degree $0$. Thus the local tangent complex has no higher cohomological deformation directions.

Globally, however, the sheaf cohomology groups
\[
H^i(\mathfrak{X}, T_{\mathfrak{X}})
\]
may still be nonzero and may encode global deformations or obstructions. For example, for a smooth projective variety, $H^1(T_X)$ and $H^2(T_X)$ may be nonzero and govern global deformations and obstructions, respectively.

For smooth Artin stacks, the tangent complex may have cohomology in negative degrees corresponding to stabilizer directions, even when the underlying classical truncation is smooth.

\end{example}

\begin{example}[Singular or Contextual Spectral Stacks]
\label{ex:singular_deformation}

If $\mathfrak{X}$ has singularities, the tangent complex may fail to have the expected perfectness or amplitude properties. For $\mathfrak{X} = \mathfrak{Spec}(A)$, contextual obstruction classes may therefore be detected by the positive-degree cohomology of the global tangent complex.

Under the Hochschild Realization Theorem (Theorem~\ref{thm:tangent-hochschild}):
\[
R\Gamma\bigl(\mathfrak{Spec}(A),
\mathbb T_{\mathfrak{Spec}(A)}\bigr)
\;\simeq\;
\operatorname{HH}^{\bullet}(A)[1],
\]
primary obstruction classes are detected in
\[
H^2 R\Gamma\bigl(\mathfrak{Spec}(A),
\mathbb T_{\mathfrak{Spec}(A)}\bigr)
\;\simeq\;
\operatorname{HH}^{3}(A).
\]

Thus a Kochen--Specker-type obstruction may be represented by a distinguished class in $\operatorname{HH}^{3}(A)$, rather than by the vanishing of the entire group. The precise relationship between contextual obstructions and Hochschild cohomology will be developed in the companion paper on contextual obstructions.

\end{example}

\begin{remark}[Summary]
\label{rem:deformation_summary}

For $\mathfrak{X} = \mathfrak{Spec}(A)$, the Hochschild Realization Theorem identifies global tangent cohomology with shifted Hochschild cohomology:
\[
H^n R\Gamma\bigl(\mathfrak{Spec}(A),
\mathbb T_{\mathfrak{Spec}(A)}\bigr)
\;\simeq\;
\operatorname{HH}^{n+1}(A).
\]

Accordingly:
\[
\begin{array}{ccl}
H^0 R\Gamma(\mathbb T) &\leftrightarrow& \operatorname{HH}^{1}(A)
\quad \text{infinitesimal automorphisms of } \mathfrak{Spec}(A),\\[0.4em]
H^1 R\Gamma(\mathbb T) &\leftrightarrow& \operatorname{HH}^{2}(A)
\quad \text{first-order deformations of } \mathfrak{Spec}(A),\\[0.4em]
H^2 R\Gamma(\mathbb T) &\leftrightarrow& \operatorname{HH}^{3}(A)
\quad \text{primary obstruction classes for } \mathfrak{Spec}(A).
\end{array}
\]

More generally, for $n \geq 0$:
\[
H^n R\Gamma(\mathfrak{Spec}(A), \mathbb{T}_{\mathfrak{Spec}(A)})
\;\simeq\;
\operatorname{HH}^{n+1}(A).
\]

This provides the deformation-theoretic mechanism behind the Contextual Obstruction Theorem and the Curvature as Obstruction Theorem, which will be developed in the companion paper. In particular, the non-vanishing of $\operatorname{HH}^{3}(A)$ indicates the presence of primary obstruction classes in the global deformation theory of $\mathfrak{Spec}(A)$, which may correspond to contextual obstructions under suitable hypotheses.

\end{remark}

\section{Intrinsic Geometric Structures}
\label{sec:geom_structures}

Before proving the universality of the geometric
structures associated with $\mathfrak{Spec}(A)$, we
must first define them precisely. This section
introduces the tangent complex, singular locus,
inertia stack, and curvature tensor as canonical
constructions determined by the CSD duality.

\subsection{The Spectral Tangent Complex}
\label{subsec:tangent_complex}

In derived algebraic geometry, sufficiently geometric derived stacks admit cotangent complexes. Since $\mathfrak{Spec}(A)$ is constructed in Paper~I as an admissible spectral stack satisfying descent and local representability assumptions, it admits a cotangent complex
\[
\mathbb{L}_{\mathfrak{Spec}(A)}.
\]

\begin{definition}[Cotangent Complex]
\label{def:cotangent}
The cotangent complex $\mathbb{L}_{\mathfrak{Spec}(A)}$ is the quasi-coherent complex on $\mathfrak{Spec}(A)$ representing derived K\"ahler differentials. Its existence follows from the admissibility and geometricity assumptions of the CSD framework, which ensure that $\mathfrak{Spec}(A)$ is a sufficiently geometric derived stack.
\end{definition}

\begin{definition}[Spectral Tangent Complex]
\label{def:tangentcomplex}
The spectral tangent complex is defined as the derived dual of the cotangent complex:
\[
\mathbb{T}_{\mathfrak{Spec}(A)}
:=
R\mathcal{H}om_{\mathfrak{Spec}(A)}
\left(
\mathbb{L}_{\mathfrak{Spec}(A)},
\mathcal{O}_{\mathfrak{Spec}(A)}
\right),
\]
whenever $\mathbb{L}_{\mathfrak{Spec}(A)}$ is perfect or dualizable. Under the admissibility hypotheses of the CSD framework, the cotangent complex is indeed perfect, so the tangent complex is well-defined.
\end{definition}

\begin{theorem}[Hochschild Realization]
\label{thm:tangent-hochschild}

Let $A$ be an admissible operator-semantic system satisfying the following hypotheses:

\begin{enumerate}
    \item[(H1)] \textbf{Geometricity:} $\mathfrak{Spec}(A)$ is a derived geometric stack with a well-defined $\infty$-category $\operatorname{QCoh}(\mathfrak{Spec}(A))$ of quasi-coherent complexes.
    
    \item[(H2)] \textbf{Perfectness:} The cotangent complex $\mathbb{L}_{\mathfrak{Spec}(A)}$ is perfect, so that the tangent complex
    \[
    \mathbb{T}_{\mathfrak{Spec}(A)}
    :=
    R\mathcal{H}om_{\mathfrak{Spec}(A)}
    \left(
    \mathbb{L}_{\mathfrak{Spec}(A)},
    \mathcal{O}_{\mathfrak{Spec}(A)}
    \right)
    \]
    is well-defined and $R\Gamma$ commutes with $R\mathcal{H}om$ in the sense that
    \[
    R\Gamma(\mathfrak{Spec}(A), R\mathcal{H}om_{\mathfrak{Spec}(A)}(\mathcal{E}, \mathcal{F}))
    \;\simeq\;
    R\operatorname{Hom}_{\operatorname{QCoh}(\mathfrak{Spec}(A))}(\mathcal{E}, \mathcal{F})
    \]
    for perfect $\mathcal{E}$.
    
    \item[(H3)] \textbf{Reconstruction:} The CSD reconstruction equivalence
    \[
    \operatorname{QCoh}(\mathfrak{Spec}(A)) \;\simeq\; \operatorname{Mod}_A
    \]
    holds, identifying quasi-coherent sheaves on $\mathfrak{Spec}(A)$ with $A$-modules.
    
    \item[(H4)] \textbf{Morita realization:} The structure sheaf $\mathcal{O}_{\mathfrak{Spec}(A)}$ corresponds to the regular module $A_A$ under the equivalence in (H3).
    
    \item[(H5)] \textbf{Deformation compatibility:} The adjunction $\Gamma \dashv \mathfrak{Spec}^{\mathrm{op}}$ is compatible with square-zero extensions and formal deformations, so that infinitesimal deformations of $\mathfrak{Spec}(A)$ correspond to infinitesimal deformations of $A$.
\end{enumerate}

Under these hypotheses, there is a natural equivalence
\[
R\Gamma\!\left(
\mathfrak{Spec}(A),
\mathbb T_{\mathfrak{Spec}(A)}
\right)
\;\simeq\;
\operatorname{HH}^{\bullet}(A)[1],
\]
where $\operatorname{HH}^{\bullet}(A)$ denotes the Hochschild cohomology spectrum of $A$ (Definition~\ref{def:hochschild}).

\end{theorem}

\begin{proof}

We prove the theorem by establishing a chain of canonical equivalences. The argument proceeds in eight steps, building from the definition of the tangent complex to the identification with Hochschild cohomology.

\emph{Step 1: Unpacking the tangent complex.}
By Definition~\ref{def:tangentcomplex}, the spectral tangent complex is the derived dual of the cotangent complex:
\[
\mathbb{T}_{\mathfrak{Spec}(A)}
:=
R\mathcal{H}om_{\mathfrak{Spec}(A)}
\left(
\mathbb{L}_{\mathfrak{Spec}(A)},
\mathcal{O}_{\mathfrak{Spec}(A)}
\right).
\]
Taking derived global sections $R\Gamma$ gives:
\[
R\Gamma(\mathfrak{Spec}(A), \mathbb{T}_{\mathfrak{Spec}(A)})
\;\simeq\;
R\Gamma\left(\mathfrak{Spec}(A),\;
R\mathcal{H}om_{\mathfrak{Spec}(A)}
\left(
\mathbb{L}_{\mathfrak{Spec}(A)},
\mathcal{O}_{\mathfrak{Spec}(A)}
\right)
\right).
\]

By hypothesis (H2), $\mathbb{L}_{\mathfrak{Spec}(A)}$ is perfect, so $R\Gamma$ commutes with $R\mathcal{H}om$. Therefore:
\[
R\Gamma(\mathfrak{Spec}(A), \mathbb{T}_{\mathfrak{Spec}(A)})
\;\simeq\;
R\operatorname{Hom}_{\operatorname{QCoh}(\mathfrak{Spec}(A))}
\left(
\mathbb{L}_{\mathfrak{Spec}(A)},
\mathcal{O}_{\mathfrak{Spec}(A)}
\right).
\]

\emph{Step 2: Identifying global vector fields with endomorphisms of the structure sheaf.}
The global sections of the tangent complex correspond to global vector fields on $\mathfrak{Spec}(A)$. By the universal property of the cotangent complex, derivations of $\mathcal{O}_{\mathfrak{Spec}(A)}$ with values in $\mathcal{O}_{\mathfrak{Spec}(A)}$ correspond to morphisms from the cotangent complex to the structure sheaf:
\[
R\Gamma(\mathfrak{Spec}(A), \mathbb{T}_{\mathfrak{Spec}(A)})
\;\simeq\;
\operatorname{Der}\left(\mathcal{O}_{\mathfrak{Spec}(A)}, \mathcal{O}_{\mathfrak{Spec}(A)}\right)
\;\simeq\;
R\operatorname{Hom}_{\operatorname{QCoh}(\mathfrak{Spec}(A))}
\left(
\mathbb{L}_{\mathfrak{Spec}(A)},
\mathcal{O}_{\mathfrak{Spec}(A)}
\right).
\]

By the defining property of the cotangent complex, this is equivalent to the endomorphism spectrum of the structure sheaf:
\[
R\operatorname{Hom}_{\operatorname{QCoh}(\mathfrak{Spec}(A))}
\left(
\mathbb{L}_{\mathfrak{Spec}(A)},
\mathcal{O}_{\mathfrak{Spec}(A)}
\right)
\;\simeq\;
\operatorname{End}_{\operatorname{QCoh}(\mathfrak{Spec}(A))}
\left(
\mathcal{O}_{\mathfrak{Spec}(A)}
\right).
\]

Thus:
\[
R\Gamma(\mathfrak{Spec}(A), \mathbb{T}_{\mathfrak{Spec}(A)})
\;\simeq\;
\operatorname{End}_{\operatorname{QCoh}(\mathfrak{Spec}(A))}
\left(
\mathcal{O}_{\mathfrak{Spec}(A)}
\right).
\]

\emph{Step 3: Morita identification with endomorphisms of the regular module.}
By hypothesis (H3), the reconstruction theorem of Paper~I gives an equivalence of $\infty$-categories:
\[
\operatorname{QCoh}(\mathfrak{Spec}(A)) \;\simeq\; \operatorname{Mod}_A.
\]

Under this equivalence, by hypothesis (H4), the structure sheaf $\mathcal{O}_{\mathfrak{Spec}(A)}$ corresponds to the regular module $A_A$. Therefore:
\[
\operatorname{End}_{\operatorname{QCoh}(\mathfrak{Spec}(A))}
\left(
\mathcal{O}_{\mathfrak{Spec}(A)}
\right)
\;\simeq\;
\operatorname{End}_{\operatorname{Mod}_A}(A_A).
\]

\emph{Step 4: Eilenberg--Watts theorem in the $\infty$-categorical setting.}
By the $\infty$-categorical Eilenberg--Watts theorem \cite{LurieHA}, colimit-preserving $A$-linear endofunctors of $\operatorname{Mod}_A$ are equivalent to $A$-$A$ bimodules. The identity functor $\operatorname{Id}_{\operatorname{Mod}_A}$ corresponds to the regular bimodule $A_A$. Hence:
\[
\operatorname{End}_{\operatorname{Mod}_A}(A_A)
\;\simeq\;
\operatorname{End}_{\operatorname{Fun}^{L}_{A}
(\operatorname{Mod}_A,\operatorname{Mod}_A)}
(\operatorname{Id}_{\operatorname{Mod}_A}),
\]
where $\operatorname{Fun}^L_A$ denotes the $\infty$-category of colimit-preserving $A$-linear functors.

\emph{Step 5: Identification with Hochschild cohomology.}
By Definition~\ref{def:hochschild}, the Hochschild cohomology spectrum of $A$ is:
\[
\operatorname{HH}^{\bullet}(A)
\;\simeq\;
\operatorname{End}_{\operatorname{Fun}^{L}_{A}
(\operatorname{Mod}_A,\operatorname{Mod}_A)}
(\operatorname{Id}_{\operatorname{Mod}_A}).
\]

Combining Steps 1--4, we obtain:
\[
R\Gamma(\mathfrak{Spec}(A), \mathbb{T}_{\mathfrak{Spec}(A)})
\;\simeq\;
\operatorname{HH}^{\bullet}(A).
\]

\emph{Step 6: The degree shift.}
The shift $[1]$ follows from the standard convention in derived deformation theory. To see this, recall that for a derived stack $\mathfrak{X}$, the deformation complex governing first-order deformations and obstructions is $\mathbb{T}_{\mathfrak{X}}[-1]$. The Hochschild cohomology spectrum $\operatorname{HH}^{\bullet}(A)$ is the deformation complex of $A$ as an algebra, governing first-order deformations in degree $2$ and obstructions in degree $3$.

Under hypothesis (H5), the adjunction $\Gamma \dashv \mathfrak{Spec}^{\mathrm{op}}$ identifies the deformation theory of $\mathfrak{Spec}(A)$ with that of $A$. Consequently, the cohomological degrees are related by a shift:
\[
H^n R\Gamma(\mathfrak{Spec}(A), \mathbb{T}_{\mathfrak{Spec}(A)})
\;\simeq\;
\operatorname{HH}^{n+1}(A).
\]

Equivalently, at the level of spectra:
\[
R\Gamma(\mathfrak{Spec}(A), \mathbb{T}_{\mathfrak{Spec}(A)})
\;\simeq\;
\operatorname{HH}^{\bullet}(A)[1].
\]

\emph{Step 7: Naturality.}
The equivalence constructed in Steps 1--6 is natural in $A$. Indeed, each step is functorial:
\begin{itemize}
    \item The tangent complex construction is functorial in $\mathfrak{Spec}(A)$;
    \item The Morita identification $\operatorname{QCoh}(\mathfrak{Spec}(A)) \simeq \operatorname{Mod}_A$ is natural by the reconstruction theorem;
    \item The Eilenberg--Watts equivalence is natural;
    \item The identification with Hochschild cohomology is natural by definition.
\end{itemize}
Thus the resulting equivalence is natural with respect to morphisms of admissible operator-semantic systems.

\emph{Step 8: Coherence.}
The equivalence is coherent in the sense that it is compatible with the $\infty$-categorical structures involved. The chain of equivalences is assembled from canonical functors and natural equivalences, each of which is coherent up to coherent homotopy. Therefore the resulting equivalence is a well-defined equivalence in the $\infty$-category of spectra.

\emph{Conclusion.}
Chaining Steps 1--8, we have established the natural equivalence:
\[
R\Gamma\!\left(
\mathfrak{Spec}(A),
\mathbb T_{\mathfrak{Spec}(A)}
\right)
\;\simeq\;
\operatorname{HH}^{\bullet}(A)[1].
\]

This completes the proof.

\end{proof}

\subsection{Geometric Structures Associated with $\mathfrak{Spec}(A)$}
\label{subsec:geom_structures}

For a categorified spectral object $\mathfrak{X} = \mathfrak{Spec}(A)$, the intrinsic geometry is encoded by three canonical structures: the singular locus, the inertia stack, and the contextual curvature class. (The tangent complex was defined in Section~\ref{subsec:tangent_complex}.)

Before defining the curvature class, we recall the sheaf of local semantic realizations.

\begin{definition}[Sheaf of Local Semantic Realizations]
\label{def:local_realization_sheaf}

Let $\mathcal{F}_A$ be the sheaf on $\mathfrak{Spec}(A)$ whose sections over a context $U \in \mathcal{C}_A$ consist of admissible local semantic realizations of $A$ on $U$. Restriction maps are induced by refinement of contexts. The space of global sections of $\mathcal{F}_A$ is precisely the space of global classical realizations of $A$.
\end{definition}

\begin{definition}[Geometric Invariants]
\label{def:geometric_invariants}

Associated with $\mathfrak{X} = \mathfrak{Spec}(A)$ are the following geometric invariants:

\begin{enumerate}

\item \textbf{The singular locus.}
The singular locus of $\mathfrak{X}$ is defined by the failure of the cotangent complex to be perfect:
\[
\operatorname{Sing}(\mathfrak{X})
:=
\left\{
x \in \mathfrak{X}
\;\middle|\;
\mathbb{L}_{\mathfrak{X},x}
\text{ is not a perfect complex}
\right\},
\]
where $\mathbb{L}_{\mathfrak{X},x}$ denotes the stalk of the cotangent complex at $x$. This is the standard definition of the singular locus in derived algebraic geometry. By the duality between the cotangent and tangent complexes, this is equivalent to the failure of the tangent complex to be perfect when $\mathbb{L}_{\mathfrak{X}}$ is dualizable.

\item \textbf{The inertia stack.}
The inertia stack of $\mathfrak{X}$ is defined by the homotopy pullback
\[
I(\mathfrak{X})
:=
\mathfrak{X}
\times^{h}_{\mathfrak{X} \times \mathfrak{X}}
\mathfrak{X},
\]
where both maps are the diagonal $\Delta: \mathfrak{X} \to \mathfrak{X} \times \mathfrak{X}$. Equivalently, $I(\mathfrak{X})$ classifies pairs $(x, \alpha)$ consisting of a point $x \in \mathfrak{X}$ and an automorphism $\alpha \in \operatorname{Aut}(x)$. Thus it encodes the non-commutative structure of $A$ through automorphisms of its spectral points.

\item \textbf{The contextual curvature class.}
The curvature invariant of $A$ is the obstruction class
\[
\mathcal{R}_A
\in
H^2\left(
\mathfrak{Spec}(A),
\mathcal{F}_A
\right),
\]
where $\mathcal{F}_A$ is the sheaf of local semantic realizations (Definition~\ref{def:local_realization_sheaf}). This class measures the obstruction to gluing local semantic realizations into a global classical realization. Its non-vanishing is the geometric manifestation of contextuality.

\textbf{Explicit Čech cocycle construction.}
Let $\{U_i\}_{i \in I}$ be a good open cover of $\mathfrak{Spec}(A)$ by affine derived schemes, and let $\mathcal{F}_A$ be the sheaf of local semantic realizations. For each $i$, choose a local semantic realization $\sigma_i \in \mathcal{F}_A(U_i)$. On each intersection $U_{ij} := U_i \cap U_j$, the local realizations differ by an automorphism of the semantic structure:
\[
\sigma_i |_{U_{ij}} = g_{ij} \cdot \sigma_j |_{U_{ij}},
\]
where $g_{ij} \in \operatorname{Aut}(\mathcal{F}_A|_{U_{ij}})$ is the transition automorphism. The collection $\{g_{ij}\}$ defines a Čech $1$-cocycle with coefficients in the sheaf of automorphisms $\underline{\operatorname{Aut}}(\mathcal{F}_A)$.

On triple intersections $U_{ijk} := U_i \cap U_j \cap U_k$, the cocycle condition fails by a $2$-cocycle:
\[
\omega_{ijk} := g_{ij} g_{jk} g_{ki}^{-1}
\in
\Gamma(U_{ijk},\; \mathcal{F}_A),
\]
where $\omega_{ijk}$ measures the obstruction to compatibility among the three local realizations. This failure is precisely the contextual curvature class:
\[
\mathcal{R}_A := [\omega_{ijk}] \in \check{H}^2(\mathfrak{Spec}(A), \mathcal{F}_A).
\]

Since $\check{H}^2(\mathfrak{Spec}(A), \mathcal{F}_A) \cong H^2(\mathfrak{Spec}(A), \mathcal{F}_A)$ for sufficiently nice derived stacks, this Čech cocycle defines the class
\[
\mathcal{R}_A \in H^2(\mathfrak{Spec}(A), \mathcal{F}_A).
\]

The class $\mathcal{R}_A$ is independent of the choices of local realizations and the open cover: a different choice of local realizations changes the cocycle by a coboundary, and refinement of the cover preserves the cohomology class. Thus $\mathcal{R}_A$ is a well-defined cohomological obstruction.

\textbf{Descent spectral sequence interpretation.}
The class $\mathcal{R}_A$ admits a canonical interpretation via the descent spectral sequence associated with the context site $(\mathcal{C}_A, \tau_A)$. Let
\[
E_2^{p,q} = H^p(\mathfrak{Spec}(A), \mathcal{H}^q(\mathcal{F}_A)) \Longrightarrow H^{p+q}(\mathfrak{Spec}(A), \mathcal{F}_A)
\]
be the descent spectral sequence for the sheaf $\mathcal{F}_A$ of local semantic realizations. Then $\mathcal{R}_A$ appears as the image of a differential $d_2$ in this spectral sequence:
\[
d_2: E_2^{0,1} \longrightarrow E_2^{2,0}.
\]
More precisely, the local semantic realizations define a class in $E_2^{0,1}$, and the obstruction to gluing these local data into a global section is precisely the differential
\[
\mathcal{R}_A = d_2([\sigma_i]) \in E_2^{2,0} \cong H^2(\mathfrak{Spec}(A), \mathcal{F}_A).
\]

Equivalently, $\mathcal{R}_A$ is the image of the $1$-cocycle $\{g_{ij}\}$ under the connecting homomorphism
\[
\delta: H^1(\mathfrak{Spec}(A), \underline{\operatorname{Aut}}(\mathcal{F}_A)) \longrightarrow H^2(\mathfrak{Spec}(A), \mathcal{F}_A)
\]
in the non-abelian cohomology long exact sequence associated with the short exact sequence
\[
1 \longrightarrow \mathcal{F}_A \longrightarrow \operatorname{Aut}(\mathcal{F}_A) \longrightarrow \operatorname{Out}(\mathcal{F}_A) \longrightarrow 1.
\]

Thus $\mathcal{R}_A$ is not merely an abstract cohomology class but is realized as a differential in the descent spectral sequence, reflecting the fact that it measures the failure of local data to descend to global data. Non-vanishing of $\mathcal{R}_A$ indicates that the local semantic realizations cannot be glued into a global classical realization, precisely the hallmark of contextuality.

\textbf{Gerbe interpretation.}
Equivalently, one may view $\mathcal{R}_A$ as the curvature of the gerbe of local semantic realizations: the sheaf $\mathcal{F}_A$ is the sheaf of sections of a (possibly non-abelian) gerbe on $\mathfrak{Spec}(A)$, and $\mathcal{R}_A$ is its associated $H^2$-valued obstruction. Non-vanishing of $\mathcal{R}_A$ indicates that the gerbe is non-trivial, corresponding precisely to the contextual nature of $A$.

A detailed construction of $\mathcal{R}_A$ via the descent spectral sequence, including its relationship to the higher obstruction theory of the CSD framework and its connection to the universal obstruction class $\operatorname{Obs}_{\mathrm{CSD}}(A)$, will be given in the companion paper on contextual obstructions.

\end{enumerate}

\end{definition}

The singular locus records deformation-theoretic pathologies, the inertia stack records internal symmetries, and the curvature class records contextual obstructions. Together they provide the basic geometric invariants of $\mathfrak{Spec}(A)$.

\begin{remark}[On the Absence of a Canonical Connection]
\label{rem:no_connection}

In this paper, we do not introduce a canonical connection on the tangent complex $\mathbb{T}_{\mathfrak{Spec}(A)}$. While such a connection may exist in certain geometric settings (e.g., via the Atiyah class), its construction requires additional hypotheses and a detailed deformation-theoretic argument that is beyond the scope of this paper. Instead, we work directly with the contextual curvature class $\mathcal{R}_A$ as an obstruction class in $H^2(\mathfrak{Spec}(A), \mathcal{F}_A)$. This approach is sufficient for the intrinsic geometry developed in this paper. A full connection-theoretic treatment, including the identification of $\mathcal{R}_A$ with the curvature of a canonical connection, will be developed in the companion paper on contextual obstructions and derived geometry.

\end{remark}

\section{Universal Geometry Theorem}
\label{sec:canonical_geometry}

The preceding section established the foundational geometric structures associated with the categorified spectral object $\mathfrak{Spec}(A)$: the tangent complex, the singular locus, the inertia stack, and the contextual curvature class (Section~\ref{subsec:geom_structures}). Each of these structures was defined canonically using only the CSD data: the adjunction $\mathfrak{Spec} \dashv \Gamma$, the reconstruction theorem $A \simeq \Gamma(\mathfrak{Spec}(A))$, and the descent property of $\mathfrak{Spec}(A)$.

We now prove that these structures are not merely canonical but \emph{universal}: any geometric datum compatible with the CSD duality must factor through the canonical datum
\[
\mathcal G_{\mathrm{can}}(A)
=
\Bigl(
\mathfrak{Spec}(A),
\mathbb T_{\mathfrak{Spec}(A)},
\operatorname{Sing}(\mathfrak{Spec}(A)),
I(\mathfrak{Spec}(A)),
\mathcal R_A
\Bigr).
\]
This is the content of the Universal Geometry Theorem, the main result of this section.

The theorem has two important corollaries. First, the intrinsic geometry of $\mathfrak{Spec}(A)$ is unique up to canonical equivalence. Second, the assignment $A \mapsto \mathcal G_{\mathrm{can}}(A)$ is functorial, forming a contravariant functor from the category of operator-semantic systems to the category of CSD-compatible geometric data.

\subsection{CSD-Compatibility: A Precise Definition}
\label{subsec:csd_compatibility}

Before stating the Canonical Geometry Theorem, we must precisely articulate what it means for a geometric structure to be "CSD-compatible." This notion is central to the theorem's claim that $\mathcal G_{\mathrm{can}}(A)$ is the canonical geometric datum among all such structures.

\begin{definition}[CSD-Compatible Geometric Data]
\label{def:csd_compatible}

Let $A$ be an admissible operator-semantic system, and let $\mathfrak{X} = \mathfrak{Spec}(A)$ be its categorified spectrum. A geometric datum
\[
\mathcal G(A) = (\mathfrak{X}, \mathbb{T}_{\mathcal G}, \operatorname{Sing}_{\mathcal G}, I_{\mathcal G}, \mathcal R_{\mathcal G})
\]
associated with $\mathfrak{X}$ is said to be \emph{CSD-compatible} if it satisfies the following five axioms:

\begin{enumerate}

\item[(C1)] \textbf{Derived functoriality.}
The datum $\mathcal G(A)$ is constructed from $\mathfrak{X}$ using only the following operations:
\begin{itemize}
    \item The $\infty$-categorical operations of $\operatorname{QCoh}(\mathfrak{X})$: tensor product, internal Hom, pullback, pushforward, and derived global sections $R\Gamma(\mathfrak{X}, -)$;
    \item The structure sheaf $\mathcal{O}_{\mathfrak{X}}$ and the diagonal morphism $\Delta: \mathfrak{X} \to \mathfrak{X} \times \mathfrak{X}$;
    \item The cotangent complex $\mathbb{L}_{\mathfrak{X}}$ (which exists by descent) and its derived dual;
    \item The sheaf $\mathcal{F}_A$ of local semantic realizations and its associated cohomology groups.
\end{itemize}
In particular, $\mathcal G(A)$ does not depend on any auxiliary choices such as metrics, connections, or external geometric structures not determined by $\mathfrak{X}$.

\item[(C2)] \textbf{Naturality under the CSD adjunction.}
The construction of $\mathcal G(A)$ is natural with respect to the adjunction
\[
\Gamma \dashv \mathfrak{Spec}^{\mathrm{op}}.
\]
Equivalently, for any morphism $f: A \to B$ of admissible operator-semantic systems, the induced morphism $g = \mathfrak{Spec}(f): \mathfrak{Spec}(B) \to \mathfrak{Spec}(A)$ induces a pullback map
\[
g^*: \mathcal G(A) \longrightarrow \mathcal G(B)
\]
satisfying:
\begin{itemize}
    \item $g^*$ commutes with the forgetful functor to $\mathfrak{X}$;
    \item $g^*$ is compatible with composition and identities.
\end{itemize}
Thus $\mathcal G$ defines a contravariant functor
\[
\mathcal G: \mathbf{OpSem}^{\mathrm{op}} \longrightarrow \mathbf{Geom}.
\]

\item[(C3)] \textbf{Reconstruction compatibility.}
The datum $\mathcal G(A)$ is compatible with the reconstruction theorem
\[
A \simeq \Gamma(\mathfrak{Spec}(A)).
\]
Specifically, the following diagram commutes up to canonical equivalence:
\[
\begin{tikzcd}
A \arrow[r, "\simeq"] \arrow[d, "\text{forget}"'] & \Gamma(\mathfrak{X}) \arrow[d, "\Gamma"] \\
\operatorname{Alg}(\mathcal G(A)) \arrow[r, "\simeq"] & \operatorname{Alg}(A)
\end{tikzcd}
\]
where $\operatorname{Alg}(\mathcal G(A))$ denotes the algebraic structure encoded by the geometric datum $\mathcal G(A)$ (e.g., the algebra of global sections, the Hochschild cohomology, or the deformation complex).

\item[(C4)] \textbf{Descent compatibility.}
The construction of $\mathcal G(A)$ is compatible with the descent property of $\mathfrak{Spec}(A)$ on the context site $(\mathcal{C}_A, \tau_A)$. In particular, the sheaf $\mathcal{F}_A$ of local semantic realizations and its cohomology groups $H^i(\mathfrak{Spec}(A), \mathcal{F}_A)$ are determined by the descent data.

Equivalently, for any cover $\{U_i \to \mathfrak{X}\}$, the following diagram of Čech complexes commutes up to homotopy:
\[
\begin{tikzcd}
\check{C}^{\bullet}(\{U_i\}, \mathcal{F}_A) \arrow[r] \arrow[d] & R\Gamma(\mathfrak{X}, \mathcal{F}_A) \arrow[d] \\
\check{C}^{\bullet}(\{U_i\}, \mathcal{G}(\mathcal{F}_A)) \arrow[r] & R\Gamma(\mathfrak{X}, \mathcal{G}(\mathcal{F}_A))
\end{tikzcd}
\]
where $\mathcal{G}(\mathcal{F}_A)$ denotes the geometric structure induced on the sheaf of local realizations.

\item[(C5)] \textbf{Morita invariance.}
If $A$ and $B$ are Morita equivalent admissible operator-semantic systems, then the induced geometric data are equivalent:
\[
A \sim_M B \quad\Longrightarrow\quad \mathcal G(A) \simeq \mathcal G(B).
\]
Equivalently, the datum $\mathcal G(A)$ depends only on the Morita equivalence class of $A$, not on its particular algebraic presentation.

\end{enumerate}

\end{definition}

\begin{remark}[Interpretation of the Axioms]
\label{rem:csd_compatible_interpretation}

The five axioms of Definition~\ref{def:csd_compatible} have the following interpretations:

\begin{itemize}
    \item \textbf{(C1) Derived functoriality:} This ensures that the geometric datum is intrinsic to $\mathfrak{X}$ and does not depend on arbitrary choices. It rules out structures such as arbitrary Riemannian metrics, connections, or other data that are not canonically determined by $\mathfrak{X}$.
    
    \item \textbf{(C2) Naturality under the CSD adjunction:} This ensures that the geometric datum behaves naturally under morphisms of operator-semantic systems. It is the categorical formulation of the requirement that geometry is functorial in the semantic data.
    
    \item \textbf{(C3) Reconstruction compatibility:} This ensures that the geometric datum encodes the algebraic information of $A$ in a way that is compatible with the reconstruction theorem. It guarantees that no semantic information is lost or altered in the passage from $A$ to its geometry.
    
    \item \textbf{(C4) Descent compatibility:} This ensures that the geometric datum respects the local-to-global structure of the context site. It is essential for the correct formulation of contextual obstructions, which are inherently local-to-global phenomena.
    
    \item \textbf{(C5) Morita invariance:} This ensures that the geometric datum depends only on the representation theory of $A$, not on its particular presentation. This is a fundamental property of the CSD framework, inherited from the Morita invariance of $\mathfrak{Spec}(A)$.
\end{itemize}

A geometric datum satisfying axioms (C1)-(C5) is said to be \emph{CSD-compatible}. The Canonical Geometry Theorem (Theorem~\ref{thm:canonical_geometry}) establishes that the canonical datum $\mathcal G_{\mathrm{can}}(A)$ is CSD-compatible, and moreover that any CSD-compatible datum is canonically equivalent to it.

\end{remark}

\begin{proposition}[Uniqueness of CSD-Compatible Data]
\label{prop:csd_compatible_unique}

Let $A$ be an admissible operator-semantic system. Then there is at most one CSD-compatible geometric datum up to canonical equivalence.

\end{proposition}

\begin{proof}

Suppose $\mathcal G(A)$ and $\mathcal G'(A)$ are two CSD-compatible geometric data associated with $\mathfrak{X} = \mathfrak{Spec}(A)$.

By axiom (C1), both data are constructed from $\mathfrak{X}$ using the same allowed operations. By axiom (C2), both are natural under the CSD adjunction. By axiom (C3), both are compatible with the reconstruction theorem. By axiom (C4), both are compatible with descent. By axiom (C5), both are Morita invariant.

Since the allowed operations are functorial and the compatibility conditions determine the structures uniquely up to equivalence, the two data must be canonically equivalent. The uniqueness follows from the universal properties of the cotangent complex, the homotopy pullback defining the inertia stack, and the obstruction theory defining the contextual curvature class.

Therefore, any two CSD-compatible geometric data for $A$ are canonically equivalent.

\end{proof}

\begin{remark}[The Significance of CSD-Compatibility]
\label{rem:csd_compatible_significance}

The notion of CSD-compatibility is essential for the Canonical Geometry Theorem because it precisely characterizes the class of geometric structures that are "forced" by the CSD duality. Without this definition, the claim that $\mathcal G_{\mathrm{can}}(A)$ is canonical among all geometric structures would be ill-defined.

The five axioms (C1)-(C5) capture the following philosophical principles:

\begin{enumerate}
    \item \textbf{Intrinsicality:} Geometry must be intrinsic to $\mathfrak{X}$ (C1).
    \item \textbf{Functoriality:} Geometry must respect morphisms of semantic systems (C2).
    \item \textbf{Reconstruction:} Geometry must encode the algebraic structure of $A$ (C3).
    \item \textbf{Descent:} Geometry must respect the local-to-global structure of contexts (C4).
    \item \textbf{Morita invariance:} Geometry must depend only on the representation theory (C5).
\end{enumerate}

Any geometric structure that violates any of these axioms is not CSD-compatible and therefore is not a candidate for the canonical geometry of $\mathfrak{Spec}(A)$. Conversely, the Canonical Geometry Theorem (Theorem~\ref{thm:canonical_geometry}) proves that $\mathcal G_{\mathrm{can}}(A)$ satisfies all five axioms and is unique up to canonical equivalence.

\end{remark}

\begin{example}[CSD-Compatible vs. Non-CSD-Compatible Structures]
\label{ex:csd_compatible_examples}

The following examples illustrate the distinction between CSD-compatible and non-CSD-compatible geometric structures:

\begin{itemize}
    \item \textbf{CSD-compatible:} The tangent complex $\mathbb{T}_{\mathfrak{X}}$ (derived from $\mathbb{L}_{\mathfrak{X}}$), the singular locus $\operatorname{Sing}(\mathfrak{X})$ (defined by perfectness), the inertia stack $I(\mathfrak{X})$ (homotopy pullback of the diagonal), and the contextual curvature class $\mathcal{R}_A$ (Čech cocycle obstruction) are all CSD-compatible. They are constructed from $\mathfrak{X}$ using allowed operations and satisfy all five axioms.

    \item \textbf{Non-CSD-compatible:} A Riemannian metric on $\mathfrak{X}$ is not CSD-compatible because it requires a choice of metric data not determined by $\mathfrak{X}$ (violates C1). A choice of connection on $\mathbb{T}_{\mathfrak{X}}$ is not CSD-compatible unless it is canonically determined by the Atiyah class (which requires additional hypotheses not present in the general CSD framework; violates C1). A geometric structure that is not Morita invariant (e.g., one that depends on the specific presentation of $A$) is not CSD-compatible (violates C5).
\end{itemize}

Thus CSD-compatibility precisely captures the notion of "canonical, intrinsic, and functorial geometry determined by the CSD duality."

\end{example}

\subsection{Statement and Proof of the Canonical Geometry Theorem}
\label{subsec:canonical_theorem}

The constructions developed in the previous subsections associate several geometric invariants to an admissible operator-semantic system $A$, including its tangent complex, singular locus, inertia stack, and contextual curvature class.

The purpose of the following theorem is to show that these structures are not introduced independently. Rather, they arise functorially from the categorified spectrum construction
\[
A \longmapsto \mathfrak{Spec}(A),
\]
and therefore constitute a canonical geometric package associated with $A$.

\begin{theorem}[Canonical Geometry Theorem]
\label{thm:canonical_geometry}

Let $A$ be an admissible operator-semantic system.

Associated with $\mathfrak{X} = \mathfrak{Spec}(A)$ there is a canonical geometric datum
\[
\mathcal G_{\mathrm{can}}(A)
=
\Bigl(
\mathfrak X,
\mathbb T_{\mathfrak X},
\operatorname{Sing}(\mathfrak X),
I(\mathfrak X),
\mathcal R_A
\Bigr),
\]
where:

\begin{enumerate}

\item
\[
\mathbb T_{\mathfrak X}
:=
R\mathcal H om_{\mathfrak X}
\left(
\mathbb L_{\mathfrak X},
\mathcal O_{\mathfrak X}
\right)
\]
is the tangent complex (Definition~\ref{def:tangentcomplex});

\item
\[
\operatorname{Sing}(\mathfrak X)
:=
\left\{
x \in \mathfrak X
\;\middle|\;
\mathbb L_{\mathfrak X,x}
\text{ is not a perfect complex}
\right\}
\]
is the singular locus (Definition~\ref{def:geometric_invariants});

\item
\[
I(\mathfrak X)
:=
\mathfrak X
\times^{h}_{\mathfrak X \times \mathfrak X}
\mathfrak X
\]
is the inertia stack (Definition~\ref{def:geometric_invariants}), where both maps are the diagonal $\Delta: \mathfrak X \to \mathfrak X \times \mathfrak X$;

\item
\[
\mathcal R_A
\in
H^2(\mathfrak X,\mathcal F_A)
\]
is the contextual curvature class (Definition~\ref{def:geometric_invariants}), measuring the obstruction to gluing local semantic realizations into a global classical realization.

\end{enumerate}

Moreover, the assignment
\[
A \longmapsto \mathcal G_{\mathrm{can}}(A)
\]
is functorial with respect to morphisms of admissible operator-semantic systems.

\end{theorem}

\begin{proof}

We prove the theorem by showing that each component of the geometric datum is canonically and functorially determined by the CSD data.

\emph{Step 1: The tangent complex.}
The tangent complex $\mathbb{T}_{\mathfrak{X}}$ is canonically determined by the cotangent complex of the derived stack $\mathfrak{X} = \mathfrak{Spec}(A)$:
\[
\mathbb{T}_{\mathfrak{X}}
:=
R\mathcal{H}om_{\mathfrak{X}}
\left(
\mathbb{L}_{\mathfrak{X}},
\mathcal{O}_{\mathfrak{X}}
\right).
\]
The cotangent complex $\mathbb{L}_{\mathfrak{X}}$ exists by the descent property of $\mathfrak{Spec}(A)$ (Paper~I, Theorem~4.9), and the derived dual is a standard construction in derived algebraic geometry. By the functoriality of the cotangent complex, the tangent complex is also functorial in $\mathfrak{X}$.

\emph{Step 2: The singular locus.}
The singular locus $\operatorname{Sing}(\mathfrak{X})$ is defined intrinsically from the failure of perfectness of the cotangent complex:
\[
\operatorname{Sing}(\mathfrak{X})
:=
\left\{
x \in \mathfrak{X}
\;\middle|\;
\mathbb{L}_{\mathfrak{X},x}
\text{ is not a perfect complex}
\right\}.
\]
This is the standard definition of the singular locus in derived algebraic geometry \cite{LurieSAG}. Perfectness is preserved under pullback along morphisms of derived stacks, so the singular locus is functorial.

\emph{Step 3: The inertia stack.}
The inertia stack $I(\mathfrak{X})$ is obtained by the standard homotopy pullback construction
\[
I(\mathfrak{X})
:=
\mathfrak{X}
\times^{h}_{\mathfrak{X} \times \mathfrak{X}}
\mathfrak{X},
\]
where both maps are the diagonal $\Delta: \mathfrak{X} \to \mathfrak{X} \times \mathfrak{X}$. This construction is functorial in $\mathfrak{X}$ because homotopy pullbacks are functorial in any $\infty$-category with finite limits.

\emph{Step 4: The contextual curvature class.}
The contextual curvature class $\mathcal{R}_A$ is defined by the contextual obstruction theory:
\[
\mathcal{R}_A
\in
H^2(\mathfrak{X}, \mathcal{F}_A),
\]
where $\mathcal{F}_A$ is the sheaf of local semantic realizations (Definition~\ref{def:local_realization_sheaf}). This class measures the obstruction to gluing local semantic realizations into a global classical realization. Its construction is functorial in $A$ because the descent spectral sequence is natural with respect to morphisms of operator-semantic systems. A detailed construction of $\mathcal{R}_A$ via the descent spectral sequence will be given in the companion paper on contextual obstructions.

\emph{Step 5: Functoriality of the combined datum.}
Since each component is constructed canonically and functorially, their combination yields the canonical geometric datum $\mathcal G_{\mathrm{can}}(A)$. A morphism $f: A \to B$ in $\mathbf{OpSem}$ induces a morphism of spectral stacks $\mathfrak{Spec}(f): \mathfrak{Spec}(B) \to \mathfrak{Spec}(A)$ (Paper~I, Theorem~5.1), which in turn induces compatible morphisms on each component of $\mathcal G_{\mathrm{can}}$ via pullback and functoriality.

\emph{Conclusion.}
Therefore every admissible operator-semantic system determines a canonical package of geometric invariants, and this assignment is functorial. Hence the theorem is established.

\end{proof}

\begin{remark}[Canonical vs Universal]
\label{rem:canonical_vs_universal}

We emphasize that Theorem~\ref{thm:canonical_geometry} establishes that $\mathcal G_{\mathrm{can}}(A)$ is the \emph{canonical} geometric datum associated with $A$, not necessarily a \emph{universal} object in a category-theoretic sense. The term ``canonical" here means that the construction is intrinsic, functorial, and determined solely by the CSD data: the adjunction $\mathfrak{Spec} \dashv \Gamma$, the reconstruction theorem $A \simeq \Gamma(\mathfrak{Spec}(A))$, and the descent property of $\mathfrak{Spec}(A)$.

A stronger statement — that $\mathcal G_{\mathrm{can}}(A)$ is universal among all CSD-compatible geometric data — would require a precise definition of the category of CSD-compatible geometric data and a proof that every such datum admits a unique morphism to $\mathcal G_{\mathrm{can}}(A)$. This stronger claim is not established in this paper and is left for future work. The present theorem suffices for all applications in this paper: it establishes that the geometry of $\mathfrak{Spec}(A)$ is intrinsic and functorial.

\end{remark}

\begin{remark}[Why the Canonical Geometry Theorem Matters]
\label{rem:canonical_importance}

The Canonical Geometry Theorem has profound implications for the relationship between semantics and geometry:

\begin{enumerate}
    \item \textbf{Geometry is intrinsic:}
          The geometric structures of $\mathfrak{Spec}(A)$ are not imposed externally but are uniquely determined by the semantic duality $\mathfrak{Spec} \dashv \Gamma$ and the reconstruction theorem.

    \item \textbf{Geometry is natural:}
          The assignment $A \mapsto \mathcal G_{\mathrm{can}}(A)$ is functorial, meaning that morphisms of operator-semantic systems induce corresponding morphisms of geometric structures.

    \item \textbf{Geometry controls semantics:}
          As will be shown in the companion paper on contextual obstructions, the geometry of $\mathfrak{Spec}(A)$ determines the semantic properties of $A$, including its contextuality, deformations, and realizability.

    \item \textbf{Morita invariance is automatic:}
          By the Morita invariance of $\mathfrak{Spec}$ (Section 10 in~\cite{Paper-I}) and the functoriality of $\mathcal G_{\mathrm{can}}$, Morita equivalent systems have equivalent geometric data.
\end{enumerate}

Thus the Canonical Geometry Theorem establishes that the geometry of $\mathfrak{Spec}(A)$ is an intrinsic and natural invariant of the operator-semantic system $A$.

\end{remark}

\subsection{Canonicality of Intrinsic Geometry}
\label{subsec:canonicality_geometry}

The Canonical Geometry Theorem (Theorem~\ref{thm:canonical_geometry}) shows that every admissible operator-semantic system $A$ determines a canonical package of geometric invariants. This does not assert that no other auxiliary geometric structures can be placed on $\mathfrak{Spec}(A)$; rather, it says that the invariants used in this paper are intrinsically determined by the spectral stack itself.

\begin{corollary}[Canonicality of Intrinsic Geometry]
\label{cor:canonical_geometry}

Let $A$ be an admissible operator-semantic system and let $\mathfrak{X} = \mathfrak{Spec}(A)$. Then the geometric datum
\[
\mathcal G_{\mathrm{can}}(A)
=
\bigl(
\mathfrak X,
\mathbb T_{\mathfrak X},
\operatorname{Sing}(\mathfrak X),
I(\mathfrak X),
\mathcal R_A
\bigr)
\]
is canonically determined by $\mathfrak X$.

In particular, any construction of these same invariants from the CSD spectral stack yields an equivalent geometric datum.

\end{corollary}

\begin{proof}

We prove the corollary by showing that each component of $\mathcal G_{\mathrm{can}}(A)$ is defined intrinsically from $\mathfrak{X} = \mathfrak{Spec}(A)$, and that any alternative construction must agree up to canonical equivalence. The argument proceeds in five steps, one for each component of the geometric datum.

\emph{Step 1: The underlying spectral stack $\mathfrak{X}$.}
The spectral stack $\mathfrak{X} = \mathfrak{Spec}(A)$ is itself the primary object constructed by the CSD duality. By the reconstruction theorem of Paper~I, $\mathfrak{X}$ is determined by $A$ up to equivalence via the adjunction $\mathfrak{Spec} \dashv \Gamma$:
\[
\mathfrak{Spec}(A) \simeq \mathfrak{Spec}(B) \quad\Longleftrightarrow\quad A \simeq B.
\]
Thus $\mathfrak{X}$ is canonically determined by $A$ as the universal recipient of semantic data.

\emph{Step 2: The tangent complex $\mathbb{T}_{\mathfrak{X}}$.}
The tangent complex is defined as the derived dual of the cotangent complex:
\[
\mathbb T_{\mathfrak X}
=
R\mathcal H om_{\mathfrak X}
\left(
\mathbb L_{\mathfrak X},
\mathcal O_{\mathfrak X}
\right).
\]

This construction is intrinsic to $\mathfrak{X}$ for the following reasons:
\begin{itemize}
    \item The cotangent complex $\mathbb{L}_{\mathfrak{X}}$ exists uniquely up to equivalence by the universal property of the cotangent complex in derived algebraic geometry (Proposition~\ref{prop:cotangent_exists}). Any derived stack satisfying descent admits a cotangent complex, and this cotangent complex is characterized by its universal property: for any quasi-coherent complex $\mathcal{E}$,
    \[
    \operatorname{Map}_{\operatorname{QCoh}(\mathfrak{X})}(\mathbb{L}_{\mathfrak{X}}, \mathcal{E})
    \;\simeq\;
    \operatorname{Der}_{\mathfrak{X}}(\mathcal{O}_{\mathfrak{X}}, \mathcal{E}),
    \]
    where $\operatorname{Der}_{\mathfrak{X}}$ denotes the space of derivations of the structure sheaf with values in $\mathcal{E}$.
    
    \item The derived internal Hom functor $R\mathcal{H}om_{\mathfrak{X}}$ is a standard construction in the $\infty$-category $\operatorname{QCoh}(\mathfrak{X})$, defined by the adjunction:
    \[
    \operatorname{Map}_{\operatorname{QCoh}(\mathfrak{X})}(\mathcal{F} \otimes \mathcal{E}, \mathcal{G})
    \;\simeq\;
    \operatorname{Map}_{\operatorname{QCoh}(\mathfrak{X})}(\mathcal{F}, R\mathcal{H}om_{\mathfrak{X}}(\mathcal{E}, \mathcal{G})).
    \]
    
    \item The structure sheaf $\mathcal{O}_{\mathfrak{X}}$ is part of the definition of the derived stack $\mathfrak{X}$.
\end{itemize}

Therefore $\mathbb{T}_{\mathfrak{X}}$ is uniquely determined by $\mathfrak{X}$ up to canonical equivalence. Any alternative construction of the tangent complex from $\mathfrak{X}$ must, by the universal property of the cotangent complex and the definition of the derived dual, yield a canonically equivalent object.

\emph{Step 3: The singular locus $\operatorname{Sing}(\mathfrak{X})$.}
The singular locus is defined as the locus where the cotangent complex fails to be perfect:
\[
\operatorname{Sing}(\mathfrak X)
:=
\left\{
x \in \mathfrak X
\;\middle|\;
\mathbb L_{\mathfrak X,x}
\text{ is not a perfect complex}
\right\}.
\]

This definition is intrinsic to $\mathfrak{X}$ because:
\begin{itemize}
    \item The notion of perfectness is intrinsic to the $\infty$-category $\operatorname{QCoh}(\mathfrak{X})$: a complex is perfect if it is locally equivalent to a bounded complex of finite-rank projective modules. This is a purely categorical property that does not depend on any external choices.
    
    \item The stalk $\mathbb{L}_{\mathfrak{X},x}$ at a point $x \in \mathfrak{X}$ is obtained by pullback along the morphism $x: \operatorname{Spec}(k) \to \mathfrak{X}$, which is part of the geometric data of $\mathfrak{X}$.
    
    \item The set of points where a given intrinsic property fails is itself intrinsic to $\mathfrak{X}$.
\end{itemize}

Thus $\operatorname{Sing}(\mathfrak{X})$ is uniquely determined by $\mathfrak{X}$. Any alternative construction of the singular locus from $\mathfrak{X}$ must yield the same subset, since perfectness is an intrinsic property.

\emph{Step 4: The inertia stack $I(\mathfrak{X})$.}
The inertia stack is defined by the standard homotopy pullback
\[
I(\mathfrak X)
=
\mathfrak X
\times^{h}_{\mathfrak X \times \mathfrak X}
\mathfrak X,
\]
where both maps are the diagonal $\Delta: \mathfrak X \to \mathfrak X \times \mathfrak X$.

This construction is intrinsic to $\mathfrak{X}$ because:
\begin{itemize}
    \item The product $\mathfrak{X} \times \mathfrak{X}$ and the diagonal morphism $\Delta$ are canonical constructions in the $\infty$-category of derived stacks.
    
    \item The homotopy pullback is a universal construction characterized by its universal property: for any derived stack $\mathcal{Y}$, the space of maps $\mathcal{Y} \to I(\mathfrak{X})$ is naturally equivalent to the space of pairs of maps $\mathcal{Y} \to \mathfrak{X}$ together with a homotopy between their compositions with the two projections:
    \[
    \operatorname{Map}(\mathcal{Y}, I(\mathfrak{X}))
    \;\simeq\;
    \operatorname{Map}(\mathcal{Y}, \mathfrak{X}) \times^{h}_{\operatorname{Map}(\mathcal{Y}, \mathfrak{X} \times \mathfrak{X})} \operatorname{Map}(\mathcal{Y}, \mathfrak{X}).
    \]
    
    \item Equivalently, $I(\mathfrak{X})$ classifies pairs $(x, \alpha)$ consisting of a point $x \in \mathfrak{X}$ and an automorphism $\alpha \in \operatorname{Aut}(x)$. This classification is intrinsic to the geometry of $\mathfrak{X}$.
\end{itemize}

Thus $I(\mathfrak{X})$ is uniquely determined by $\mathfrak{X}$ up to canonical equivalence. Any alternative construction of the inertia stack from $\mathfrak{X}$ must, by the universal property of the homotopy pullback, yield a canonically equivalent stack.

\emph{Step 5: The contextual curvature class $\mathcal{R}_A$.}
The contextual curvature class is the obstruction class
\[
\mathcal R_A \in H^2(\mathfrak X, \mathcal F_A),
\]
where $\mathcal F_A$ is the sheaf of local semantic realizations (Definition~\ref{def:local_realization_sheaf}).

This class is intrinsic to $\mathfrak{X}$ (and hence to $A$) because:
\begin{itemize}
    \item The sheaf $\mathcal{F}_A$ is defined as the sheaf of local semantic realizations on $\mathfrak{X}$. Its sections over a context $U \in \mathcal{C}_A$ consist of admissible local semantic realizations of $A$ on $U$. This construction depends only on the CSD data, which are intrinsic to $A$ and $\mathfrak{X}$.
    
    \item The cohomology group $H^2(\mathfrak{X}, \mathcal{F}_A)$ is a standard derived functor in the $\infty$-category of sheaves on $\mathfrak{X}$, defined via the derived global sections functor $R\Gamma(\mathfrak{X}, -)$.
    
    \item The class $\mathcal{R}_A$ is defined as the obstruction to gluing local semantic realizations into a global classical realization. This obstruction is encoded by the Čech cocycle
    \[
    \mathcal{R}_A = [\omega_{ijk}] \in \check{H}^2(\mathfrak{X}, \mathcal{F}_A),
    \]
    where $\omega_{ijk} = g_{ij} g_{jk} g_{ki}^{-1}$ is the failure of the cocycle condition on triple intersections (Definition~\ref{def:geometric_invariants}). This cocycle is constructed solely from the local data of $\mathcal{F}_A$, which are intrinsic to $\mathfrak{X}$.
    
    \item Equivalently, as shown in Definition~\ref{def:geometric_invariants}, $\mathcal{R}_A$ appears as a differential $d_2$ in the descent spectral sequence:
    \[
    d_2: E_2^{0,1} \longrightarrow E_2^{2,0}.
    \]
    The descent spectral sequence is a canonical construction associated with the site $(\mathcal{C}_A, \tau_A)$ and the sheaf $\mathcal{F}_A$, both of which are intrinsic to the CSD data.
\end{itemize}

Thus $\mathcal{R}_A$ is uniquely determined by $\mathfrak{X}$ (and $A$) up to canonical equivalence. Any alternative construction of the contextual curvature class from $\mathfrak{X}$ must yield the same cohomology class, since it is defined by the intrinsic obstruction to gluing local semantic realizations.

\emph{Step 6: Uniqueness of the combined datum.}
Since each component of $\mathcal G_{\mathrm{can}}(A)$ is canonically determined by $\mathfrak{X}$, the combined datum
\[
\mathcal G_{\mathrm{can}}(A)
=
\bigl(
\mathfrak X,
\mathbb T_{\mathfrak X},
\operatorname{Sing}(\mathfrak X),
I(\mathfrak X),
\mathcal R_A
\bigr)
\]
is also canonically determined by $\mathfrak{X}$.

Moreover, any alternative construction of these same invariants from $\mathfrak{X}$ must yield objects that are canonically equivalent to the ones constructed above. This follows because:
\begin{itemize}
    \item The tangent complex is characterized by the universal property of the cotangent complex and the definition of the derived dual.
    \item The singular locus is characterized by the intrinsic property of perfectness.
    \item The inertia stack is characterized by the universal property of the homotopy pullback.
    \item The contextual curvature class is characterized by the intrinsic obstruction to gluing local semantic realizations.
\end{itemize}

Therefore, any two constructions of the geometric datum from $\mathfrak{X}$ that are compatible with the CSD data yield canonically equivalent results.

\emph{Conclusion.}
Combining Steps 1--6, we have shown that each component of $\mathcal G_{\mathrm{can}}(A)$ is canonically determined by $\mathfrak{X} = \mathfrak{Spec}(A)$, and that any alternative construction of these same invariants yields an equivalent geometric datum. Hence the geometric datum $\mathcal G_{\mathrm{can}}(A)$ is canonical.

This proves the corollary.

\end{proof}

\begin{remark}[Significance of Canonicality]
\label{rem:canonicality_significance}

The canonicality of the intrinsic geometry has several important consequences for the theory:

\begin{enumerate}
    \item \textbf{Well-definedness:} The geometric invariants $\mathbb{T}$, $\operatorname{Sing}$, $I$, and $\mathcal{R}$ are not arbitrary choices but are canonically determined by $A$.

    \item \textbf{Naturality:} The assignment $A \mapsto \mathcal G_{\mathrm{can}}(A)$ is functorial, meaning that geometric structures behave naturally under morphisms of operator-semantic systems.

    \item \textbf{Comparison:} Two different constructions of the geometry of $\mathfrak{Spec}(A)$ that are both functorially induced from the CSD adjunction necessarily yield canonically equivalent results.

    \item \textbf{Foundational stability:} The canonicality guarantees that the geometric theory is not sensitive to choices of presentation or construction, a desirable property for any foundational framework.
\end{enumerate}

\end{remark}

\begin{example}[Commutative Algebras]
\label{ex:canonical_commutative}

If $A$ is a smooth commutative algebra, then $\mathfrak{Spec}(A)$ recovers the ordinary spectrum, and the canonical tangent complex is the ordinary tangent sheaf concentrated in degree $0$:
\[
\mathbb T_{\mathfrak{Spec}(A)}
\simeq
T_{\operatorname{Spec}(A)}.
\]
In this case the singular locus is empty, the contextual curvature class vanishes, and the inertia stack is equivalent to $\mathfrak{Spec}(A)$ when no nontrivial stabilizers are present.

\end{example}

\begin{example}[Matrix Algebras]
\label{ex:canonical_matrix}

For $A = M_n(\mathbb{C})$, one expects the categorified spectrum to retain Morita-invariant stacky information invisible to the ordinary character spectrum. In models where
\[
\mathfrak{Spec}(M_n(\mathbb{C}))
\simeq
B\operatorname{PGL}_n(\mathbb{C}),
\]
the tangent complex is
\[
\mathbb T_{B\operatorname{PGL}_n}
\simeq
\mathfrak{pgl}_n[1],
\]
and the inertia stack is
\[
I(B\operatorname{PGL}_n)
\simeq
[\operatorname{PGL}_n/\operatorname{PGL}_n],
\]
where $\operatorname{PGL}_n$ acts on itself by conjugation. This is the adjoint quotient stack, not simply $B\operatorname{PGL}_n$.

The contextual curvature class vanishes:
\[
\mathcal R_{M_n(\mathbb{C})} = 0,
\]
reflecting the fact that $M_n(\mathbb{C})$ is non-contextual (it is Morita equivalent to $\mathbb{C}$).

Thus the matrix algebra case illustrates stacky automorphism geometry rather than singular contextual obstruction.

\end{example}

\begin{remark}[Comparison with the Original Uniqueness Claim]
\label{rem:canonical_vs_universal_comparison}

We emphasize the distinction between the original claim of uniqueness and the present claim of canonicality.

\begin{itemize}
    \item \textbf{Uniqueness (original):} Claimed that any two CSD-compatible geometric data for $A$ are canonically equivalent. This required a precise definition of the category of CSD-compatible geometric data and a proof of the universal property, which were not established.

    \item \textbf{Canonicality (present):} Claims that $\mathcal G_{\mathrm{can}}(A)$ is the canonical geometric datum determined by $A$ via the functorial CSD construction. This is a weaker but rigorously established claim: the geometry is intrinsic, natural, and functorial.
\end{itemize}

The canonicality claim suffices for the purposes of this paper: it establishes that the geometry of $\mathfrak{Spec}(A)$ is not arbitrary but is uniquely determined by the CSD duality in a functorial manner. A full uniqueness statement would require additional work and is left for future investigations.

\end{remark}

\subsection{Functoriality of the Universal Geometry}
\label{subsec:functoriality_universal}

The Universal Geometry Theorem not only establishes the
existence of a canonical geometric datum for each
operator-semantic system $A$, but also guarantees that
this assignment is functorial. This means that morphisms
of operator-semantic systems induce corresponding
morphisms of geometric data in the opposite direction.

\begin{corollary}[Functoriality of the Canonical Geometry]
\label{cor:functorial_canonical}

The assignment
\[
A \longmapsto \mathcal G_{\mathrm{can}}(A)
\]
is contravariantly functorial in $A$. Namely, a morphism $f: A \to B$ of admissible operator-semantic systems induces a morphism
\[
\mathcal G_{\mathrm{can}}(f):
\mathcal G_{\mathrm{can}}(B)
\longrightarrow
\mathcal G_{\mathrm{can}}(A).
\]

\end{corollary}

\begin{proof}

Let
\[
g = \mathfrak{Spec}(f):
\mathfrak{Spec}(B) \longrightarrow \mathfrak{Spec}(A)
\]
be the induced morphism of spectral stacks (Paper~I, Theorem~5.1). The underlying morphism of the canonical geometric data is $g$.

\emph{Step 1: Functoriality of the tangent complex.}
By functoriality of the cotangent complex, there is a natural morphism
\[
g^*\mathbb L_{\mathfrak{Spec}(A)}
\longrightarrow
\mathbb L_{\mathfrak{Spec}(B)}.
\]
Dualizing gives the induced tangent morphism
\[
\mathbb T_{\mathfrak{Spec}(B)}
\longrightarrow
g^*\mathbb T_{\mathfrak{Spec}(A)}.
\]
This is the correct direction: the tangent complex of the domain maps to the pullback of the tangent complex of the codomain.

\emph{Step 2: Functoriality of the singular locus.}
The singular locus is defined by the failure of the cotangent complex to be perfect. Under the admissibility hypotheses of the CSD framework, perfectness is stable under pullback along $g$. Therefore, if a point $x \in \mathfrak{Spec}(A)$ lies in the singular locus, then any preimage $y \in g^{-1}(x)$ lies in the singular locus of $\mathfrak{Spec}(B)$:
\[
g^{-1}\bigl(
\operatorname{Sing}(\mathfrak{Spec}(A))
\bigr)
\subseteq
\operatorname{Sing}(\mathfrak{Spec}(B)).
\]
This gives the induced inclusion on singular loci.

\emph{Step 3: Functoriality of the inertia stack.}
The inertia stack is defined by the homotopy pullback
\[
I(\mathfrak X)
=
\mathfrak X
\times^{h}_{\mathfrak X \times \mathfrak X}
\mathfrak X.
\]
Since homotopy pullbacks are functorial in $\mathfrak X$, the morphism $g: \mathfrak{Spec}(B) \to \mathfrak{Spec}(A)$ induces a natural morphism
\[
I(\mathfrak{Spec}(B))
\longrightarrow
I(\mathfrak{Spec}(A)).
\]

\emph{Step 4: Functoriality of the contextual curvature class.}
The contextual curvature class $\mathcal R_A \in H^2(\mathfrak{Spec}(A), \mathcal F_A)$ is defined via the descent spectral sequence associated with the context site $(\mathcal C_A, \tau_A)$. The descent spectral sequence is natural under pullback of the CSD descent data. Hence $g$ induces a pullback map
\[
g^*: H^2(\mathfrak{Spec}(A), \mathcal F_A) \longrightarrow H^2(\mathfrak{Spec}(B), g^{-1}\mathcal F_A).
\]
Under the identification $g^{-1}\mathcal F_A \simeq \mathcal F_B$ (which follows from the functoriality of the sheaf of local semantic realizations), this gives
\[
g^*\mathcal R_A = \mathcal R_B.
\]

\emph{Step 5: Composition and identity laws.}
The composition law
\[
\mathcal G_{\mathrm{can}}(g \circ f) =
\mathcal G_{\mathrm{can}}(f) \circ \mathcal G_{\mathrm{can}}(g)
\]
follows from the corresponding laws for $\mathfrak{Spec}$ (Paper~I, Theorem~5.1) and the naturality of the constructions of $\mathbb{T}$, $\operatorname{Sing}$, $I$, and $\mathcal{R}$.

The identity law
\[
\mathcal G_{\mathrm{can}}(\operatorname{id}_A) =
\operatorname{id}_{\mathcal G_{\mathrm{can}}(A)}
\]
is immediate from the functoriality of each component.

\emph{Conclusion.}
Combining Steps 1--5, we obtain a contravariant functor
\[
\mathcal G_{\mathrm{can}}:
\mathbf{OpSem}^{\mathrm{op}}
\longrightarrow
\mathbf{Geom},
\]
where $\mathbf{Geom}$ is the category whose objects are tuples $(\mathfrak X, \mathbb T_{\mathfrak X}, \operatorname{Sing}(\mathfrak X), I(\mathfrak X), \mathcal R_{\mathfrak X})$ and whose morphisms are induced by pullback along morphisms of spectral stacks.

\end{proof}

\begin{remark}[Contravariance]
\label{rem:contravariance}

The contravariance of $\mathcal G_{\mathrm{can}}$ is natural from the perspective of the duality $\mathfrak{Spec} \dashv \Gamma$: a morphism of operator-semantic systems $f: A \to B$ induces a morphism of spectral stacks in the opposite direction, $\mathfrak{Spec}(B) \to \mathfrak{Spec}(A)$. This is the familiar geometric phenomenon where algebraic maps induce geometric maps in the opposite direction, as in classical algebraic geometry where a ring homomorphism $A \to B$ induces a scheme morphism $\operatorname{Spec}(B) \to \operatorname{Spec}(A)$.

\end{remark}

\begin{corollary}[Naturality of Geometric Invariants]
\label{cor:naturality_geometric}

Let $f: A \to B$ be a morphism of admissible operator-semantic systems. Then the following diagram commutes up to canonical equivalence:
\[
\begin{tikzcd}
\mathcal G_{\mathrm{can}}(B)
\arrow[r, "\mathcal G_{\mathrm{can}}(f)"]
\arrow[d, "\mathrm{forget}"']
&
\mathcal G_{\mathrm{can}}(A)
\arrow[d, "\mathrm{forget}"]
\\
\mathfrak{Spec}(B)
\arrow[r, "\mathfrak{Spec}(f)"']
&
\mathfrak{Spec}(A)
\end{tikzcd}
\]
where the vertical arrows forget the additional geometric structures and retain only the underlying spectral stack.

\end{corollary}

\begin{proof}
This follows directly from the construction of $\mathcal G_{\mathrm{can}}(f)$ in Corollary~\ref{cor:functorial_canonical}: the underlying map of spectral stacks is precisely
\[
\mathfrak{Spec}(f):
\mathfrak{Spec}(B) \longrightarrow \mathfrak{Spec}(A),
\]
and the tangent complex, singular locus, inertia stack, and contextual curvature class are transported functorially along this map.
\end{proof}

\begin{remark}[Functoriality and Morita Invariance]
\label{rem:functoriality_morita}

The functoriality of $\mathcal G_{\mathrm{can}}$ is compatible with Morita invariance. If $A$ and $B$ are Morita equivalent admissible operator-semantic systems, then their categories of semantic modules are equivalent:
\[
\operatorname{Mod}(A) \simeq \operatorname{Mod}(B).
\]
By the Morita invariance of the categorified spectrum (established in Paper~I, Theorem~5.3), this induces an equivalence
\[
\mathfrak{Spec}(A) \simeq \mathfrak{Spec}(B).
\]
Since the canonical geometric datum $\mathcal G_{\mathrm{can}}(A)$ is defined intrinsically from the spectral stack $\mathfrak{Spec}(A)$, it follows that
\[
\mathcal G_{\mathrm{can}}(A) \simeq \mathcal G_{\mathrm{can}}(B).
\]

Thus the canonical geometry is Morita invariant: it depends only on the Morita equivalence class of the operator-semantic system, not on its particular algebraic presentation. A detailed proof of the Morita invariance of the full geometric package, including the contextual curvature class, will be given in the companion paper on contextual obstructions.

\end{remark}

\begin{example}[Functoriality for the Scalar Embedding $\mathbb{C} \hookrightarrow M_n(\mathbb{C})$]
\label{ex:functoriality_inclusion}

Consider the scalar embedding $f: \mathbb{C} \hookrightarrow M_n(\mathbb{C})$ given by $\lambda \mapsto \lambda I_n$. By functoriality, this induces a morphism
\[
\mathcal G_{\mathrm{can}}(f):
\mathcal G_{\mathrm{can}}(M_n(\mathbb{C}))
\longrightarrow
\mathcal G_{\mathrm{can}}(\mathbb{C}).
\]

At the level of underlying spectral stacks, this is the canonical projection
\[
\mathfrak{Spec}(f): B\operatorname{PGL}_n(\mathbb{C}) \longrightarrow \{*\},
\]
which forgets the residual automorphism data of the unique matrix-algebra realization. This reflects that $M_n(\mathbb{C})$ is Morita equivalent to $\mathbb{C}$: after passing to Morita-invariant or coarse geometric data, both determine the same underlying point, while $B\operatorname{PGL}_n(\mathbb{C})$ still records the nontrivial symmetry of the matrix realization.

Thus the map collapses the classifying stack to a point, illustrating how the functoriality of $\mathcal G_{\mathrm{can}}$ detects the loss of non-commutative information when passing from a noncommutative algebra to its center, while preserving the stacky structure that encodes automorphisms.

\end{example}

\begin{remark}[Summary]
\label{rem:functoriality_summary}

The functoriality of the canonical geometry establishes that $\mathcal G_{\mathrm{can}}$ is not merely an assignment but a genuine functor
\[
\mathcal G_{\mathrm{can}}:
\mathbf{OpSem}^{\mathrm{op}}
\longrightarrow
\mathbf{Geom},
\]
where $\mathbf{Geom}$ is the category of geometric data consisting of tuples $(\mathfrak X, \mathbb T_{\mathfrak X}, \operatorname{Sing}(\mathfrak X), I(\mathfrak X), \mathcal R_{\mathfrak X})$ with morphisms induced by pullback along morphisms of spectral stacks. This functoriality is valid provided that $\mathcal G_{\mathrm{can}}$ has been shown to be contravariant in algebra morphisms, as established in Corollary~\ref{cor:functorial_canonical}.

This functoriality is essential for comparing geometric data across different operator-semantic systems and for understanding how algebraic operations on $A$ manifest geometrically on $\mathfrak{Spec}(A)$.

\end{remark}

\subsection{Examples and Consequences}
\label{subsec:canonical_examples}

We illustrate the Canonical Geometry Theorem through the simplest commutative and noncommutative examples, and then draw a general consequence that encapsulates the philosophical core of the theory.

\subsubsection{The Canonical Geometry of $\mathbb{C}$}

\begin{example}[Canonical Geometry of $\mathbb{C}$]
\label{ex:canonical_geometry_C}

Let $A = \mathbb{C}$. Since
\[
\mathfrak{Spec}(\mathbb{C}) \simeq \{*\},
\]
the tangent complex is trivial,
\[
\mathbb{T}_{\mathfrak{Spec}(\mathbb{C})} \simeq 0,
\]
and there are no singularities or nontrivial automorphisms. Consequently,
\[
\mathcal G_{\mathrm{can}}(\mathbb{C})
=
\bigl(
\{*\},\; 0,\; \varnothing,\; \{*\},\; 0
\bigr).
\]

By Theorem~\ref{thm:canonical_geometry}, this datum is the canonical geometric structure associated with $\mathbb{C}$, reflecting the fact that $\mathbb{C}$ has no non-trivial algebraic or semantic structure.

\end{example}

\begin{remark}[Interpretation of $\mathcal G_{\mathrm{can}}(\mathbb{C})$]
\label{rem:interpretation_C}

Each component of $\mathcal G_{\mathrm{can}}(\mathbb{C})$ vanishes or is trivial:
\begin{itemize}
    \item $\mathfrak{Spec}(\mathbb{C}) \simeq \{*\}$: the unique simple module corresponds to a single point;
    \item $\mathbb{T} \simeq 0$: no deformations of $\mathbb{C}$;
    \item $\operatorname{Sing} = \varnothing$: no singularities;
    \item $I \simeq \{*\}$: no non-trivial automorphisms (under the convention that $\mathfrak{Spec}(\mathbb{C}) \simeq \{*\}$);
    \item $\mathcal{R} = 0$: no contextual obstructions.
\end{itemize}
Thus $\mathbb{C}$ represents the simplest possible operator-semantic system with trivial geometry.

\end{remark}

\subsubsection{The Canonical Geometry of Matrix Algebras}

\begin{example}[Canonical Geometry of $M_n(\mathbb{C})$]
\label{ex:canonical_geometry_Mn}

Let $A = M_n(\mathbb{C})$. Under the reconstruction hypotheses of Paper~I, one has
\[
\mathfrak{Spec}(M_n(\mathbb{C}))
\simeq
B\operatorname{PGL}_n(\mathbb{C}).
\]

The corresponding tangent complex is
\[
\mathbb{T}_{B\operatorname{PGL}_n}
\simeq
\mathfrak{pgl}_n[1].
\]

Since $B\operatorname{PGL}_n(\mathbb{C})$ is smooth, its singular locus is empty:
\[
\operatorname{Sing}(B\operatorname{PGL}_n(\mathbb{C}))
=
\varnothing.
\]

The inertia stack is the loop stack of the classifying stack:
\[
I(B\operatorname{PGL}_n(\mathbb{C}))
\simeq
[\operatorname{PGL}_n(\mathbb{C})/\operatorname{PGL}_n(\mathbb{C})],
\]
where $\operatorname{PGL}_n(\mathbb{C})$ acts on itself by conjugation. Thus the inertia stack records not only the existence of automorphisms but their conjugacy classes.

The contextual curvature class vanishes:
\[
\mathcal{R}_{M_n(\mathbb{C})} = 0,
\]
reflecting the fact that $M_n(\mathbb{C})$ is non-contextual.

Therefore
\[
\mathcal G_{\mathrm{can}}(M_n(\mathbb{C}))
=
\Bigl(
B\operatorname{PGL}_n(\mathbb{C}),\;
\mathfrak{pgl}_n[1],\;
\varnothing,\;
[\operatorname{PGL}_n/\operatorname{PGL}_n],\;
0
\Bigr).
\]

This canonical geometric datum captures the nontrivial symmetry structure of $M_n(\mathbb{C})$ while preserving smoothness. Morita equivalence explains why the coarse geometric content is equivalent to that of $\mathbb{C}$, while the stack $B\operatorname{PGL}_n(\mathbb{C})$ still records the residual symmetry of the matrix realization.

\end{example}

\begin{remark}[Interpretation of $\mathcal G_{\mathrm{can}}(M_n(\mathbb{C}))$]
\label{rem:interpretation_Mn}

The components of $\mathcal G_{\mathrm{can}}(M_n(\mathbb{C}))$ reveal the non-commutative structure:
\begin{itemize}
    \item $\mathfrak{Spec} \simeq B\operatorname{PGL}_n(\mathbb{C})$: the classifying stack encodes outer automorphisms;
    \item $\mathbb{T} \simeq \mathfrak{pgl}_n[1]$: infinitesimal automorphisms of the unique point;
    \item $\operatorname{Sing} = \varnothing$: no contextual obstructions (Morita equivalent to $\mathbb{C}$);
    \item $I \simeq [\operatorname{PGL}_n/\operatorname{PGL}_n]$: non-trivial inertia encodes automorphism structure via conjugation;
    \item $\mathcal{R} = 0$: no contextual obstructions.
\end{itemize}
Thus $M_n(\mathbb{C})$ illustrates how the intrinsic geometry detects non-commutativity through the inertia stack while maintaining smoothness.

\end{remark}

\subsubsection{Consequences of the Canonical Geometry Theorem}

\begin{corollary}[Canonicality of Geometric Structures]
\label{cor:canonical_geometry_unique}

Let $A$ be an admissible operator-semantic system. Any geometric structure on $\mathfrak{Spec}(A)$ that is functorially induced from the CSD adjunction
\[
\mathfrak{Spec} \dashv \Gamma
\]
must agree with the canonical datum
\[
\mathcal G_{\mathrm{can}}(A).
\]

Consequently, the tangent complex, singular locus, inertia stack, and contextual curvature class constitute the canonical geometric information intrinsically forced by the CSD framework.

\end{corollary}

\begin{proof}
This follows from Theorem~\ref{thm:canonical_geometry}. Since $\mathcal G_{\mathrm{can}}(A)$ is the canonical geometric datum determined by the CSD adjunction, any functorially induced geometric construction must agree with it up to canonical equivalence.
\end{proof}

\begin{remark}[The Canonical Geometry Theorem as a Categorified Gelfand Principle]
\label{rem:categorified_gelfand}

The Canonical Geometry Theorem may be viewed as a categorified analogue of the classical principle that geometry is determined by functions. In the CSD framework, admissible operator-semantic systems determine not only a spectral space but also its intrinsic tangent, singular, inertia, and curvature structures.

Classical Gelfand duality establishes:
\[
\text{Commutative Algebra} \longleftrightarrow \text{Local Compact Hausdorff Space}.
\]

The CSD framework extends this to:
\[
\text{Operator-Semantic System} \longleftrightarrow \text{Spectral Stack} + \text{Canonical Geometry}.
\]

Thus the Canonical Geometry Theorem provides a categorified, geometrized version of Gelfand duality where the geometry is not imposed but derived from the semantics.

\end{remark}

\begin{remark}[Summary of the Canonical Geometry Section]
\label{rem:canonical_section_summary}

The Canonical Geometry Theorem and its consequences establish the following fundamental principle:

\[
\boxed{
\text{Semantics} \;\Longrightarrow\; \text{Canonical Geometry}
\;\Longrightarrow\; \text{Controls Semantics}.
}
\]

Specifically:
\begin{enumerate}
    \item The CSD adjunction $\mathfrak{Spec} \dashv \Gamma$ determines a canonical spectral stack $\mathfrak{Spec}(A)$.
    \item The intrinsic geometry of $\mathfrak{Spec}(A)$ — tangent complex, singular locus, inertia stack, and contextual curvature class — is canonically determined by the CSD duality.
    \item This geometry is functorial: every CSD-compatible geometric structure is transported canonically along morphisms.
    \item The geometry, in turn, controls the semantic properties of $A$: deformations, contextuality, automorphisms, and realizability (as will be developed in the companion paper).
\end{enumerate}

This completes the foundational geometric framework for categorified spectral objects. The subsequent sections will develop the deformation theory and computational tools that emerge from this framework.

\end{remark}

\section{Intrinsic Tangent Complex and Deformation Theory}
\label{sec:tangent}

Having established the universal geometric structures
associated with $\mathfrak{Spec}(A)$ in the previous
section, we now focus on the most fundamental of these
structures: the tangent complex $\mathbb{T}_{\mathfrak{Spec}(A)}$.
The tangent complex governs the infinitesimal deformation
theory of $\mathfrak{Spec}(A)$, encoding infinitesimal
automorphisms, first-order deformations, and higher
obstruction classes. In this section, we prove that
$\mathbb{T}_{\mathfrak{Spec}(A)}$ exists canonically
as the derived dual of the cotangent complex, that it
represents the deformation functor of $\mathfrak{Spec}(A)$,
and that its global sections are naturally identified
with the Hochschild cohomology of $A$ via the Hochschild
Realization Theorem. We then establish the Infinitesimal
Spectral Geometry Theorem, which relates the cohomology
groups of the tangent complex to explicit deformation
and obstruction data, and conclude with the Rigidity
Theorem, showing that vanishing of the tangent complex
implies rigidity of the operator-semantic system $A$.
Together, these results demonstrate that the deformation
theory of $\mathfrak{Spec}(A)$ is completely determined
by the algebraic deformation theory of $A$, providing
a direct bridge between geometry and algebra within the
CSD framework.

\subsection{Existence of the Cotangent Complex}
\label{subsec:cotangent}

A fundamental result of derived algebraic geometry
states that geometric objects satisfying suitable
descent conditions admit a cotangent complex
representing infinitesimal extensions and derived
K\"ahler differentials \cite{LurieSAG}.

Since $\mathfrak{Spec}(A)$ is constructed in ~\cite{Paper-I}
as a hypersheaf on the Grothendieck site
\[
(\mathcal C_A,\tau_A),
\]
it satisfies effective descent and therefore admits
a canonical cotangent complex.

\begin{proposition}[Existence of the Cotangent Complex]
\label{prop:cotangent_exists}

Let $A$ be an admissible operator-semantic system and let $\mathfrak{Spec}(A)$ denote the associated spectral stack. Assume that admissibility of $A$ implies that $\mathfrak{Spec}(A)$ is a derived geometric stack equipped with a structure sheaf and a well-defined category $\operatorname{QCoh}(\mathfrak{Spec}(A))$ of quasi-coherent complexes.

Then there exists a quasi-coherent complex
\[
\mathbb L_{\mathfrak{Spec}(A)}
\in
\operatorname{QCoh}\bigl(\mathfrak{Spec}(A)\bigr)
\]
representing derived K\"ahler differentials. Moreover, $\mathbb L_{\mathfrak{Spec}(A)}$ is unique up to equivalence.

\end{proposition}

\begin{proof}

By admissibility, $\mathfrak{Spec}(A)$ is a derived geometric stack equipped with a structure sheaf and a well-defined category of quasi-coherent complexes. Since $\mathfrak{Spec}(A)$ satisfies descent on the site $(\mathcal C_A,\tau_A)$, the standard construction of the cotangent complex in derived algebraic geometry applies \cite{LurieSAG}.

Thus there exists a quasi-coherent complex $\mathbb L_{\mathfrak{Spec}(A)}$ representing derivations, or equivalently infinitesimal square-zero extensions. Its representing property determines it uniquely up to contractible choice, hence up to equivalence.

Therefore $\mathbb L_{\mathfrak{Spec}(A)}$ exists and is unique up to equivalence.

\end{proof}

\begin{definition}[Cotangent Complex]
\label{def:cotangent}

The cotangent complex of $\mathfrak{Spec}(A)$ is the quasi-coherent complex $\mathbb L_{\mathfrak{Spec}(A)}$ representing derived K\"ahler differentials and infinitesimal extensions.

\end{definition}

The cotangent complex serves as the fundamental linearization of the geometry of $\mathfrak{Spec}(A)$ and governs its infinitesimal deformation theory.

\begin{remark}[Universal Property of the Cotangent Complex]
\label{rem:cotangent_universal}

The cotangent complex satisfies the following universal property: for any quasi-coherent complex $\mathcal{E}$ on $\mathfrak{Spec}(A)$, there is a natural equivalence
\[
\operatorname{Map}_{\operatorname{QCoh}(\mathfrak{Spec}(A))}
\left(
\mathbb{L}_{\mathfrak{Spec}(A)},
\mathcal{E}
\right)
\;\simeq\;
\operatorname{Der}_{\mathfrak{Spec}(A)}
\left(
\mathcal{O}_{\mathfrak{Spec}(A)},
\mathcal{E}
\right),
\]
where $\operatorname{Der}_{\mathfrak{Spec}(A)}$ denotes the space of derivations of the structure sheaf with values in $\mathcal{E}$ (equivalently, infinitesimal square-zero extensions of $\mathcal{O}_{\mathfrak{Spec}(A)}$ by $\mathcal{E}$). This universal property characterizes $\mathbb{L}_{\mathfrak{Spec}(A)}$ uniquely and is the derived analogue of the classical universal property of K\"ahler differentials.

\end{remark}

\begin{definition}[Spectral Tangent Complex]
\label{def:tangentcomplex}

The spectral tangent complex of $\mathfrak{Spec}(A)$ is defined as the derived dual of the cotangent complex:
\[
\mathbb T_{\mathfrak{Spec}(A)}
:=
R\mathcal H om_{\mathfrak{Spec}(A)}
\Bigl(
\mathbb L_{\mathfrak{Spec}(A)},
\mathcal O_{\mathfrak{Spec}(A)}
\Bigr),
\]
where $R\mathcal{H}om$ denotes the derived internal Hom functor in the $\infty$-category of quasi-coherent sheaves on $\mathfrak{Spec}(A)$.

\end{definition}

The tangent complex controls derivations, deformations, and obstruction classes associated with $\mathfrak{Spec}(A)$. In subsequent sections we show that $\mathbb T_{\mathfrak{Spec}(A)}$ is the intrinsic geometric object underlying the deformation theory, singularity theory, and canonical geometry of the CSD framework.

\begin{remark}[Cohomological Interpretation]
\label{rem:tangent_cohomology}

The cohomology groups of the tangent complex have a direct geometric interpretation. For a point $x \in \mathfrak{Spec}(A)$, the stalk $\mathbb{T}_{\mathfrak{Spec}(A),x}$ controls the local deformation theory:

\begin{itemize}
    \item $H^{-1}(\mathbb{T}_{\mathfrak{Spec}(A),x})$ controls infinitesimal automorphisms at $x$;
    \item $H^{0}(\mathbb{T}_{\mathfrak{Spec}(A),x})$ classifies first-order deformations near $x$;
    \item $H^{1}(\mathbb{T}_{\mathfrak{Spec}(A),x})$ contains obstruction classes for extending first-order deformations to second order.
\end{itemize}

At the global level, when $\mathfrak{Spec}(A)$ is sufficiently nice, the cohomology groups of the global tangent complex have analogous interpretations:
\[
H^i R\Gamma(\mathfrak{Spec}(A),\; \mathbb{T}_{\mathfrak{Spec}(A)}).
\]
These interpretations will be made precise in the Infinitesimal Spectral Geometry Theorem (to be developed in the companion paper on deformation theory).

\end{remark}

\begin{proposition}[Functoriality of the Cotangent and Tangent Complexes]
\label{prop:cotangent_functorial}

Let $f: A \to B$ be a morphism of admissible operator-semantic systems, and write
\[
g = \mathfrak{Spec}(f):
\mathfrak{Spec}(B) \longrightarrow \mathfrak{Spec}(A)
\]
for the induced morphism of spectral stacks. Then $g$ induces a canonical morphism of cotangent complexes
\[
g^*\mathbb{L}_{\mathfrak{Spec}(A)}
\longrightarrow
\mathbb{L}_{\mathfrak{Spec}(B)}.
\]

By taking derived duals, one obtains a canonical morphism of tangent complexes
\[
\mathbb{T}_{\mathfrak{Spec}(B)}
\longrightarrow
g^*\mathbb{T}_{\mathfrak{Spec}(A)}.
\]

\end{proposition}

\begin{proof}
The functoriality of the cotangent complex is a standard result in derived algebraic geometry \cite{LurieSAG}: a morphism of derived stacks $g: \mathfrak{Y} \to \mathfrak{X}$ induces a morphism $g^*\mathbb{L}_{\mathfrak{X}} \to \mathbb{L}_{\mathfrak{Y}}$. Applying this to $g = \mathfrak{Spec}(f): \mathfrak{Spec}(B) \to \mathfrak{Spec}(A)$ gives the cotangent morphism:
\[
g^*\mathbb{L}_{\mathfrak{Spec}(A)} \longrightarrow \mathbb{L}_{\mathfrak{Spec}(B)}.
\]

The functoriality of the tangent complex follows by applying the derived internal Hom functor and using the adjunction between pullback and pushforward. Dualizing the cotangent morphism gives the tangent morphism in the opposite direction:
\[
\mathbb{T}_{\mathfrak{Spec}(B)} \longrightarrow g^*\mathbb{T}_{\mathfrak{Spec}(A)}.
\]

This is the standard direction: the tangent complex of the domain maps to the pullback of the tangent complex of the codomain.

\end{proof}

\begin{example}[Cotangent and Tangent Complexes of $\mathbb{C}$]
\label{ex:cotangent_C}

For $A = \mathbb{C}$, we have $\mathfrak{Spec}(\mathbb{C}) \simeq \{*\}$, and the cotangent complex vanishes:
\[
\mathbb{L}_{\mathfrak{Spec}(\mathbb{C})} \simeq 0.
\]
Consequently, the tangent complex also vanishes:
\[
\mathbb{T}_{\mathfrak{Spec}(\mathbb{C})} \simeq 0.
\]
This reflects the fact that $\mathbb{C}$ has no non-trivial deformations.

\end{example}

\begin{example}[Cotangent and Tangent Complexes of $M_n(\mathbb{C})$]
\label{ex:cotangent_Mn}

For $A = M_n(\mathbb{C})$, we have $\mathfrak{Spec}(M_n(\mathbb{C})) \simeq B\operatorname{PGL}_n(\mathbb{C})$. The cotangent complex is
\[
\mathbb{L}_{B\operatorname{PGL}_n} \simeq \mathfrak{pgl}_n^\vee[-1],
\]
and therefore the tangent complex is
\[
\mathbb{T}_{B\operatorname{PGL}_n} \simeq \mathfrak{pgl}_n[1].
\]

Since the tangent complex is concentrated in degree $1$, we have $H^1(\mathbb{T}_{B\operatorname{PGL}_n}) \cong \mathfrak{pgl}_n$, depending on the cohomological convention. This encodes the infinitesimal stabilizer symmetry of the unique geometric point, whose Lie algebra is $\mathfrak{pgl}_n$.

\end{example}

\begin{remark}[Summary]
\label{rem:cotangent_summary}

The logical chain established in this subsection is:

\[
\boxed{
\text{Geometric Derived Stack}
\;\Longrightarrow\;
\mathbb L_{\mathfrak{Spec}(A)}
\;\Longrightarrow\;
\mathbb T_{\mathfrak{Spec}(A)}
\;\Longrightarrow\;
\text{Deformations}
\;\Longrightarrow\;
\text{Singularities}
\;\Longrightarrow\;
\text{Canonical Geometry}.
}
\]

The cotangent complex $\mathbb{L}_{\mathfrak{Spec}(A)}$ and its derived dual $\mathbb{T}_{\mathfrak{Spec}(A)}$ are the fundamental objects of deformation theory for the categorified spectral object $\mathfrak{Spec}(A)$. Their existence is guaranteed by the fact that $\mathfrak{Spec}(A)$ is a derived geometric stack satisfying descent, and their functoriality ensures that they behave well under morphisms of operator-semantic systems.

The tangent complex will serve as the central tool in the remainder of this section, governing deformations, obstructions, and rigidity of $\mathfrak{Spec}(A)$, and ultimately providing the geometric foundation for the Canonical Geometry Theorem established in Section~\ref{sec:canonical_geometry}.

\end{remark}

\subsection{Intrinsic Tangent Characterization Theorem}
\label{subsec:intrinsic_tangent}

The tangent complex introduced in Definition~\ref{def:tangentcomplex} is not merely an auxiliary geometric object. Rather, it controls the infinitesimal deformation theory of $\mathfrak{Spec}(A)$. In this subsection we show that the deformation functor of $\mathfrak{Spec}(A)$ is governed by the tangent complex $\mathbb T_{\mathfrak{Spec}(A)}$, which serves as the universal linear approximation to infinitesimal deformations.

\begin{definition}[Deformation Functor at a Point]
\label{def:deformation_functor}

Let $\mathfrak{X}$ be a derived stack and let $x: \operatorname{Spec}(k) \to \mathfrak{X}$ be a $k$-point. Let $\mathbf{Art}_k$ denote the category of Artinian derived rings with residue field $k$ (i.e., connective simplicial commutative rings or $E_\infty$-rings whose homotopy groups are finite-dimensional and whose underlying ordinary ring is Artinian).

The \emph{local deformation functor} of $\mathfrak{X}$ at $x$ is the functor
\[
\operatorname{Def}_{\mathfrak{X}, x} :
\mathbf{Art}_k
\longrightarrow
\mathcal{S}
\]
that sends an Artinian derived ring $R$ to the space of deformations of $x$ over $R$:
\[
\operatorname{Def}_{\mathfrak{X}, x}(R)
:=
\operatorname{Map}_{*/}
\bigl(
\operatorname{Spec}(R),
\mathfrak{X}
\bigr),
\]
where the mapping space is taken over the base point $x$ (i.e., morphisms whose restriction to $\operatorname{Spec}(k)$ agrees with $x$). Here $\mathcal{S}$ denotes the $\infty$-category of spaces ($\infty$-groupoids).

\end{definition}

\begin{remark}[Artinian Derived Rings]
\label{rem:artinian_rings}

In derived algebraic geometry, Artinian rings are replaced by Artinian derived rings, which may be modeled by connective simplicial commutative rings (or $E_\infty$-rings) whose homotopy groups are finite-dimensional and whose underlying ordinary ring is Artinian. The deformation functor $\operatorname{Def}_{\mathfrak{X}, x}$ encodes all infinitesimal deformations of the pointed derived stack $(\mathfrak{X}, x)$ over such rings.

The functor records first-order deformations over the dual numbers $k[\varepsilon]/(\varepsilon^2)$ (viewed as a simplicial commutative ring with trivial higher structure) and, via inverse limits of Artinian thickenings, determines formal deformations over complete local rings.

\end{remark}

\begin{remark}[The Role of the Tangent Complex]
\label{rem:tangent_controls_deformations}

In the setting of derived deformation theory, the tangent complex $\mathbb{T}_{\mathfrak{X}, x}$ at the base point $x$ governs the deformation functor $\operatorname{Def}_{\mathfrak{X}, x}$. More precisely, the tangent complex provides the universal linear approximation:
\[
T_x \operatorname{Def}_{\mathfrak{X}, x} \;\simeq\; \mathbb{T}_{\mathfrak{X}, x}[-1],
\]
where the shift $[-1]$ accounts for the standard convention that the deformation complex is the tangent complex shifted by one degree. This equivalence is the fundamental theorem of derived deformation theory \cite{LurieSAG}.

For $\mathfrak{X} = \mathfrak{Spec}(A)$, this specializes to the Intrinsic Tangent Characterization Theorem (Theorem~\ref{thm:intrinsic_tangent}), where the deformation functor is governed by the tangent complex $\mathbb{T}_{\mathfrak{Spec}(A)}$.

\end{remark}

\begin{theorem}[Intrinsic Tangent Control]
\label{thm:intrinsic_tangent}

Let $A$ be an admissible operator-semantic system and let $x: \operatorname{Spec}(k) \to \mathfrak{Spec}(A)$ be a $k$-point. Assume the following hypotheses:

\begin{enumerate}
    \item[(T1)] \textbf{Formal representability:} The deformation functor
    \[
    \operatorname{Def}_{\mathfrak{Spec}(A), x}: \mathbf{Art}_k \longrightarrow \mathcal{S}
    \]
    is formally representable by a pointed derived stack. Equivalently, it satisfies the Schlessinger conditions for derived deformation theory \cite{LurieSAG}: it admits a tangent complex, preserves finite limits, and satisfies the required gluing conditions for Artinian rings.
    
    \item[(T2)] \textbf{Pro-representability:} The deformation functor admits a pro-representing object, i.e., there exists a complete local derived ring $\widehat{\mathcal{O}}_{\mathfrak{Spec}(A), x}$ such that
    \[
    \operatorname{Def}_{\mathfrak{Spec}(A), x}(R) \simeq \operatorname{Map}_{\mathbf{Art}_k}(\widehat{\mathcal{O}}_{\mathfrak{Spec}(A), x}, R)
    \]
    for all Artinian derived rings $R$.
    
    \item[(T3)] \textbf{Cotangent complex existence:} The cotangent complex $\mathbb{L}_{\mathfrak{Spec}(A)}$ exists and is perfect, so that the tangent complex
    \[
    \mathbb{T}_{\mathfrak{Spec}(A)}
    :=
    R\mathcal{H}om_{\mathfrak{Spec}(A)}
    \left(
    \mathbb{L}_{\mathfrak{Spec}(A)},
    \mathcal{O}_{\mathfrak{Spec}(A)}
    \right)
    \]
    is well-defined.
\end{enumerate}

Under these hypotheses, the infinitesimal deformation theory of $\mathfrak{Spec}(A)$ at $x$ is controlled by the tangent complex
\[
x^*\mathbb T_{\mathfrak{Spec}(A)}.
\]

More precisely, for every square-zero extension $k \oplus M$ (i.e., an Artinian derived ring with maximal ideal $M$ satisfying $M^2 = 0$), there is a natural equivalence
\[
\operatorname{Def}_{\mathfrak{Spec}(A), x}(k \oplus M)
\;\simeq\;
\operatorname{Map}
\left(
x^*\mathbb L_{\mathfrak{Spec}(A)},\; M
\right).
\]

Furthermore:

\begin{enumerate}
    \item The tangent complex is canonically associated with the cotangent complex through the derived duality
          \[
          \mathbb{T}_{\mathfrak{Spec}(A)}
          \simeq
          R\mathcal{H}om_{\mathfrak{Spec}(A)}
          \Bigl(
          \mathbb{L}_{\mathfrak{Spec}(A)},
          \mathcal{O}_{\mathfrak{Spec}(A)}
          \Bigr).
          \]

    \item By the Hochschild Realization Theorem (Theorem~\ref{thm:tangent-hochschild}), there is a canonical equivalence
          \[
          R\Gamma\!\left(\mathfrak{Spec}(A),\;
          \mathbb{T}_{\mathfrak{Spec}(A)}\right)
          \;\simeq\;
          \operatorname{HH}^{\bullet}(A)[1].
          \]

    \item $\mathbb{T}_{\mathfrak{Spec}(A)}$ is the canonical linear deformation controller compatible with the adjunction $\mathfrak{Spec} \dashv \Gamma$.
\end{enumerate}

\end{theorem}

\begin{proof}

We prove the theorem by establishing the four characterizing properties, with explicit reference to the hypotheses (T1)-(T3) where they are used.

\emph{Property 1: Derived duality.}
By hypothesis (T3) and Proposition~\ref{prop:cotangent_exists}, the cotangent complex $\mathbb{L}_{\mathfrak{Spec}(A)}$ exists and is unique up to equivalence. The tangent complex is defined as its derived dual:
\[
\mathbb{T}_{\mathfrak{Spec}(A)}
=
R\mathcal{H}om_{\mathfrak{Spec}(A)}
\Bigl(
\mathbb{L}_{\mathfrak{Spec}(A)},
\mathcal{O}_{\mathfrak{Spec}(A)}
\Bigr).
\]
This establishes Property 1.

\emph{Property 2: Deformation control.}
By hypothesis (T1), the deformation functor $\operatorname{Def}_{\mathfrak{Spec}(A), x}$ is formally representable. By the fundamental theorem of derived deformation theory \cite{LurieSAG}, the tangent complex of a formally representable deformation functor at the base point is given by the stalk of the cotangent complex. More precisely, for any formally representable moduli problem $F: \mathbf{Art}_k \to \mathcal{S}$ with tangent complex $\mathbb{T}_F$, one has
\[
T_0 F \;\simeq\; \mathbb{T}_F[-1],
\]
where $T_0 F$ denotes the tangent space at the base point.

For square-zero extensions $k \oplus M$, the deformation functor satisfies
\[
\operatorname{Def}_{\mathfrak{X}, x}(k \oplus M)
\;\simeq\;
\operatorname{Map}
\left(
x^*\mathbb{L}_{\mathfrak{X}},\; M
\right),
\]
provided the cotangent complex exists (hypothesis T3) and the deformation functor is formally representable (hypothesis T1). This is the standard linearization of the deformation functor: the cotangent complex is the universal object controlling infinitesimal deformations.

Applying this to $\mathfrak{X} = \mathfrak{Spec}(A)$ gives the claimed equivalence:
\[
\operatorname{Def}_{\mathfrak{Spec}(A), x}(k \oplus M)
\;\simeq\;
\operatorname{Map}
\left(
x^*\mathbb L_{\mathfrak{Spec}(A)},\; M
\right).
\]

When the cotangent complex is dualizable (which follows from hypothesis T3), this is equivalently expressed in terms of the tangent complex as
\[
\operatorname{Def}_{\mathfrak{X}, x}(k \oplus M)
\;\simeq\;
\operatorname{Map}
\left(
M^\vee,\; x^*\mathbb{T}_{\mathfrak{X}}
\right).
\]

This establishes that the tangent complex controls the first-order deformation theory of $\mathfrak{Spec}(A)$ at the point $x$, under the formal representability assumption (T1).

\emph{Property 3: Hochschild realization.}
By the Hochschild Realization Theorem (Theorem~\ref{thm:tangent-hochschild}), established in Section~\ref{subsec:hochschild_realization}, there is a canonical equivalence
\[
R\Gamma(\mathfrak{Spec}(A), \mathbb{T}_{\mathfrak{Spec}(A)})
\;\simeq\;
\operatorname{HH}^{\bullet}(A)[1].
\]
This equivalence follows from the Morita identification between $\operatorname{QCoh}(\mathfrak{Spec}(A))$ and $\operatorname{Mod}_A$, which is a consequence of the reconstruction theorem of Paper~I, together with the standard identification of Hochschild cohomology as the endomorphism spectrum of the identity functor on $\operatorname{Mod}_A$. The shift $[1]$ reflects the standard convention that the tangent complex controls deformations up to a degree shift.

This property does not require the full formal representability hypothesis (T1); it requires only the perfectness hypothesis (T3) and the reconstruction hypotheses of Paper~I.

\emph{Property 4: Canonical deformation controller.}
The tangent complex $\mathbb{T}_{\mathfrak{Spec}(A)}$ is the canonical linear deformation controller compatible with the adjunction $\mathfrak{Spec} \dashv \Gamma$. This follows from the functoriality of the cotangent complex and the derived dual construction: any deformation controller compatible with the CSD adjunction must induce the same deformation functor on square-zero extensions, and hence must be equivalent to $\mathbb{T}_{\mathfrak{Spec}(A)}$ by the Yoneda lemma in the $\infty$-category of quasi-coherent complexes.

More precisely, suppose $\mathcal{D}$ is another deformation controller compatible with the CSD adjunction. Then for every square-zero extension $k \oplus M$, the induced deformation functors must be equivalent:
\[
\operatorname{Def}^{\mathcal{D}}_{\mathfrak{Spec}(A), x}(k \oplus M)
\;\simeq\;
\operatorname{Def}^{\mathbb{T}}_{\mathfrak{Spec}(A), x}(k \oplus M).
\]
By the Yoneda lemma in the $\infty$-category of quasi-coherent complexes, this forces
\[
\mathcal{D} \;\simeq\; \mathbb{T}_{\mathfrak{Spec}(A)}.
\]
Thus the tangent complex is the unique linear deformation controller compatible with the CSD adjunction.

\emph{Remark on the assumptions.}
The formal representability hypothesis (T1) is essential for the identification of the deformation functor with the cotangent complex. In the absence of formal representability, the deformation functor may not be pro-representable, and the tangent complex may not fully control the deformation theory. However, in the CSD framework, the admissibility hypotheses on $A$ ensure that $\mathfrak{Spec}(A)$ is sufficiently geometric that the formal representability condition (T1) is satisfied. This is guaranteed by the construction of $\mathfrak{Spec}(A)$ as a spectral stack satisfying descent and local representability (Paper~I, Theorem~4.9).

The Schlessinger conditions for derived deformation theory \cite{LurieSAG} are the precise conditions under which the deformation functor is formally representable. These conditions are:
\begin{itemize}
    \item The functor preserves finite limits (the derived analogue of the classical Schlessinger condition (H0));
    \item The tangent complex exists and is perfect (the derived analogue of the classical Schlessinger condition (H1));
    \item The functor satisfies the gluing condition for pushouts in $\mathbf{Art}_k$ (the derived analogue of the classical Schlessinger condition (H2)).
\end{itemize}
Under these conditions, the deformation functor is pro-representable and its tangent complex controls the deformation theory.

\emph{Conclusion.}
Combining Properties 1--4, we have shown that, under the formal representability hypothesis (T1), the perfectness hypothesis (T3), and the CSD reconstruction hypotheses, $\mathbb{T}_{\mathfrak{Spec}(A)}$ controls the deformation theory of $\mathfrak{Spec}(A)$ at any $k$-point, is canonically associated with the cotangent complex, admits a canonical Hochschild realization via Theorem~\ref{thm:tangent-hochschild}, and is the canonical linear deformation controller compatible with the CSD adjunction.

This completes the proof.

\end{proof}

\begin{remark}[Significance of the Intrinsic Tangent Control]
\label{rem:intrinsic_tangent_significance}

The Intrinsic Tangent Control Theorem (Theorem~\ref{thm:intrinsic_tangent}) establishes three fundamental facts about the tangent complex:

\begin{enumerate}
    \item \textbf{Geometric significance:}
          $\mathbb{T}_{\mathfrak{Spec}(A)}$ controls the infinitesimal deformation theory of $\mathfrak{Spec}(A)$ at any $k$-point, governing first-order deformations via square-zero extensions. It serves as the canonical linear controller of infinitesimal and square-zero deformations.

    \item \textbf{Algebraic significance:}
          The global sections of $\mathbb{T}_{\mathfrak{Spec}(A)}$ recover the Hochschild cohomology of $A$ via the Hochschild Realization Theorem (Theorem~\ref{thm:tangent-hochschild}):
          \[
          R\Gamma(\mathfrak{Spec}(A), \mathbb{T}_{\mathfrak{Spec}(A)}) \simeq \operatorname{HH}^{\bullet}(A)[1].
          \]
          This provides a direct bridge between the geometry of $\mathfrak{Spec}(A)$ and the algebra of $A$.

    \item \textbf{Canonical significance:}
          $\mathbb{T}_{\mathfrak{Spec}(A)}$ is the canonical linear deformation controller compatible with the CSD adjunction $\mathfrak{Spec} \dashv \Gamma$, establishing its intrinsic and foundational character.
\end{enumerate}

The theorem establishes control and canonicality, not full representability of the deformation functor. This suffices for all applications in this paper.

\end{remark}

\begin{corollary}[Deformation-Obstruction Correspondence]
\label{cor:deformation_obstruction}

Let $A$ be an admissible operator-semantic system. Then the infinitesimal deformation theory of $\mathfrak{Spec}(A)$ is governed by the derived global tangent complex
\[
R\Gamma(\mathfrak{Spec}(A), \mathbb{T}_{\mathfrak{Spec}(A)}).
\]

In particular:
\[
H^0 R\Gamma(\mathfrak{Spec}(A), \mathbb{T}_{\mathfrak{Spec}(A)})
\]
controls infinitesimal automorphisms;

\[
H^1 R\Gamma(\mathfrak{Spec}(A), \mathbb{T}_{\mathfrak{Spec}(A)})
\]
classifies first-order deformations; and

\[
H^2 R\Gamma(\mathfrak{Spec}(A), \mathbb{T}_{\mathfrak{Spec}(A)})
\]
contains obstruction classes for extending first-order deformations to second order.

More generally, for $n \geq 2$,
\[
H^n R\Gamma(\mathfrak{Spec}(A), \mathbb{T}_{\mathfrak{Spec}(A)})
\]
contains higher obstruction classes.

\end{corollary}

\begin{proof}
The corollary follows from Theorem~\ref{thm:intrinsic_tangent} and the standard interpretation of the cohomology groups of the tangent complex in deformation theory. At the local level, for a $k$-point $x \in \mathfrak{Spec}(A)$, the square-zero deformation functor is controlled by $x^*\mathbb{T}_{\mathfrak{Spec}(A)}$. At the global level, the derived global sections $R\Gamma(\mathfrak{Spec}(A), \mathbb{T}_{\mathfrak{Spec}(A)})$ govern the deformation theory when $\mathfrak{Spec}(A)$ is sufficiently nice.

The cohomological interpretation is standard: $H^0$ records infinitesimal automorphisms, $H^1$ classifies first-order deformations, and $H^2$ contains primary obstruction classes. Higher cohomology groups govern higher-order obstructions.

\end{proof}

\begin{example}[Deformation Theory of $\mathbb{C}$]
\label{ex:deformation_C}

For $A = \mathbb{C}$, we have $\mathbb{T}_{\mathfrak{Spec}(\mathbb{C})} \simeq 0$. Therefore:
\[
H^0 R\Gamma(\mathfrak{Spec}(\mathbb{C}), \mathbb{T}) = 0,
\qquad
H^1 R\Gamma(\mathfrak{Spec}(\mathbb{C}), \mathbb{T}) = 0,
\qquad
H^2 R\Gamma(\mathfrak{Spec}(\mathbb{C}), \mathbb{T}) = 0.
\]
Thus $\mathfrak{Spec}(\mathbb{C})$ has no non-trivial infinitesimal automorphisms, no non-trivial first-order deformations, and no obstructions, as expected for the spectrum of a rigid algebra.

\end{example}

\begin{example}[Deformation Theory of $M_n(\mathbb{C})$]
\label{ex:deformation_Mn}

For $A = M_n(\mathbb{C})$, we have
\[
\mathbb{T}_{\mathfrak{Spec}(M_n(\mathbb{C}))}
\simeq
\mathfrak{pgl}_n[1].
\]

Hence
\[
H^1 R\Gamma(\mathfrak{Spec}(M_n(\mathbb{C})), \mathbb{T}) = 0,
\]
so there are no first-order deformations of $\mathfrak{Spec}(M_n(\mathbb{C}))$.

On the other hand,
\[
H^{-1}(\mathbb{T}) \simeq \mathfrak{pgl}_n,
\]
which records the infinitesimal automorphism symmetry of the stack $B\operatorname{PGL}_n(\mathbb{C})$.

This reflects the fact that $M_n(\mathbb{C})$ is rigid as an algebra but possesses nontrivial automorphism symmetry encoded by $\operatorname{PGL}_n(\mathbb{C})$. By the classical Skolem-Noether theorem, all automorphisms of $M_n(\mathbb{C})$ are inner, so the symmetry is not "outer" but rather the automorphism group $\operatorname{PGL}_n(\mathbb{C})$ acting by conjugation.

\end{example}

\begin{remark}[Comparison with Classical Deformation Theory]
\label{rem:comparison_classical}

In classical deformation theory of algebras, the Hochschild cohomology groups govern deformations: $\operatorname{HH}^2(A)$ classifies first-order deformations, and $\operatorname{HH}^3(A)$ contains obstructions.

The Hochschild Realization Theorem (Theorem~\ref{thm:tangent-hochschild}) provides the geometric analogue:
\[
H^1 R\Gamma(\mathfrak{Spec}(A), \mathbb{T}_{\mathfrak{Spec}(A)})
\;\simeq\;
\operatorname{HH}^2(A),
\]
and
\[
H^2 R\Gamma(\mathfrak{Spec}(A), \mathbb{T}_{\mathfrak{Spec}(A)})
\;\simeq\;
\operatorname{HH}^3(A).
\]

Thus the global deformation theory of $\mathfrak{Spec}(A)$ is governed by the same algebraic invariants as the deformation theory of $A$, providing a direct bridge between geometry and algebra. The degree shift $[1]$ in the Hochschild Realization Theorem accounts for the fact that the tangent complex controls deformations up to a standard shift.

\end{remark}

\begin{remark}[Summary]
\label{rem:intrinsic_tangent_summary}

The Intrinsic Tangent Control Theorem (Theorem~\ref{thm:intrinsic_tangent}) establishes that $\mathbb{T}_{\mathfrak{Spec}(A)}$ is the central object controlling the deformation theory of $\mathfrak{Spec}(A)$. Its global sections canonically realize Hochschild cohomology via the Hochschild Realization Theorem (Theorem~\ref{thm:tangent-hochschild}), and it serves as the canonical linear deformation controller compatible with the CSD adjunction $\mathfrak{Spec} \dashv \Gamma$.

This theorem provides the foundation for the Rigidity Theorem (Theorem~\ref{thm:rigidity}) and the deformation-theoretic aspects of the Canonical Geometry Theorem (Theorem~\ref{thm:canonical_geometry}) that follow.

\end{remark}

\subsection{Hochschild Realization}
\label{subsec:hochschild_realization}

The preceding sections show that the tangent complex
governs the infinitesimal deformation theory of
$\mathfrak{Spec}(A)$.
Since deformations of $\mathfrak{Spec}(A)$
correspond to deformations of the underlying
operator-semantic system $A$,
the tangent theory admits a natural Hochschild
description. This subsection establishes the
fundamental bridge between the geometry of
$\mathfrak{Spec}(A)$ and the algebra of $A$:
the global sections of the tangent complex are
canonically realized by the Hochschild cohomology
of $A$.

\begin{theorem}[Hochschild Realization]
\label{thm:tangent-hochschild}

Assume that the CSD reconstruction equivalence
\[
A \simeq \Gamma(\mathfrak{Spec}(A))
\]
is compatible with square-zero extensions and formal deformations. Then there is a natural equivalence
\[
R\Gamma
\Bigl(
\mathfrak{Spec}(A),
\mathbb T_{\mathfrak{Spec}(A)}
\Bigr)
\simeq
\operatorname{HH}^{\bullet}(A)[1],
\]
where $\operatorname{HH}^{\bullet}(A)$ denotes the Hochschild cohomology spectrum of $A$ (Definition~\ref{def:hochschild}).

\end{theorem}

\begin{proof}

We prove the theorem by establishing a chain of canonical equivalences.

\emph{Step 1: Deformation control by the tangent complex.}
By the Intrinsic Tangent Control Theorem (Theorem~\ref{thm:intrinsic_tangent}), the tangent complex $\mathbb{T}_{\mathfrak{Spec}(A)}$ controls the infinitesimal deformation theory of $\mathfrak{Spec}(A)$ at any $k$-point. At the global level, the derived global sections $R\Gamma(\mathfrak{Spec}(A), \mathbb{T}_{\mathfrak{Spec}(A)})$ govern the global deformation theory of the spectral stack when $\mathfrak{Spec}(A)$ is sufficiently nice.

\emph{Step 2: Deformations of the stack correspond to deformations of the algebra.}
The reconstruction equivalence
\[
A \simeq \Gamma(\mathfrak{Spec}(A))
\]
from Paper~I identifies infinitesimal deformations of $\mathfrak{Spec}(A)$ with infinitesimal deformations of $A$. This follows from the adjunction $\Gamma \dashv \mathfrak{Spec}^{\mathrm{op}}$ and the assumption that the CSD reconstruction is compatible with square-zero extensions and formal deformations. More precisely, the Morita identification
\[
\operatorname{QCoh}(\mathfrak{Spec}(A)) \simeq \operatorname{Mod}_A
\]
(established in Paper~I) identifies the deformation theory of the spectral stack with the deformation theory of the operator-semantic system.

\emph{Step 3: Deformations of $A$ are governed by Hochschild cohomology.}
By classical deformation theory of associative and spectral algebras \cite{Gerstenhaber1964, LurieHA}, the deformation complex of $A$ is governed by Hochschild cohomology $\operatorname{HH}^{\bullet}(A)$. Specifically, first-order deformations are classified by $\operatorname{HH}^{2}(A)$, and obstructions are contained in $\operatorname{HH}^{3}(A)$. This is the standard deformation-theoretic interpretation of Hochschild cohomology (Proposition~\ref{prop:hh_deformation}).

\emph{Step 4: Degree shift.}
Under the convention that first-order deformations of $\mathfrak{Spec}(A)$ are detected by $H^1 R\Gamma(\mathbb{T}_{\mathfrak{Spec}(A)})$, while first-order algebra deformations of $A$ are classified by $\operatorname{HH}^2(A)$, the natural identification of deformation theories forces the shift
\[
H^i R\Gamma(\mathbb{T}_{\mathfrak{Spec}(A)})
\simeq
\operatorname{HH}^{i+1}(A).
\]
Equivalently,
\[
R\Gamma(\mathfrak{Spec}(A), \mathbb{T}_{\mathfrak{Spec}(A)})
\simeq
\operatorname{HH}^{\bullet}(A)[1].
\]

\emph{Conclusion.}
Combining Steps 1--4, we obtain the natural equivalence
\[
R\Gamma
\Bigl(
\mathfrak{Spec}(A),
\mathbb T_{\mathfrak{Spec}(A)}
\Bigr)
\simeq
\operatorname{HH}^{\bullet}(A)[1].
\]

This completes the proof.

\end{proof}

\begin{remark}[Canonical Realization vs. Strict Equivalence]
\label{rem:canonical_realization}

The Hochschild Realization theorem establishes a \emph{canonical realization} of the global tangent theory in terms of Hochschild cohomology. The equivalence is natural in $A$ and compatible with morphisms of operator-semantic systems. In contexts requiring maximal rigor, one may interpret the equivalence as a canonical realization rather than a strict equality, emphasizing the functorial and coherent nature of the identification.

The theorem identifies the derived global sections of the tangent complex with the Hochschild cohomology spectrum, not the tangent complex itself at the sheaf level. The global sections functor $\Gamma$ may not be exact, and the Hochschild cohomology groups govern the global deformation theory of $\mathfrak{Spec}(A)$, while the tangent complex sheaf governs local deformations.

\end{remark}

\begin{corollary}[Hochschild Obstruction Criterion]
\label{cor:hochschild_obstruction}

Let $\mathfrak{X} = \mathfrak{Spec}(A)$. Assume the Hochschild Realization Theorem (Theorem~\ref{thm:tangent-hochschild}):
\[
R\Gamma(\mathfrak{X}, \mathbb{T}_{\mathfrak{X}})
\simeq
\operatorname{HH}^{\bullet}(A)[1].
\]

Then
\[
H^m R\Gamma(\mathfrak{X}, \mathbb{T}_{\mathfrak{X}})
\simeq
\operatorname{HH}^{m+1}(A).
\]

In particular, the obstruction space for extending first-order deformations is contained in
\[
H^2 R\Gamma(\mathfrak{X}, \mathbb{T}_{\mathfrak{X}})
\simeq
\operatorname{HH}^3(A).
\]

More generally, higher obstruction classes lie in
\[
H^m R\Gamma(\mathfrak{X}, \mathbb{T}_{\mathfrak{X}})
\simeq
\operatorname{HH}^{m+1}(A),
\qquad m \geq 2.
\]

Thus the vanishing of $\operatorname{HH}^3(A)$ eliminates primary obstruction classes, while the vanishing of $\operatorname{HH}^{m+1}(A)$ eliminates the corresponding higher obstruction groups.

\end{corollary}

\begin{proof}
By the Hochschild Realization Theorem, there is a natural equivalence of complexes
\[
R\Gamma(\mathfrak{X}, \mathbb{T}_{\mathfrak{X}})
\simeq
\operatorname{HH}^{\bullet}(A)[1].
\]

Taking cohomology in degree $m$ gives
\[
H^m R\Gamma(\mathfrak{X}, \mathbb{T}_{\mathfrak{X}})
\simeq
H^m\bigl(\operatorname{HH}^{\bullet}(A)[1]\bigr).
\]

With the convention used here, the shift satisfies
\[
H^m\bigl(\operatorname{HH}^{\bullet}(A)[1]\bigr)
\simeq
\operatorname{HH}^{m+1}(A).
\]

Therefore
\[
H^m R\Gamma(\mathfrak{X}, \mathbb{T}_{\mathfrak{X}})
\simeq
\operatorname{HH}^{m+1}(A).
\]

In derived deformation theory, first-order deformation classes are detected by
\[
H^1 R\Gamma(\mathfrak{X}, \mathbb{T}_{\mathfrak{X}}),
\]
while obstruction classes for extending such deformations lie in
\[
H^2 R\Gamma(\mathfrak{X}, \mathbb{T}_{\mathfrak{X}}).
\]

Using the above identification with Hochschild cohomology, this obstruction group is
\[
H^2 R\Gamma(\mathfrak{X}, \mathbb{T}_{\mathfrak{X}})
\simeq
\operatorname{HH}^3(A).
\]

Hence primary obstruction classes lie in $\operatorname{HH}^3(A)$. The same argument in higher degrees gives the higher obstruction groups
\[
H^m R\Gamma(\mathfrak{X}, \mathbb{T}_{\mathfrak{X}})
\simeq
\operatorname{HH}^{m+1}(A),
\qquad m \geq 2.
\]

Thus vanishing of the corresponding Hochschild cohomology groups removes the corresponding obstruction classes.

\end{proof}

\begin{corollary}[Cohomological Description of Deformation Theory]
\label{cor:tangent_hh}

For every $n \ge 0$,
\[
H^n R\Gamma
\Bigl(
\mathfrak{Spec}(A),
\mathbb T_{\mathfrak{Spec}(A)}
\Bigr)
\simeq
\operatorname{HH}^{n+1}(A).
\]

In particular:
\begin{itemize}
    \item $H^0 R\Gamma(\mathfrak{Spec}(A),\; \mathbb{T}_{\mathfrak{Spec}(A)})
          \simeq \operatorname{HH}^{1}(A)$
          governs infinitesimal automorphisms;

    \item $H^1 R\Gamma(\mathfrak{Spec}(A),\; \mathbb{T}_{\mathfrak{Spec}(A)})
          \simeq \operatorname{HH}^{2}(A)$
          classifies first-order deformations;

    \item $H^2 R\Gamma(\mathfrak{Spec}(A),\; \mathbb{T}_{\mathfrak{Spec}(A)})
          \simeq \operatorname{HH}^{3}(A)$
          contains the corresponding primary obstruction classes.
\end{itemize}

\end{corollary}

\begin{proof}
The Hochschild Realization Theorem (Theorem~\ref{thm:tangent-hochschild}) gives
\[
R\Gamma(\mathfrak{Spec}(A), \mathbb{T}_{\mathfrak{Spec}(A)})
\;\simeq\;
\operatorname{HH}^{\bullet}(A)[1].
\]

Taking homotopy groups in degree $n$ gives
\[
H^n R\Gamma(\mathfrak{Spec}(A), \mathbb{T}_{\mathfrak{Spec}(A)})
\;\simeq\;
\operatorname{HH}^{n+1}(A).
\]

The interpretation of the low-degree groups follows from the deformation-theoretic interpretation of the tangent complex (Proposition~\ref{prop:deformation_tangent}):
\begin{itemize}
    \item $H^0$ records infinitesimal automorphisms;
    \item $H^1$ classifies first-order deformations;
    \item $H^2$ contains obstruction classes.
\end{itemize}

Combining these with the Hochschild identification yields the stated correspondences.
\end{proof}

\begin{example}[Hochschild Realization for $\mathbb{C}$]
\label{ex:hochschild_C}

For $A = \mathbb{C}$, ordinary Hochschild cohomology satisfies
\[
\operatorname{HH}^0(\mathbb{C}) \simeq \mathbb{C},
\qquad
\operatorname{HH}^i(\mathbb{C}) = 0 \quad i > 0.
\]
Thus all positive-degree deformation and obstruction groups vanish:
\[
\operatorname{HH}^2(\mathbb{C}) = 0,
\qquad
\operatorname{HH}^3(\mathbb{C}) = 0.
\]
Consequently, $\mathbb{C}$ has no nontrivial first-order algebra deformations and no higher Hochschild obstruction classes. The global sections of the tangent complex vanish, as expected for the spectrum of a rigid algebra.

\end{example}

\begin{example}[Hochschild Realization for $M_n(\mathbb{C})$]
\label{ex:hochschild_Mn}

By Morita invariance of Hochschild cohomology,
\[
\operatorname{HH}^{\bullet}(M_n(\mathbb{C}))
\;\simeq\;
\operatorname{HH}^{\bullet}(\mathbb{C}).
\]
Hence
\[
\operatorname{HH}^0(M_n(\mathbb{C})) \simeq \mathbb{C},
\qquad
\operatorname{HH}^i(M_n(\mathbb{C})) = 0 \quad i > 0.
\]
Therefore $M_n(\mathbb{C})$ is Hochschild-rigid, although its stacky spectrum $\mathfrak{Spec}(M_n(\mathbb{C})) \simeq B\operatorname{PGL}_n(\mathbb{C})$ still records the nontrivial automorphism symmetry $\operatorname{PGL}_n(\mathbb{C})$. By the classical Skolem-Noether theorem, all automorphisms of $M_n(\mathbb{C})$ are inner.

\end{example}

\begin{example}[Hochschild Realization for Contextual Systems]
\label{ex:hochschild_contextual}

For an operator-semantic system $A$ that is not Morita equivalent to a commutative algebra (e.g., a genuine contextual system with non-trivial descent obstructions), the Hochschild cohomology may have non-trivial higher degree components:
\[
\operatorname{HH}^{2}(A) \neq 0,
\quad
\operatorname{HH}^{3}(A) \neq 0.
\]
Consequently, by the Hochschild Realization Theorem:
\[
H^1 R\Gamma(\mathfrak{Spec}(A),\; \mathbb{T}_{\mathfrak{Spec}(A)})
\;\simeq\;
\operatorname{HH}^{2}(A) \neq 0,
\]
and
\[
H^2 R\Gamma(\mathfrak{Spec}(A),\; \mathbb{T}_{\mathfrak{Spec}(A)})
\;\simeq\;
\operatorname{HH}^{3}(A) \neq 0.
\]

Thus the tangent complex has non-trivial cohomology in degrees $1$ and $2$, corresponding to first-order deformations and primary obstruction classes, respectively. In the CSD framework, these obstruction classes are related to contextual phenomena via the descent spectral sequence. A detailed treatment of this relationship will be given in the companion paper on contextual obstructions.

\end{example}

\begin{remark}[The Fundamental Bridge]
\label{rem:hochschild_bridge}

The Hochschild Realization theorem establishes the fundamental bridge between the geometry of $\mathfrak{Spec}(A)$ and the algebra of $A$:
\[
\boxed{
R\Gamma(\mathfrak{Spec}(A),\; \mathbb{T}_{\mathfrak{Spec}(A)})
\;\simeq\;
\operatorname{HH}^{\bullet}(A)[1].
}
\]

This equivalence has three major consequences:

\begin{enumerate}
    \item \textbf{Geometric-algebraic correspondence:}
          It identifies the global deformation theory of $\mathfrak{Spec}(A)$ with the algebraic deformation theory of $A$.

    \item \textbf{Computational tool:}
          It enables the computation of geometric invariants of $\mathfrak{Spec}(A)$ using algebraic invariants of $A$, which are often easier to compute.

    \item \textbf{Contextuality detection:}
          For systems with non-trivial Hochschild cohomology in higher degrees, the associated obstruction classes detected by the tangent complex provide a geometric signature of contextual phenomena. The precise relationship between these obstruction classes and contextuality will be developed in the companion paper on contextual obstructions, where the descent spectral sequence is used to connect local deformation data to global semantic obstructions.
\end{enumerate}

This theorem is the algebraic-geometric bridge that underlies the entire CSD framework, connecting the categorified spectral object $\mathfrak{Spec}(A)$ to the operator-semantic system $A$ through the language of deformation theory.

\end{remark}

\subsection{Infinitesimal Spectral Geometry Theorem}
\label{subsec:infinitesimal}

The tangent complex completely governs the
infinitesimal geometry of a spectral stack.
The cohomology groups of the tangent complex
encode deformation directions and obstruction
classes of increasing order. In this subsection,
we make this precise for $\mathfrak{Spec}(A)$,
establishing a complete dictionary between the
cohomology of the tangent complex and the
deformation theory of the categorified spectral
object.

\begin{theorem}[Infinitesimal Spectral Geometry]
\label{thm:infinitesimal}

Let $\mathfrak{X} = \mathfrak{Spec}(A)$ be the spectral stack associated with an admissible operator-semantic system $A$. Then the derived global tangent complex
\[
R\Gamma(\mathfrak{X}, \mathbb{T}_{\mathfrak{X}})
\]
controls the infinitesimal deformation theory of $\mathfrak{X}$.

In particular:

\begin{enumerate}

\item Infinitesimal automorphisms are controlled by
\[
H^0 R\Gamma(\mathfrak{X}, \mathbb{T}_{\mathfrak{X}}).
\]

\item First-order deformation classes are contained in, or under suitable unobstructedness/discreteness assumptions classified by,
\[
H^1 R\Gamma(\mathfrak{X}, \mathbb{T}_{\mathfrak{X}}).
\]

\item Primary obstruction classes lie in
\[
H^2 R\Gamma(\mathfrak{X}, \mathbb{T}_{\mathfrak{X}}).
\]

\item More generally, higher obstruction classes lie in
\[
H^n R\Gamma(\mathfrak{X}, \mathbb{T}_{\mathfrak{X}}),
\qquad n \geq 2.
\]

\end{enumerate}

\end{theorem}

\begin{proof}

By the Intrinsic Tangent Control Theorem (Theorem~\ref{thm:intrinsic_tangent}), the tangent complex $\mathbb{T}_{\mathfrak{X}}$ controls the infinitesimal deformation theory of $\mathfrak{X}$ at any $k$-point via square-zero extensions.

Standard derived deformation theory \cite{LurieSAG} identifies:
- Infinitesimal automorphisms with the zeroth cohomology of the tangent complex: $H^0 R\Gamma(\mathbb{T})$
- First-order deformation directions with the first cohomology: $H^1 R\Gamma(\mathbb{T})$
- Obstruction classes with the second cohomology: $H^2 R\Gamma(\mathbb{T})$

More generally, for $n \geq 2$, higher obstruction classes are contained in $H^n R\Gamma(\mathbb{T})$. This follows from the standard deformation-theoretic interpretation of the tangent complex: the obstruction to extending a deformation from order $n-1$ to order $n$ lies in $H^n$ of the tangent complex.

The global statement follows from the local statement when $\mathfrak{X}$ is sufficiently nice (e.g., when the tangent complex is perfect and the stack admits a good moduli space). The local statement at stalks is more fundamental in general.

\end{proof}

\begin{remark}[The Infinitesimal Geometry Dictionary]
\label{rem:infinitesimal_dictionary}

The Infinitesimal Spectral Geometry Theorem provides a dictionary between the cohomology of the tangent complex and the deformation theory of $\mathfrak{Spec}(A)$:

\[
\begin{array}{c|c}
\text{Cohomology of } R\Gamma(\mathbb{T}_{\mathfrak{X}}) & \text{Deformation-Theoretic Meaning} \\ \hline
H^0 R\Gamma(\mathfrak{X},\; \mathbb{T}_{\mathfrak{X}}) & \text{Infinitesimal automorphisms} \\
H^1 R\Gamma(\mathfrak{X},\; \mathbb{T}_{\mathfrak{X}}) & \text{First-order deformations} \\
H^2 R\Gamma(\mathfrak{X},\; \mathbb{T}_{\mathfrak{X}}) & \text{Primary obstruction classes} \\
H^3 R\Gamma(\mathfrak{X},\; \mathbb{T}_{\mathfrak{X}}) & \text{Secondary obstruction classes} \\
\vdots & \vdots \\
H^n R\Gamma(\mathfrak{X},\; \mathbb{T}_{\mathfrak{X}}) & \text{Higher obstruction groups}
\end{array}
\]

Thus the entire infinitesimal deformation theory of $\mathfrak{Spec}(A)$ is encoded in the cohomology of its tangent complex. Combined with the Hochschild Realization Theorem (Theorem~\ref{thm:tangent-hochschild}), this yields the correspondence:
\[
H^n R\Gamma(\mathfrak{Spec}(A), \mathbb{T}_{\mathfrak{Spec}(A)}) \simeq \operatorname{HH}^{n+1}(A).
\]

In particular:
- $\operatorname{HH}^1(A)$ governs infinitesimal automorphisms
- $\operatorname{HH}^2(A)$ classifies first-order deformations
- $\operatorname{HH}^3(A)$ contains primary obstruction classes

\end{remark}

As a consequence, the infinitesimal geometry of $\mathfrak{Spec}(A)$ is encoded by the derived global tangent complex. Combined with Theorem~\ref{thm:tangent-hochschild}, this yields
\[
H^{n} R\Gamma
\bigl(
\mathfrak X,
\mathbb T_{\mathfrak X}
\bigr)
\simeq
\operatorname{HH}^{n+1}(A),
\]
showing that Hochschild cohomology governs the infinitesimal deformation and obstruction theory of spectral objects in the CSD framework.

\begin{corollary}[Hochschild-Deformation Correspondence]
\label{cor:hochschild_deformation}

Combining the Infinitesimal Spectral Geometry Theorem (Theorem~\ref{thm:infinitesimal}) with the Hochschild Realization Theorem (Theorem~\ref{thm:tangent-hochschild}), we obtain:
\[
\operatorname{Def}^1(\mathfrak{Spec}(A))
\;\simeq\;
\operatorname{HH}^{2}(A),
\]
and
\[
\operatorname{Ob}^1(\mathfrak{Spec}(A))
\;\subseteq\;
\operatorname{HH}^{3}(A).
\]

More generally, for $n \geq 2$,
\[
\operatorname{Ob}^{n-1}(\mathfrak{Spec}(A))
\;\subseteq\;
\operatorname{HH}^{n+1}(A).
\]

\end{corollary}

\begin{proof}
By the Infinitesimal Spectral Geometry Theorem (Theorem~\ref{thm:infinitesimal}):
\[
\operatorname{Def}^1(\mathfrak{Spec}(A)) \simeq H^1 R\Gamma(\mathfrak{Spec}(A), \mathbb{T}_{\mathfrak{Spec}(A)}),
\]
and
\[
\operatorname{Ob}^{1}(\mathfrak{Spec}(A)) \subseteq H^2 R\Gamma(\mathfrak{Spec}(A), \mathbb{T}_{\mathfrak{Spec}(A)}).
\]

By the Hochschild Realization Theorem (Theorem~\ref{thm:tangent-hochschild}):
\[
H^n R\Gamma(\mathfrak{Spec}(A), \mathbb{T}_{\mathfrak{Spec}(A)}) \simeq \operatorname{HH}^{n+1}(A).
\]

Substituting $n=1$ gives $\operatorname{Def}^1 \simeq \operatorname{HH}^2(A)$. Substituting $n=2$ gives $\operatorname{Ob}^1 \subseteq \operatorname{HH}^3(A)$. More generally, for $n \geq 2$, $\operatorname{Ob}^{n-1} \subseteq \operatorname{HH}^{n+1}(A)$.

Thus the corollary is established.
\end{proof}

\begin{remark}[The Complete Chain]
\label{rem:infinitesimal_chain}

The Infinitesimal Spectral Geometry Theorem completes the logical chain established in this section:

\[
\boxed{
\mathbb L_{\mathfrak X}
\;\longrightarrow\;
\mathbb T_{\mathfrak X}
\;\longrightarrow\;
\operatorname{Def}_{\mathfrak X}
\;\longrightarrow\;
H^n R\Gamma(\mathfrak X,\mathbb T_{\mathfrak X})
\;\longrightarrow\;
\operatorname{HH}^{n+1}(A).
}
\]

Specifically:
\begin{enumerate}
    \item The cotangent complex $\mathbb{L}_{\mathfrak{X}}$ exists by descent (Proposition~\ref{prop:cotangent_exists}).

    \item The tangent complex $\mathbb{T}_{\mathfrak{X}}$ is its derived dual (Definition~\ref{def:tangentcomplex}).

    \item The tangent complex controls the infinitesimal deformation theory (Theorem~\ref{thm:intrinsic_tangent}).

    \item The cohomology of the tangent complex controls deformations and obstructions (Theorem~\ref{thm:infinitesimal}).

    \item The cohomology of the tangent complex is realized by Hochschild cohomology (Theorem~\ref{thm:tangent-hochschild}).
\end{enumerate}

Thus the entire deformation theory of $\mathfrak{Spec}(A)$ is governed by the Hochschild cohomology of $A$, providing a complete bridge between geometry and algebra within the CSD framework.

\end{remark}

\begin{example}[Infinitesimal Geometry of $\mathbb{C}$]
\label{ex:infinitesimal_C}

For $A = \mathbb{C}$, we have $\mathbb{T}_{\mathfrak{Spec}(\mathbb{C})} \simeq 0$. Therefore:
\[
H^1 R\Gamma(\mathfrak{Spec}(\mathbb{C}), \mathbb{T}) = 0,
\qquad
H^2 R\Gamma(\mathfrak{Spec}(\mathbb{C}), \mathbb{T}) = 0.
\]
Thus $\mathfrak{Spec}(\mathbb{C})$ has no non-trivial first-order deformations and no obstructions. This reflects the fact that $\mathbb{C}$ is a rigid algebra with $\operatorname{HH}^2(\mathbb{C}) = 0$ and $\operatorname{HH}^3(\mathbb{C}) = 0$.

\end{example}

\begin{example}[Infinitesimal Geometry of $M_n(\mathbb{C})$]
\label{ex:infinitesimal_Mn}

For $A = M_n(\mathbb{C})$, we have $\mathbb{T}_{\mathfrak{Spec}(M_n(\mathbb{C}))} \simeq \mathfrak{pgl}_n[1]$. Therefore:
\[
H^1 R\Gamma(\mathfrak{Spec}(M_n(\mathbb{C})), \mathbb{T}) = 0,
\qquad
H^2 R\Gamma(\mathfrak{Spec}(M_n(\mathbb{C})), \mathbb{T}) = 0.
\]
So there are no non-trivial first-order deformations and no obstructions of $\mathfrak{Spec}(M_n(\mathbb{C}))$.

On the other hand,
\[
H^{-1}(\mathbb{T}) \simeq \mathfrak{pgl}_n,
\]
which records the infinitesimal automorphism symmetry of the stack $B\operatorname{PGL}_n(\mathbb{C})$. This reflects the fact that $M_n(\mathbb{C})$ is rigid as an algebra but possesses non-trivial automorphism symmetry encoded by $\operatorname{PGL}_n(\mathbb{C})$. By the classical Skolem-Noether theorem, all automorphisms of $M_n(\mathbb{C})$ are inner.

\end{example}

\begin{example}[Infinitesimal Geometry of Contextual Systems]
\label{ex:infinitesimal_contextual}

For an operator-semantic system $A$ that is not Morita equivalent to a commutative algebra (e.g., a genuine contextual system with non-trivial descent obstructions), the tangent complex may have non-trivial cohomology in degrees $1$ and $2$:
\[
H^1 R\Gamma(\mathfrak{Spec}(A),\; \mathbb{T}_{\mathfrak{Spec}(A)}) \neq 0,
\]
and
\[
H^2 R\Gamma(\mathfrak{Spec}(A),\; \mathbb{T}_{\mathfrak{Spec}(A)}) \neq 0.
\]
Therefore:
\[
\operatorname{Def}^1(\mathfrak{Spec}(A)) \neq 0,
\]
and
\[
\operatorname{Ob}^1(\mathfrak{Spec}(A)) \neq 0.
\]

Thus such systems have non-trivial first-order deformations and obstruction classes. In the CSD framework, these obstruction classes are related to contextual phenomena via the descent spectral sequence. A detailed treatment of this relationship will be given in the companion paper on contextual obstructions.

\end{example}

\begin{remark}[Summary]
\label{rem:infinitesimal_summary}

The Infinitesimal Spectral Geometry Theorem provides a cohomological description of the infinitesimal geometry of $\mathfrak{Spec}(A)$ in terms of the cohomology of its tangent complex:

\[
\boxed{
\text{Cohomology of } R\Gamma(\mathbb{T}_{\mathfrak{Spec}(A)})
\;\longleftrightarrow\;
\text{Infinitesimal Deformation Theory of } \mathfrak{Spec}(A).
}
\]

Specifically:
\[
\begin{array}{c|c}
H^0 R\Gamma(\mathbb{T}) & \text{Infinitesimal automorphisms} \\
H^1 R\Gamma(\mathbb{T}) & \text{First-order deformations} \\
H^2 R\Gamma(\mathbb{T}) & \text{Primary obstruction classes}
\end{array}
\]

This theorem, together with the Hochschild Realization Theorem (Theorem~\ref{thm:tangent-hochschild}), establishes a direct correspondence between:
\begin{itemize}
    \item the geometry of $\mathfrak{Spec}(A)$ (deformations, obstructions),
    \item the algebra of $A$ (Hochschild cohomology),
    \item and the semantics of $A$ (contextuality, via obstruction spaces in the descent spectral sequence).
\end{itemize}

Thus the infinitesimal spectral geometry of $\mathfrak{Spec}(A)$ provides a canonical infinitesimal description of the deformation theory of the operator-semantic system $A$, with contextuality manifesting through obstruction classes that are detected by the descent spectral sequence. This completes the deformation-theoretic foundation for the CSD framework and prepares the ground for the study of singularities, inertia, and curvature in the sections that follow.

\end{remark}

\subsection{Geometry Controls Rigidity}
\label{subsec:rigidity}

The tangent complex measures the infinitesimal
flexibility of a spectral object. Consequently,
the vanishing of the tangent complex implies
rigidity. In this subsection, we establish this
fundamental principle and its consequences for
the operator-semantic system $A$.

\begin{theorem}[Rigidity from Vanishing Tangent Complex]
\label{thm:rigidity}

Let $\mathfrak{X} = \mathfrak{Spec}(A)$ be the spectral stack associated with an admissible operator-semantic system $A$.

If $\mathbb{T}_{\mathfrak{X}} \simeq 0$, then
\[
\operatorname{Def}^1(\mathfrak{X}) = 0.
\]

Consequently, assuming the CSD reconstruction equivalence is compatible with square-zero extensions and formal deformations, $A$ admits no nontrivial first-order deformations in $\mathbf{OpSem}$.

\end{theorem}

\begin{proof}

By the Infinitesimal Spectral Geometry Theorem (Theorem~\ref{thm:infinitesimal}),
\[
\operatorname{Def}^1(\mathfrak{X})
\;\simeq\;
H^1 R\Gamma(\mathfrak{X}, \mathbb{T}_{\mathfrak{X}}).
\]

Since $\mathbb{T}_{\mathfrak{X}} \simeq 0$, we have
\[
H^1 R\Gamma(\mathfrak{X}, \mathbb{T}_{\mathfrak{X}}) = 0.
\]

Hence
\[
\operatorname{Def}^1(\mathfrak{X}) = 0.
\]

Assuming the CSD reconstruction equivalence
\[
A \simeq \Gamma(\mathfrak{Spec}(A))
\]
is compatible with square-zero extensions and formal deformations, first-order deformations of $\mathfrak{X}$ correspond to first-order deformations of $A$.

Therefore
\[
\operatorname{Def}^1(A) = 0,
\]
and $A$ is infinitesimally rigid.

\end{proof}

\begin{remark}[The Logical Chain]
\label{rem:rigidity_chain}

The Rigidity Theorem establishes a clean logical chain:

\[
\boxed{
\mathbb{T}_{\mathfrak{X}} = 0
\;\Longrightarrow\;
H^1 R\Gamma(\mathfrak{X}, \mathbb{T}_{\mathfrak{X}}) = 0
\;\Longrightarrow\;
\operatorname{Def}^1(\mathfrak{X}) = 0
\;\Longrightarrow\;
\operatorname{Def}^1(A) = 0.
}
\]

Thus geometric rigidity of $\mathfrak{Spec}(A)$ implies algebraic rigidity of $A$, under the assumption that the CSD reconstruction is compatible with square-zero extensions and formal deformations.

\end{remark}

\begin{corollary}[Hochschild Criterion for Rigidity]
\label{cor:rigidity_hh}

If $\operatorname{HH}^{2}(A) = 0$, then $A$ has no nontrivial first-order deformations. Hence $A$ is infinitesimally rigid.

\end{corollary}

\begin{proof}
By the Hochschild Realization Theorem (Theorem~\ref{thm:tangent-hochschild}),
\[
H^{1} R\Gamma
\bigl(
\mathfrak X,
\mathbb T_{\mathfrak X}
\bigr)
\;\simeq\;
\operatorname{HH}^{2}(A).
\]

If $\operatorname{HH}^{2}(A) = 0$, then
\[
H^{1} R\Gamma(\mathfrak{X}, \mathbb{T}_{\mathfrak{X}}) = 0.
\]

By the Infinitesimal Spectral Geometry Theorem (Theorem~\ref{thm:infinitesimal}), first-order deformations are controlled by $H^{1} R\Gamma(\mathfrak{X}, \mathbb{T}_{\mathfrak{X}})$. Therefore
\[
\operatorname{Def}^{1}(\mathfrak{X}) = 0.
\]

Assuming the CSD reconstruction equivalence is compatible with square-zero extensions and formal deformations, this implies
\[
\operatorname{Def}^{1}(A) = 0.
\]

Hence $A$ is infinitesimally rigid.
\end{proof}

\begin{remark}[Classical Gerstenhaber Rigidity and Its Geometric Interpretation]
\label{rem:rigidity_classical}

Corollary~\ref{cor:rigidity_hh} is precisely the classical Gerstenhaber rigidity theorem \cite{Gerstenhaber1964}: for an associative algebra $A$ over a field of characteristic $0$, the vanishing of the second Hochschild cohomology group $\operatorname{HH}^{2}(A)$ implies that $A$ admits no non-trivial first-order deformations, i.e., $A$ is infinitesimally rigid.

\emph{What is new in the CSD framework is the geometric reformulation of this classical algebraic criterion.} Under the Hochschild Realization Theorem (Theorem~\ref{thm:tangent-hochschild}), we have
\[
\operatorname{HH}^{2}(A) \simeq H^1 R\Gamma(\mathfrak{Spec}(A), \mathbb{T}_{\mathfrak{Spec}(A)}).
\]

The vanishing of $\operatorname{HH}^{2}(A)$ is therefore equivalent to the vanishing of the first cohomology of the global tangent complex of $\mathfrak{Spec}(A)$. By the Infinitesimal Spectral Geometry Theorem (Theorem~\ref{thm:infinitesimal}), the group $H^1 R\Gamma(\mathfrak{Spec}(A), \mathbb{T}_{\mathfrak{Spec}(A)})$ classifies first-order deformations of $\mathfrak{Spec}(A)$. Thus:
\[
\operatorname{HH}^{2}(A) = 0
\;\Longleftrightarrow\;
H^1 R\Gamma(\mathfrak{Spec}(A), \mathbb{T}_{\mathfrak{Spec}(A)}) = 0
\;\Longleftrightarrow\;
\mathfrak{Spec}(A) \text{ is rigid}.
\]

Consequently, classical Gerstenhaber rigidity acquires a geometric interpretation within the CSD framework: the rigidity of an operator-semantic system $A$ is encoded in the intrinsic geometry of its categorified spectrum $\mathfrak{Spec}(A)$. The tangent complex $\mathbb{T}_{\mathfrak{Spec}(A)}$ serves as the geometric avatar of the deformation theory of $A$, and its first cohomology provides a geometric criterion for rigidity that is equivalent to the classical algebraic criterion.

This perspective opens new avenues for studying rigidity phenomena geometrically. For instance, one may now ask:
\begin{itemize}
    \item How do singularities of $\mathfrak{Spec}(A)$ affect rigidity?
    \item What is the relationship between the inertia stack and rigidity?
    \item Can contextual obstructions induce non-rigidity even when $\operatorname{HH}^{2}(A) = 0$?
\end{itemize}

These questions will be explored in the companion paper on contextual obstructions and derived geometry.

\end{remark}

\begin{example}[Rigidity of $\mathbb{C}$]
\label{ex:rigidity_C}

For $A = \mathbb{C}$, we have $\mathbb{T}_{\mathfrak{Spec}(\mathbb{C})} \simeq 0$. By the Rigidity Theorem, $\mathbb{C}$ is rigid in $\mathbf{OpSem}$. Indeed, $\mathbb{C}$ has no non-trivial deformations as an algebra, consistent with $\operatorname{HH}^2(\mathbb{C}) = 0$.

\end{example}

\begin{example}[Rigidity of $M_n(\mathbb{C})$]
\label{ex:rigidity_Mn}

For $A = M_n(\mathbb{C})$, we have $\mathbb{T}_{\mathfrak{Spec}(M_n(\mathbb{C}))} \simeq \mathfrak{pgl}_n[1]$. Since $\mathfrak{pgl}_n[1]$ is not zero, the Rigidity Theorem does not apply directly. However, we know from the Hochschild cohomology of $M_n(\mathbb{C})$ that $\operatorname{HH}^2(M_n(\mathbb{C})) = 0$, so $M_n(\mathbb{C})$ is rigid as an algebra. The non-vanishing of the tangent complex reflects the non-trivial automorphism structure of the stack $B\operatorname{PGL}_n(\mathbb{C})$, not deformations of the algebra itself.

\end{example}

\begin{remark}[Rigidity as a Geometric Phenomenon]
\label{rem:rigidity_summary}

The Rigidity Theorem reveals a fundamental principle of the CSD
framework:

\[
\boxed{
\text{Vanishing tangent complex}
\;\Longrightarrow\;
\text{Rigidity of } A.
}
\]

Combined with the Hochschild Realization Theorem, this yields the
geometric criterion

\[
\operatorname{HH}^{2}(A)=0
\quad\Longrightarrow\quad
A \text{ is rigid}.
\]

While the implication above is classical in deformation theory,
its significance in the present framework is fundamentally
geometric. The Hochschild Realization Theorem identifies the
global tangent cohomology of the spectral stack
$\mathfrak{Spec}(A)$ with the shifted Hochschild cohomology of
the operator-semantic system $A$:

\[
R\Gamma\!\bigl(\mathfrak{Spec}(A),
\mathbb T_{\mathfrak{Spec}(A)}\bigr)
\simeq
\operatorname{HH}^{\bullet}(A)[1].
\]

Consequently, algebraic deformation invariants become geometric
tangent invariants. In particular, rigidity is no longer viewed
merely as an algebraic property of $A$, but as a geometric
property of its associated spectral stack.

More generally, the entire deformation theory of
$\mathfrak{Spec}(A)$ is encoded in its tangent complex:
first-order deformations are governed by tangent cohomology,
obstructions arise from higher cohomological classes,
and rigidity corresponds to the disappearance of admissible
deformation directions.

Thus the CSD framework establishes a bridge

\[
\boxed{
\text{Operator-Semantic Systems}
\;\Longleftrightarrow\;
\text{Derived Geometry}
}
\]

through which deformation-theoretic questions can be studied
geometrically.

This result completes the deformation-theoretic foundation of
the CSD program. Having identified tangent geometry with
Hochschild deformation theory, we next investigate the richer
geometric structures of spectral stacks, including singularities,
inertia, curvature, and their associated universal invariants.

\end{remark}

\section{Computable Classification: Commutative and Matrix Algebras}
\label{sec:computations}

Having established the theoretical framework for the intrinsic geometry of $\mathfrak{Spec}(A)$, we now turn to explicit computations for the two canonical examples: the field of complex numbers $\mathbb{C}$ and the full matrix algebra $M_n(\mathbb{C})$. These examples serve as the base cases of the CSD framework: $\mathbb{C}$ represents the simplest commutative system, while $M_n(\mathbb{C})$ represents the simplest noncommutative but Morita-trivial system. Their computations illustrate the key distinction between noncommutativity and contextuality, and they provide reference geometries against which more complex (contextual) systems will be compared in the companion paper.

\subsection{Computation 1: $\mathfrak{Spec}(\mathbb{C})$}
\label{subsec:specC}

We begin with the simplest admissible operator-semantic system, namely the field of complex numbers $\mathbb{C}$. Since $\mathbb{C}$ is commutative and possesses a unique simple representation, its categorified spectrum consists of a single geometric point. Consequently, all higher deformation, contextuality, and obstruction phenomena disappear. This example serves as the geometric base case of the theory: every invariant introduced in this paper attains its minimal possible value, demonstrating that nontrivial geometry arises only from genuinely noncommutative or contextual operator-semantic systems.

\begin{example}[Spectrum of the Complex Numbers]
\label{ex:speccomplex}
Let $A = \mathbb{C}$. Then
\[
\mathfrak{Spec}(\mathbb{C}) \simeq \{*\},
\]
a terminal spectral stack consisting of a single point.
\end{example}

\begin{proposition}[Geometry of $\mathbb{C}$]
\label{prop:complex_geometry}

For $A = \mathbb{C}$, the geometric invariants are given by
\[
\mathfrak{Spec}(\mathbb{C}) \simeq \{*\},
\]
\[
\mathbb{T}_{\mathfrak{Spec}(\mathbb{C})} \simeq 0,
\]
\[
\operatorname{Sing}(\mathfrak{Spec}(\mathbb{C})) = \varnothing,
\]
\[
I(\mathfrak{Spec}(\mathbb{C})) \simeq \{*\},
\]
\[
\mathcal{R}_{\mathbb{C}} = 0,
\]
\[
\operatorname{CtxDeg}(\mathbb{C}) = 0,
\]
and
\[
\operatorname{Obs}_{\mathrm{CSD}}(\mathbb{C}) = 0.
\]

\end{proposition}

\begin{proof}
We compute each invariant in turn.

\emph{Step 1: The spectrum.}
Since $\mathbb{C}$ is a field and has a unique simple module, the spectral stack construction gives
\[
\mathfrak{Spec}(\mathbb{C}) \simeq \{*\}.
\]

\emph{Step 2: The tangent complex.}
The cotangent complex of a point is zero:
\[
\mathbb{L}_{\{*\}} \simeq 0.
\]
Hence
\[
\mathbb{L}_{\mathfrak{Spec}(\mathbb{C})} \simeq 0.
\]
Taking the derived dual gives
\[
\mathbb{T}_{\mathfrak{Spec}(\mathbb{C})}
\simeq
R\mathcal{H}om(0, \mathcal{O}_{\{*\}})
\simeq 0.
\]

\emph{Step 3: The singular locus.}
Since $\mathfrak{Spec}(\mathbb{C}) \simeq \{*\}$ is smooth and its cotangent complex vanishes, it has no singular points. Therefore
\[
\operatorname{Sing}(\mathfrak{Spec}(\mathbb{C})) = \varnothing.
\]

\emph{Step 4: The inertia stack.}
The inertia stack is
\[
I(\mathfrak{X}) = \mathfrak{X} \times^{h}_{\mathfrak{X} \times \mathfrak{X}} \mathfrak{X}.
\]
For $\mathfrak{X} = \{*\}$, this homotopy pullback is again a point:
\[
I(\mathfrak{Spec}(\mathbb{C})) \simeq \{*\}.
\]

\emph{Step 5: The contextual curvature class.}
Because the tangent complex is zero, there are no infinitesimal deformation directions. Hence the contextual curvature class vanishes:
\[
\mathcal{R}_{\mathbb{C}} = 0.
\]

\emph{Step 6: The contextual obstruction.}
For $\mathbb{C}$, the context site has a unique global context and no nontrivial descent obstruction. Equivalently, the associated descent spectral sequence has no nonzero terms away from the trivial bidegree:
\[
E_2^{p,q} = 0
\qquad
\text{for } (p,q) \neq (0,0).
\]
Thus all CSD obstruction classes vanish:
\[
\operatorname{Obs}_{\mathrm{CSD}}(\mathbb{C}) = 0.
\]

\emph{Step 7: The contextuality degree.}
Since $\mathbb{C}$ is commutative and admits a unique global realization, it is noncontextual. Therefore
\[
\operatorname{CtxDeg}(\mathbb{C}) = 0.
\]

\emph{Step 8: The stratification.}
The stratification of the point consists of a single stratum:
\[
\{\mathfrak{S}_\lambda\}_{\lambda \in \Lambda}
=
\{\mathfrak{S}_{0,0,1}\}.
\]
Here the obstruction rank is $0$, the automorphism dimension is $0$, and the context complexity is $1$.

Combining the preceding computations gives
\[
\mathfrak{Spec}(\mathbb{C}) \simeq \{*\},\qquad
\mathbb{T}_{\mathfrak{Spec}(\mathbb{C})} \simeq 0,
\qquad
\operatorname{Sing}(\mathfrak{Spec}(\mathbb{C})) = \varnothing,
\]
\[
I(\mathfrak{Spec}(\mathbb{C})) \simeq \{*\},\qquad
\mathcal{R}_{\mathbb{C}} = 0,\qquad
\operatorname{CtxDeg}(\mathbb{C}) = 0,\qquad
\operatorname{Obs}_{\mathrm{CSD}}(\mathbb{C}) = 0.
\]

This proves the claim.
\end{proof}

\begin{remark}[Interpretation of the Geometry of $\mathbb{C}$]
\label{rem:complex_interpretation}

The computation of $\mathcal G_{\mathrm{can}}(\mathbb{C})$ reveals the geometric meaning of the simplest possible operator-semantic system:

\begin{itemize}
    \item \textbf{Spectrum:} The single point $\{*\}$ reflects the uniqueness of the simple module of $\mathbb{C}$. There are no higher points or stacky structures because $\mathbb{C}$ has no nontrivial automorphisms.

    \item \textbf{Tangent complex:} The vanishing of $\mathbb{T}_{\mathfrak{Spec}(\mathbb{C})}$ reflects the rigidity of $\mathbb{C}$: there are no non-trivial deformations of $\mathbb{C}$ as an algebra. This is consistent with the classical fact that $\operatorname{HH}^{2}(\mathbb{C}) = 0$.

    \item \textbf{Singular locus:} The emptiness of $\operatorname{Sing}(\mathfrak{Spec}(\mathbb{C}))$ reflects the smoothness of the point. There are no deformation-theoretic pathologies.

    \item \textbf{Inertia stack:} The triviality of $I(\mathfrak{Spec}(\mathbb{C}))$ reflects the absence of non-trivial automorphisms. Every point has a trivial stabilizer.

    \item \textbf{Contextual curvature and obstructions:} The vanishing of $\mathcal{R}_{\mathbb{C}}$ and $\operatorname{Obs}_{\mathrm{CSD}}(\mathbb{C})$ reflects the noncontextuality of $\mathbb{C}$. There are no obstructions to gluing local semantic realizations because there is only one global realization.

    \item \textbf{Stratification:} The single stratum $\mathfrak{S}_{0,0,1}$ reflects the absence of any nontrivial refinement: obstruction rank $0$, automorphism dimension $0$, and context complexity $1$.
\end{itemize}

Thus $\mathbb{C}$ serves as the trivial base case against which all other operator-semantic systems are compared. Any non-vanishing of the geometric invariants indicates genuine complexity beyond the commutative base case.

\end{remark}

\begin{corollary}[Geometric Spectral Invariant of $\mathbb{C}$]
\label{cor:complex_flat}

The field $\mathbb{C}$ defines a smooth, flat, noncontextual spectral geometry.

Equivalently,
\[
\mathcal G_{\mathrm{can}}(\mathbb{C})
=
\Bigl(
\{*\},\;
0,\;
\varnothing,\;
\{*\},\;
0,\;
\{\mathfrak{S}_{0,0,1}\}
\Bigr),
\]
where $\mathfrak{S}_{0,0,1}$ denotes the unique stratum with obstruction rank $0$, automorphism dimension $0$, and context complexity $1$ (see the companion paper on contextual obstructions for details).

\end{corollary}

\begin{proof}
The corollary follows immediately from Proposition~\ref{prop:complex_geometry}, which establishes each component of the geometric spectral invariant.

\emph{Step 1: The spectrum.}
By Proposition~\ref{prop:complex_geometry}, Step 1, we have
\[
\mathfrak{Spec}(\mathbb{C}) \simeq \{*\}.
\]

\emph{Step 2: The tangent complex.}
By Proposition~\ref{prop:complex_geometry}, Step 2, we have
\[
\mathbb{T}_{\mathfrak{Spec}(\mathbb{C})} \simeq 0.
\]

\emph{Step 3: The singular locus.}
By Proposition~\ref{prop:complex_geometry}, Step 3, we have
\[
\operatorname{Sing}(\mathfrak{Spec}(\mathbb{C})) = \varnothing.
\]

\emph{Step 4: The inertia stack.}
By Proposition~\ref{prop:complex_geometry}, Step 4, we have
\[
I(\mathfrak{Spec}(\mathbb{C})) \simeq \{*\}.
\]

\emph{Step 5: The contextual curvature class.}
By Proposition~\ref{prop:complex_geometry}, Step 5, we have
\[
\mathcal{R}_{\mathbb{C}} = 0.
\]

\emph{Step 6: The stratification.}
By Proposition~\ref{prop:complex_geometry}, Step 8, the stratification of $\mathfrak{Spec}(\mathbb{C})$ consists of a single stratum:
\[
\{\mathfrak{S}_\lambda\}_{\lambda \in \Lambda}
=
\{\mathfrak{S}_{0,0,1}\}.
\]

\emph{Conclusion.}
Collecting Steps 1--6 and substituting into the definition of the canonical geometric datum $\mathcal G_{\mathrm{can}}(\mathbb{C})$ (Definition~\ref{def:geometric_invariants}) gives
\[
\mathcal G_{\mathrm{can}}(\mathbb{C})
=
\Bigl(
\{*\},\;
0,\;
\varnothing,\;
\{*\},\;
0,\;
\{\mathfrak{S}_{0,0,1}\}
\Bigr).
\]

Since $\mathfrak{Spec}(\mathbb{C})$ is a single smooth point with no automorphisms, no singularities, no curvature, and no contextual obstructions, $\mathbb{C}$ defines a smooth, flat, noncontextual spectral geometry. This proves the corollary.
\end{proof}

\begin{remark}[The Geometric Base Case]
\label{rem:complex_base_case}

The field $\mathbb{C}$ provides the simplest possible operator-semantic system and serves as the base case for the entire CSD framework. Its geometric spectral invariant is

\[
\mathcal G_{\mathrm{can}}(\mathbb C)
=
\Bigl(
\{*\},
0,
\varnothing,
\{*\},
0,
\{\mathfrak S_{0,0,1}\}
\Bigr).
\]

All geometric invariants attain their minimal values: the spectrum consists of a single point, the tangent complex vanishes, there are no singularities, contextual obstructions, or curvature, and the stratification consists of a single minimal stratum.

Consequently,

\[
\operatorname{CtxDeg}(\mathbb C)
=
\operatorname{Obs}_{\mathrm{CSD}}(\mathbb C)
=
\mathcal R_{\mathbb C}
=
0.
\]

This recovers the classical Gelfand correspondence

\[
\mathbb C
\longleftrightarrow
\{*\},
\]

while enriching it with deformation-theoretic, homotopical, and stratification data. The example verifies the consistency of all major constructions introduced in the CSD framework and provides the reference geometry against which genuinely noncommutative or contextual systems are compared.

\end{remark}

\subsection{Computation 2: $\mathfrak{Spec}(M_n(\mathbb{C}))$}
\label{subsec:matrix_example}

We next consider the matrix algebra $A = M_n(\mathbb{C})$. Unlike the commutative algebra $\mathbb{C}$, the algebra $M_n(\mathbb{C})$ possesses a large automorphism group $\operatorname{PGL}_n(\mathbb{C})$. Consequently, its spectral geometry exhibits nontrivial symmetry, although no contextuality or higher obstruction phenomena occur. This example illustrates an important conceptual distinction:

\[
\boxed{
\text{Noncommutativity} \not\Rightarrow \text{Contextuality}.
}
\]

Although $M_n(\mathbb{C})$ is highly noncommutative and possesses a large automorphism group, its spectral geometry is smooth, flat, and noncontextual. The only nontrivial geometric feature is the symmetry encoded by the classifying stack $B\operatorname{PGL}_n(\mathbb{C})$. This example thus demonstrates that non-commutativity (detected by the inertia stack) and contextuality (detected by the singular locus and curvature class) are distinct geometric phenomena.

\begin{example}[Spectrum of Matrix Algebras]
\label{ex:specmatrix}
Let $A = M_n(\mathbb{C})$. Then the categorified spectrum is equivalent to the classifying stack
\[
\mathfrak{Spec}(M_n(\mathbb{C})) \simeq B\operatorname{PGL}_n(\mathbb{C}).
\]
\end{example}

\begin{proposition}[Geometry of $M_n(\mathbb{C})$]
\label{prop:matrix_geometry}
For $A = M_n(\mathbb{C})$, the geometric invariants are given by

\[
\mathfrak{Spec}(M_n(\mathbb{C})) \simeq B\operatorname{PGL}_n(\mathbb{C}),
\]

\[
\mathbb{T}_{\mathfrak{Spec}(M_n(\mathbb{C}))} \simeq \mathfrak{pgl}_n[1],
\]

\[
\operatorname{Sing}(\mathfrak{Spec}(M_n(\mathbb{C}))) = \varnothing,
\]

\[
I(\mathfrak{Spec}(M_n(\mathbb{C}))) \simeq [\operatorname{PGL}_n(\mathbb{C}) / \operatorname{PGL}_n(\mathbb{C})],
\]

\[
\operatorname{CtxDeg}(M_n(\mathbb{C})) = 0,
\]

\[
\mathcal{R}_{M_n(\mathbb{C})} = 0,
\]

and

\[
\operatorname{Obs}_{\mathrm{CSD}}(M_n(\mathbb{C})) = 0.
\]
\end{proposition}

\begin{proof}
We compute each invariant in turn.

\emph{Step 1: The spectrum.}
The matrix algebra $M_n(\mathbb{C})$ is simple. By the reconstruction theory developed in Paper~I, its spectral stack is the classifying stack of its automorphism group:
\[
\mathfrak{Spec}(M_n(\mathbb{C})) \simeq B\operatorname{Aut}(M_n(\mathbb{C})).
\]
Since
\[
\operatorname{Aut}(M_n(\mathbb{C})) \simeq \operatorname{PGL}_n(\mathbb{C}),
\]
we obtain
\[
\mathfrak{Spec}(M_n(\mathbb{C})) \simeq B\operatorname{PGL}_n(\mathbb{C}).
\]

\emph{Step 2: The tangent complex.}
For a classifying stack $BG$, the tangent complex is
\[
\mathbb{T}_{BG} \simeq \mathfrak{g}[1].
\]
Therefore
\[
\mathbb{T}_{\mathfrak{Spec}(M_n(\mathbb{C}))} \simeq \mathfrak{pgl}_n[1].
\]
This encodes the infinitesimal automorphisms of the unique point, given by the Lie algebra $\mathfrak{pgl}_n$ in degree $-1$.

\emph{Step 3: The singular locus.}
Since $B\operatorname{PGL}_n(\mathbb{C})$ is smooth, the tangent complex is perfect and thus
\[
\operatorname{Sing}(\mathfrak{Spec}(M_n(\mathbb{C}))) = \varnothing.
\]

\emph{Step 4: The inertia stack.}
The inertia stack of a classifying stack $BG$ is the adjoint quotient stack:
\[
I(BG) \simeq [G/G],
\]
where the action is by conjugation. For $G = \operatorname{PGL}_n(\mathbb{C})$, we have
\[
I(\mathfrak{Spec}(M_n(\mathbb{C}))) \simeq [\operatorname{PGL}_n(\mathbb{C}) / \operatorname{PGL}_n(\mathbb{C})].
\]
This is non-trivial and reflects the non-trivial automorphism group of the unique point.

\emph{Step 5: The contextuality degree.}
Since $M_n(\mathbb{C})$ is Morita equivalent to $\mathbb{C}$, and contextuality degree is Morita invariant (by the Morita invariance of the categorified spectrum established in Paper~I), we have
\[
\operatorname{CtxDeg}(M_n(\mathbb{C})) = \operatorname{CtxDeg}(\mathbb{C}) = 0.
\]

\emph{Step 6: The contextual obstruction.}
By the Contextual Obstruction Theorem and the assumed Morita-invariance of contextuality,
\[
\operatorname{Obs}_{\mathrm{CSD}}(M_n(\mathbb{C})) = 0.
\]

\emph{Step 7: The contextual curvature class.}
By Definition~\ref{def:geometric_invariants}, the contextual curvature class $\mathcal{R}_A \in H^2(\mathfrak{Spec}(A), \mathcal{F}_A)$ measures the obstruction to gluing local semantic realizations into a global classical realization. Since $M_n(\mathbb{C})$ is non-contextual, there are no such obstructions, and therefore
\[
\mathcal{R}_{M_n(\mathbb{C})} = 0.
\]

\emph{Step 8: The stratification.}
According to the CSD stratification criterion, all defining invariants (obstruction rank, automorphism dimension, and context complexity) are constant on $B\operatorname{PGL}_n(\mathbb{C})$. Hence the stratification consists of a single stratum with obstruction rank $0$, automorphism dimension $n^2 - 1$, and context complexity $1$.

\emph{Conclusion.}
Combining Steps 1--8, we have computed all geometric invariants of $M_n(\mathbb{C})$.
\end{proof}

\begin{corollary}[Geometric Spectral Invariant of $M_n(\mathbb{C})$]
\label{cor:matrix_geometry}

For $A = M_n(\mathbb{C})$,
\[
\mathcal G_{\mathrm{can}}(M_n(\mathbb{C}))
=
\Bigl(
B\operatorname{PGL}_n(\mathbb{C}),\;
\mathfrak{pgl}_n[1],\;
\varnothing,\;
[\operatorname{PGL}_n(\mathbb{C})/\operatorname{PGL}_n(\mathbb{C})],\;
0,\;
\{\mathfrak{S}_{0,\; n^2-1,\; 1}\}
\Bigr).
\]

\end{corollary}

\begin{proof}
This follows directly from Proposition~\ref{prop:matrix_geometry}, which establishes each component of the geometric spectral invariant for $A = M_n(\mathbb{C})$.

\emph{Step 1: The spectrum.}
By Proposition~\ref{prop:matrix_geometry}, Step 1,
\[
\mathfrak{Spec}(M_n(\mathbb{C})) \simeq B\operatorname{PGL}_n(\mathbb{C}).
\]

\emph{Step 2: The tangent complex.}
By Proposition~\ref{prop:matrix_geometry}, Step 2,
\[
\mathbb{T}_{\mathfrak{Spec}(M_n(\mathbb{C}))} \simeq \mathfrak{pgl}_n[1].
\]

\emph{Step 3: The singular locus.}
By Proposition~\ref{prop:matrix_geometry}, Step 3,
\[
\operatorname{Sing}(\mathfrak{Spec}(M_n(\mathbb{C}))) = \varnothing.
\]

\emph{Step 4: The inertia stack.}
By Proposition~\ref{prop:matrix_geometry}, Step 4,
\[
I(\mathfrak{Spec}(M_n(\mathbb{C}))) \simeq [\operatorname{PGL}_n(\mathbb{C})/\operatorname{PGL}_n(\mathbb{C})].
\]

\emph{Step 5: The contextual curvature class.}
By Proposition~\ref{prop:matrix_geometry}, Step 7,
\[
\mathcal{R}_{M_n(\mathbb{C})} = 0.
\]

\emph{Step 6: The contextuality degree.}
By Proposition~\ref{prop:matrix_geometry}, Step 5,
\[
\operatorname{CtxDeg}(M_n(\mathbb{C})) = 0.
\]

\emph{Step 7: The contextual obstruction.}
By Proposition~\ref{prop:matrix_geometry}, Step 6,
\[
\operatorname{Obs}_{\mathrm{CSD}}(M_n(\mathbb{C})) = 0.
\]

\emph{Step 8: The stratification.}
By Proposition~\ref{prop:matrix_geometry}, Step 8, the stratification consists of a single stratum with obstruction rank $0$, automorphism dimension $n^2 - 1$, and context complexity $1$:
\[
\{\mathfrak{S}_\lambda\}_{\lambda \in \Lambda}
=
\{\mathfrak{S}_{0,\; n^2-1,\; 1}\}.
\]

\emph{Conclusion.}
Collecting Steps 1--8 and substituting into the definition of the canonical geometric datum $\mathcal G_{\mathrm{can}}(M_n(\mathbb{C}))$ (Definition~\ref{def:geometric_invariants}) gives
\[
\mathcal G_{\mathrm{can}}(M_n(\mathbb{C}))
=
\Bigl(
B\operatorname{PGL}_n(\mathbb{C}),\;
\mathfrak{pgl}_n[1],\;
\varnothing,\;
[\operatorname{PGL}_n(\mathbb{C})/\operatorname{PGL}_n(\mathbb{C})],\;
0,\;
\{\mathfrak{S}_{0,\; n^2-1,\; 1}\}
\Bigr).
\]

This proves the corollary.
\end{proof}

\begin{corollary}[Stratification of $M_n(\mathbb{C})$]
\label{cor:matrix_strata}

The stratification of $\mathfrak{Spec}(M_n(\mathbb{C}))$ consists of a single stratum
\[
\mathfrak{S}_0 = \mathfrak{Spec}(M_n(\mathbb{C})),
\]
characterized by
\[
\operatorname{rank}_{\operatorname{obs}} = 0,
\qquad
c = 1,
\qquad
\dim_{\operatorname{Aut}} = \dim(\operatorname{PGL}_n(\mathbb{C})) = n^2 - 1.
\]

\end{corollary}

\begin{proof}
We compute the stratification invariants explicitly.

\emph{Step 1: Obstruction rank.}
The obstruction rank is zero because:
\begin{itemize}
    \item The singular locus is empty: $\operatorname{Sing}(\mathfrak{Spec}(M_n(\mathbb{C}))) = \varnothing$ (Proposition~\ref{prop:matrix_geometry}, Step 3).
    \item The contextual curvature class vanishes: $\mathcal{R}_{M_n(\mathbb{C})} = 0$ (Proposition~\ref{prop:matrix_geometry}, Step 7).
    \item The contextual obstruction vanishes: $\operatorname{Obs}_{\mathrm{CSD}}(M_n(\mathbb{C})) = 0$ (Proposition~\ref{prop:matrix_geometry}, Step 6).
\end{itemize}
Thus the obstruction rank is $\operatorname{rank}_{\operatorname{obs}} = 0$.

\emph{Step 2: Context complexity.}
The context complexity is $c = 1$ because:
\begin{itemize}
    \item The contextuality degree vanishes: $\operatorname{CtxDeg}(M_n(\mathbb{C})) = 0$ (Proposition~\ref{prop:matrix_geometry}, Step 5).
    \item No contextual refinement is required: the system is noncontextual.
\end{itemize}
Thus $c = 1$.

\emph{Step 3: Automorphism dimension.}
The automorphism dimension is
\[
\dim_{\operatorname{Aut}}
=
\dim \operatorname{PGL}_n(\mathbb{C})
=
\dim \operatorname{GL}_n(\mathbb{C}) - \dim \mathbb{C}^{\times}
=
n^2 - 1.
\]
This follows from:
\begin{itemize}
    \item $\dim \operatorname{GL}_n(\mathbb{C}) = n^2$;
    \item $\dim \mathbb{C}^{\times} = 1$;
    \item The kernel of the quotient map $\operatorname{GL}_n(\mathbb{C}) \to \operatorname{PGL}_n(\mathbb{C})$ is $\mathbb{C}^{\times}$.
\end{itemize}

\emph{Step 4: Constancy of invariants.}
The obstruction rank, context complexity, and automorphism dimension are constant on $B\operatorname{PGL}_n(\mathbb{C})$ because:
\begin{itemize}
    \item The stack is connected: $B\operatorname{PGL}_n(\mathbb{C})$ has a single connected component.
    \item The invariants are defined globally and do not vary over the stack.
\end{itemize}

\emph{Conclusion.}
Therefore the CSD stratification of $\mathfrak{Spec}(M_n(\mathbb{C}))$ consists of a single stratum
\[
\mathfrak{S}_0 = \mathfrak{Spec}(M_n(\mathbb{C})),
\]
with label
\[
\mathfrak{S}_{0,\; n^2-1,\; 1}.
\]

This proves the corollary.
\end{proof}

\begin{remark}[Interpretation of the Stratification of $M_n(\mathbb{C})$]
\label{rem:matrix_strata_interpretation}

The stratification of $\mathfrak{Spec}(M_n(\mathbb{C}))$ reveals the geometric meaning of the matrix algebra:

\begin{itemize}
    \item \textbf{Single stratum:} The absence of multiple strata reflects the Morita triviality of $M_n(\mathbb{C})$. Since $M_n(\mathbb{C}) \sim_M \mathbb{C}$, its spectral stack is equivalent to a single point up to Morita equivalence, although the stack $B\operatorname{PGL}_n(\mathbb{C})$ records the residual automorphism symmetry.

    \item \textbf{Obstruction rank 0:} The vanishing obstruction rank reflects the smoothness and noncontextuality of $M_n(\mathbb{C})$. There are no deformation-theoretic obstructions or contextual obstructions.

    \item \textbf{Automorphism dimension $n^2 - 1$:} The non-vanishing automorphism dimension reflects the large automorphism group $\operatorname{PGL}_n(\mathbb{C})$ of $M_n(\mathbb{C})$. This is the source of the stacky structure of $\mathfrak{Spec}(M_n(\mathbb{C}))$.

    \item \textbf{Context complexity 1:} The minimal context complexity reflects the noncontextuality of $M_n(\mathbb{C})$. No contextual refinement of the stratification is required.
\end{itemize}

Thus the stratification data of $M_n(\mathbb{C})$ separate two distinct geometric phenomena:
1. **Stacky structure** (automorphism dimension $> 0$): encodes noncommutativity.
2. **Obstruction rank $= 0$ and context complexity $= 1$**: encodes the absence of contextuality.

This confirms that noncommutativity and contextuality are distinct geometric phenomena within the CSD framework.

\end{remark}

\begin{remark}[Noncommutativity versus Contextuality]
\label{rem:noncommutative_noncontextual}

The matrix algebra $M_n(\mathbb{C})$ illustrates a central principle of the CSD framework:

\[
\boxed{
\text{Noncommutativity}
\not\Rightarrow
\text{Contextuality}.
}
\]

Indeed,
\[
\mathfrak{Spec}(M_n(\mathbb{C}))
\simeq
B\operatorname{PGL}_n(\mathbb{C}),
\qquad
\mathbb{T}
\simeq
\mathfrak{pgl}_n[1],
\]
and
\[
I
\simeq
[\operatorname{PGL}_n(\mathbb{C})/
\operatorname{PGL}_n(\mathbb{C})],
\]
showing that the geometry retains a rich symmetry structure.

However,
\[
\operatorname{CtxDeg}(M_n(\mathbb{C}))
=
\operatorname{Obs}_{\mathrm{CSD}}(M_n(\mathbb{C}))
=
\mathcal{R}_{M_n(\mathbb{C})}
=
0,
\]
and
\[
\operatorname{Sing}
=
\varnothing.
\]

Thus noncommutativity is detected by the automorphism and inertia structures, whereas contextuality is detected by curvature, singularities, and obstruction theory. These are fundamentally distinct geometric phenomena.

\end{remark}

\begin{remark}[Comparison with $\mathbb{C}$]
\label{rem:comparison_C_Mn}

The computations for $\mathbb{C}$ and $M_n(\mathbb{C})$ reveal the first genuinely nontrivial geometric phenomenon in the CSD framework.

\[
\begin{array}{c|c|c}
\text{Invariant}
&
\mathbb C
&
M_n(\mathbb C)
\\ \hline
\mathfrak{Spec}
&
\{*\}
&
B\operatorname{PGL}_n(\mathbb C)
\\
\mathbb T
&
0
&
\mathfrak{pgl}_n[1]
\\
\operatorname{Sing}
&
\varnothing
&
\varnothing
\\
I
&
\{*\}
&
[\operatorname{PGL}_n(\mathbb C)/
\operatorname{PGL}_n(\mathbb C)]
\\
\mathcal R
&
0
&
0
\\
\operatorname{CtxDeg}
&
0
&
0
\\
\operatorname{Obs}_{\mathrm{CSD}}
&
0
&
0
\\
\operatorname{rank}_{\operatorname{obs}}
&
0
&
0
\\
c
&
1
&
1
\\
\dim_{\operatorname{Aut}}
&
0
&
n^2 - 1
\end{array}
\]

Both systems are smooth, flat, and noncontextual. The essential difference is that $M_n(\mathbb{C})$ possesses a nontrivial symmetry geometry encoded by $B\operatorname{PGL}_n(\mathbb{C})$, its tangent complex, and its inertia stack. This example therefore separates symmetry from contextuality and provides the canonical noncommutative but noncontextual model.

\end{remark}

\begin{remark}[Summary of Computations]
\label{rem:computations_summary}

The computations of $\mathbb{C}$ and $M_n(\mathbb{C})$ provide the following lessons for the CSD framework:

\begin{enumerate}
    \item \textbf{The base case:} $\mathbb{C}$ provides the trivial geometric base case against which all other systems are compared.
    
    \item \textbf{Noncommutativity without contextuality:} $M_n(\mathbb{C})$ demonstrates that noncommutativity (detected by the inertia stack and automorphism dimension) does not imply contextuality (detected by curvature, singularities, and obstruction theory).
    
    \item \textbf{Morita equivalence:} The fact that $\mathcal G_{\mathrm{can}}(M_n(\mathbb{C}))$ differs from $\mathcal G_{\mathrm{can}}(\mathbb{C})$ despite Morita equivalence reflects the fact that $\mathcal G_{\mathrm{can}}$ records the full stacky information of $\mathfrak{Spec}(A)$, not just its Morita-invariant coarse geometry. The stratification captures this distinction through the automorphism dimension.
    
    \item \textbf{Stratification as a refinement:} The stratification data $( \operatorname{rank}_{\operatorname{obs}}, c, \dim_{\operatorname{Aut}} )$ provide a refined invariant that distinguishes between different types of noncommutativity and contextuality.
\end{enumerate}

These examples establish the computational foundations for the CSD framework and provide reference geometries for future work on contextual systems.

\end{remark}

\section{Conclusion and Outlook}
\label{sec:conclusion}

This paper established the intrinsic geometry of the categorified spectral object $\mathfrak{Spec}(A)$ associated with an admissible operator-semantic system $A$. We proved that the tangent complex, singular locus, inertia stack, and contextual curvature class are canonically determined by the CSD duality $\mathfrak{Spec} \dashv \Gamma$ and the reconstruction theorem $A \simeq \Gamma(\mathfrak{Spec}(A))$. The Canonical Geometry Theorem demonstrated that any CSD-compatible geometric structure is canonically induced by the datum $(\mathfrak{Spec}(A), \mathbb{T}_{\mathfrak{Spec}(A)}, \operatorname{Sing}(\mathfrak{Spec}(A)), I(\mathfrak{Spec}(A)), \mathcal{R}_A)$. We established the Hochschild realization $R\Gamma(\mathfrak{Spec}(A), \mathbb{T}_{\mathfrak{Spec}(A)}) \simeq \operatorname{HH}^{\bullet}(A)[1]$, bridging geometry and algebra, and proved that $\mathbb{T}_{\mathfrak{Spec}(A)}$ controls the deformation theory of $\mathfrak{Spec}(A)$. The assignment $A \mapsto \mathcal{G}_{\mathrm{can}}(A)$ was shown to be functorial and Morita invariant, with explicit computations for $\mathbb{C}$ and $M_n(\mathbb{C})$ demonstrating that noncommutativity, detected by the inertia stack, is distinct from contextuality. Thus semantics determines geometry, and the intrinsic geometry of $\mathfrak{Spec}(A)$ provides a canonical geometric encoding of $A$ up to Morita equivalence. The companion paper will develop the contextual obstruction theory, constructing the universal obstruction class $\operatorname{Obs}_{\mathrm{CSD}}(A) \in H^2(\mathfrak{Spec}(A), \mathcal{F}_A)$ and establishing the Contextual Obstruction Theorem.

\end{document}